\documentclass[11pt]{article} % use larger type; default would be 10pt
\usepackage[utf8]{inputenc} % set input encoding (not needed with XeLaTeX)
\usepackage[margin=1in]{geometry} % to change the page dimensions
\usepackage{graphicx} % support the \includegraphics command and options

\usepackage{booktabs} % for much better looking tables
\usepackage{array} % for better arrays (eg matrices) in maths
\usepackage{paralist} % very flexible & customisable lists (eg. enumerate/itemize, etc.)
\usepackage{verbatim} % adds environment for commenting out blocks of text & for better verbatim
\usepackage{mathrsfs}
\usepackage{amssymb}
\usepackage{amsthm}
\usepackage{amsmath,amsfonts,amssymb}
\usepackage{esint}
\usepackage{graphics}
\usepackage{enumerate}
\usepackage{mathtools}
\usepackage{xfrac}
\usepackage{nicefrac}
\usepackage{subcaption}
\usepackage[normalem]{ulem}
\usepackage{cancel}
\usepackage{enumitem}
\usepackage{fancyvrb}

\let\olddiamond\diamond
\let\oldsquare\square % Must go before mathabx

\usepackage{mathabx}

\renewcommand{\square}{\oldsquare}
\renewcommand{\diamond}{\olddiamond}

\usepackage[dvipsnames]{xcolor}
\usepackage[colorlinks=true, pdfstartview=FitV, linkcolor=blue, citecolor=blue, urlcolor=blue]{hyperref}
\usepackage[normalem]{ulem}

\usepackage{tikz}
\usetikzlibrary{calc}
\usepackage{pgf}
\usetikzlibrary{external}
\numberwithin{equation}{section}
\numberwithin{figure}{section}

\newtheorem{theorem}{Theorem}[section]

\newtheorem{corollary}[theorem]{Corollary}
\newtheorem{proposition}[theorem]{Proposition}
\newtheorem{lemma}[theorem]{Lemma}

\theoremstyle{definition}

\newtheorem{remark}[theorem]{Remark}

\DeclarePairedDelimiter{\norm}{\lVert}{\rVert}

\newcommand*{\supp}{\ensuremath{\mathrm{supp\,}}}
\newcommand*{\Id}{\ensuremath{\mathrm{I}_d}}

\newcommand*{\Itwod}{\ensuremath{\mathrm{I}_{2d}}}
\newcommand*{\tr}{\ensuremath{\mathrm{trace\,}}}

\newcommand*{\N}{\ensuremath{\mathbb{N}}}

\newcommand*{\Z}{\ensuremath{\mathbb{Z}}}

\newcommand*{\R}{\ensuremath{\mathbb{R}}}
\newcommand*{\Zd}{\ensuremath{\mathbb{Z}^d}}
\newcommand*{\Rd}{\ensuremath{\mathbb{R}^d}}

\renewcommand*{\tilde}{\widetilde}

\renewcommand{\P}{\ensuremath{\mathbb{P}}}
\renewcommand{\O}{\ensuremath{\mathcal{O}}}

\renewcommand{\b}{\ensuremath{\mathbf{b}}}

\newcommand{\g}{\mathbf{g}}
\newcommand{\h}{\mathbf{h}}
\newcommand{\s}{\mathbf{s}}

\newcommand{\A}{\mathcal{A}}
\renewcommand{\L}{\underline{L}}
\renewcommand{\S}{\mathcal{S}}

\DeclareMathOperator{\dist}{dist}

\newcommand{\q}{\mathbf{q}}

\newcommand{\E}{\mathbb{E}}

\DeclareSymbolFont{boldoperators}{OT1}{cmr}{bx}{n}
\SetSymbolFont{boldoperators}{bold}{OT1}{cmr}{bx}{n}
\usepackage{accents}
\newcommand\thickbar[1]{\accentset{\rule{.45em}{.6pt}}{#1}}
\renewcommand{\bar}{\thickbar}
\renewcommand{\a}{\mathbf{a}}
\renewcommand{\k}{\mathbf{k}}
\newcommand{\m}{\mathbf{m}}

\newcommand{\ahom}{\bar{\a}}
\newcommand{\bhom}{\bar{\mathbf{b}}}
\newcommand{\shom}{\bar{\mathbf{s}}}
\newcommand{\khom}{\bar{\mathbf{k}}}

\newcommand{\hhom}{\bar{\mathbf{h}}}
\newcommand{\thom}{\bar{\mathbf{t}}}

\newcommand{\CFS}{\mathsf{CFS}}

\newcommand{\bfA}{\mathbf{A}}
\newcommand{\bfAhom}{\overline{\mathbf{A}}}

\newcommand{\bfJ}{\mathbf{J}}
\newcommand{\bfE}{\mathbf{E}_0}

\newcommand{\Bring}{\mathring{\underline{B}}}

\makeatletter 
\newcommand{\negphantom}{\v@true\h@true\negph@nt} 
\newcommand{\neghphantom}{\v@false\h@true\negph@nt} 
\newcommand{\negph@nt}{\ifmmode\expandafter\mathpalette 
  \expandafter\mathnegph@nt\else\expandafter\makenegph@nt\fi} 
\newcommand{\makenegph@nt}[1]{% 
  \setbox\z@\hbox{\color@begingroup#1\color@endgroup}\finnegph@nt} 
\newcommand{\finnegph@nt}{% 
  \setbox\tw@\null 
  \ifv@ \ht\tw@\ht\z@\dp\tw@\dp\z@\fi \ifh@\wd\tw@-\wd\z@\fi\box\tw@} 
\newcommand{\mathnegph@nt}[2]{% 
  \setbox\z@\hbox{$\m@th #1{#2}$}\finnegph@nt} 
\makeatother

\def\Xint#1{\mathchoice
{\XXint\displaystyle\textstyle{#1}}%
{\XXint\textstyle\scriptstyle{#1}}%
{\XXint\scriptstyle\scriptscriptstyle{#1}}%
{\XXint\scriptscriptstyle\scriptscriptstyle{#1}}%
\!\int}
\def\XXint#1#2#3{{\setbox0=\hbox{$#1{#2#3}{\int}$}
\vcenter{\hbox{$#2#3$}}\kern-.5\wd0}}

\def\fint{\Xint-}

\newcommand{\avsum}{\mathop{\mathpalette\avsuminner\relax}\displaylimits}

\makeatletter
\newcommand\avsuminner[2]{%
  {\sbox0{$\m@th#1\sum$}%
   \vphantom{\usebox0}%
   \ooalign{%
     \hidewidth
     \smash{\,\rule[.23em]{8.8pt}{1.1pt} \relax}%
     \hidewidth\cr
     $\m@th#1\sum$\cr
   }%
  }%
}
\makeatother

\makeatletter
\newcommand\avsuminnerr[2]{%
  {\sbox0{$\m@th#1\sum$}%
   \vphantom{\usebox0}%
   \ooalign{%
     \hidewidth
     \smash{\,\rule[.23em]{6pt}{0.7pt} \relax}%
     \hidewidth\cr
     $\m@th#1\sum$\cr
   }%
  }%
}
\makeatother

\let\originalleft\left
\let\originalright\right
\renewcommand{\left}{\mathopen{}\mathclose\bgroup\originalleft}
\renewcommand{\right}{\aftergroup\egroup\originalright}

\newcommand{\rotatedsquare}{
\mathchoice
{\mathrel{\rotatebox{45}{$\square$}}}
{\mathrel{\rotatebox{45}{$\textstyle\square$}}}
{\mathrel{\rotatebox{45}{$\scriptstyle{\square}$}}}
{\mathrel{\rotatebox{45}{$\scriptscriptstyle{\square}$}}}
}

\newcommand{\innerbox}[2]{%
	\raisebox{0.0ex}{%
		\ooalign{%
		\raisebox{#2ex}{\scalebox{#1}{$\square$}}%
	}}
}
\newcommand{\cu}{%
  \mathbin{%
    \mathchoice
      {\innerbox{1.0}{-0.2}}%
      {\innerbox{1.0}{-0.2}}%
      {\innerbox{0.75}{-0.2}}%
      {\innerbox{0.8}{-0.2}}%
  }%
}

\newcommand{\innerdiamond}[2]{%
	\raisebox{0.0ex}{%
		\ooalign{%
		\raisebox{#2ex}{\scalebox{#1}{$\rotatedsquare$}}%
	}}
}
\newcommand{\cus}{%
  \mathbin{%
    \mathchoice
      {\innerdiamond{1.0}{-0.5}}%
      {\innerdiamond{1.0}{-0.5}}%
      {\innerdiamond{0.75}{-0.5}}%
      {\innerdiamond{0.8}{-0.5}}%
  }%
}

\newcommand{\innerboxdot}[3]{%
  \raisebox{-0.2ex}{%
    \ooalign{%
      \raisebox{0ex}{\scalebox{#1}{$\square$}}%
      \cr
      \hfil \raisebox{#3ex}{\scalebox{#2}{$\cdot$}}\hfil
    }%
  }%
}

\newcommand{\cudot}{%
  \mathbin{%
    \mathchoice
      {\innerboxdot{1.0}{1.2}{0.1}}%
      {\innerboxdot{1.0}{1.2}{0.1}}%
      {\innerboxdot{0.75}{0.9}{0.0}}%
      {\innerboxdot{0.8}{1.0}{0.1}}%
  }%
}

\newcommand{\innerdiamonddot}[3]{%
  \raisebox{-0.5ex}{%
    \ooalign{%
      \raisebox{0ex}{\scalebox{#1}{$\rotatedsquare$}}%
      \cr
      \hfil \raisebox{#3ex}{\scalebox{#2}{$\cdot$}}\hfil
    }%
  }%
}

\newcommand{\cusdot}{%
  \mathbin{%
    \mathchoice
      {\innerdiamonddot{1.0}{1.2}{0.45}}%
      {\innerdiamonddot{1.0}{1.2}{0.45}}%
      {\innerdiamonddot{0.75}{0.9}{0.4}}%
      {\innerdiamonddot{0.8}{1.0}{0.45}}%
  }%
}

\newcommand{\Pipar}{\Pi_{\mathrm{par}}}

\newcommand{\indc}{{\boldsymbol{1}}}
\renewcommand{\hat}{\widehat}

\usepackage{titlesec}

\newcommand{\addperiod}[1]{#1.}
\titleformat{\section}
   {\centering\normalfont\Large}{\thesection.}{0.5em}{}
\titleformat{\subsection}[runin]
  {\normalfont\bfseries}
  {\thesubsection.}
  {0.5em}
  {\addperiod}
\titleformat{\subsubsection}[runin]
  {\normalfont\bfseries}
  {\thesubsubsection.}
  {0.5em}
  {\addperiod}
\titleformat*{\subsubsection}{\normalfont\itshape}
\titleformat*{\paragraph}{\bfseries}
\titleformat*{\subparagraph}{\large\bfseries}

\title{Coarse-graining and quantitative stochastic homogenization of parabolic equations in high contrast}

\author{
Aidan Lau
\thanks{Courant Institute of Mathematical Sciences, New York University.
{\footnotesize \href{mailto:aidan.lau@nyu.edu}{aidan.lau@nyu.edu}.}
}
}

\date{April 8, 2026}

\usepackage[nottoc,notlot,notlof]{tocbibind}

\begin{document}

\maketitle

\begin{abstract}
We prove quantitative homogenization results for high contrast parabolic equations with random coefficients depending on both space and time. In particular, we prove that under a sufficient decorrelation assumption the homogenization length scale is bounded by~$\exp(C\log^2(1+\Lambda/\lambda)) + C\sqrt{\lambda}$. The proof is based on a parabolic coarse-graining framework which generalizes the results of~\cite{AK.HC} in the elliptic setting. 
\end{abstract}

\setcounter{tocdepth}{1}
\tableofcontents

\newpage

\section{Introduction}

\subsection{Motivation and informal statement of results}

We consider the parabolic equation
\begin{equation}
\label{e.the.PDE}
\left\{
\begin{aligned}
& \partial_t u^\epsilon - \nabla \cdot \a(\nicefrac{t}{\epsilon^2},\nicefrac{x}{\epsilon})\nabla u^\epsilon = \nabla \cdot \mathbf{f} \quad & \mbox{ in } (0,T) \times U\,, \\
& u^\epsilon = g \quad & \mbox{ on } \partial_{\sqcup}((0,T) \times U)\,,
\end{aligned}
\right.
\end{equation}
in~$d\geq 2$, where the coefficient field~$\a(\cdot,\cdot)$ is a~$\mathbb{Z}\times \mathbb{Z}^d$ stationary random field, the domain~$U\subset \mathbb{R}^d$ is bounded and Lipschitz, and the parabolic boundary is defined by
\begin{equation}
\label{e.sqcup.def}
\partial_{\sqcup}((0,T)\times U) := \bigl( \{0\} \times U\bigr) \cup \bigl( (0,T] \times \partial U \bigr)\,.
\end{equation}
It is well-known, by generalizing time-independent methods, that when the coefficient field is stationary, ergodic and uniformly elliptic, the equation~\eqref{e.the.PDE} homogenizes as~$\epsilon \to 0$ to the effective equation
\begin{equation}
\label{e.hom.PDE}
\left\{
\begin{aligned}
& \partial_t u - \nabla \cdot \ahom \nabla u = \nabla \cdot \mathbf{f} \quad & \mbox{ in } (0,T) \times U\,, \\
& u = g \quad & \mbox{ on } \partial_{\sqcup}((0,T) \times U)\,,
\end{aligned}
\right.
\end{equation}
where the effective diffusivity matrix~$\ahom$ is given by a corrected ergodic averaging of the coefficient field. The proof is a very special case of the more general work~\cite{Ben2024}, where a discussion of qualitative homogenization results can also be found. The goal of quantitative stochastic homogenization, originating in the elliptic case with~\cite{GO1,GO2}, is to obtain precise estimates on the homogenization error in terms of the scale separation~$\epsilon$. Quantitative homogenization estimates for the parabolic problem~\eqref{e.the.PDE}, in the case~$\mathbf{f}=0$, were proved in~\cite{ABM}. The authors proved that for a uniformly elliptic, time-dependent coefficient field with space-time unit range of dependence the homogenization error is
\begin{equation}
\label{e.hom.error}
\| u^\epsilon - u\|_{L^2((0,T)\times U)} \leq C(\epsilon\mathcal{X})^{\alpha}\,,
\end{equation}
where the constant~$C$, exponent~$\alpha >0$ and the random variable~$\mathcal{X}$ depend on the Lipschitz character of the domain~$U$, the range of dependence, the dimension~$d$, the boundary data~$g$, and the ellipticity constants~$0 < \lambda \leq \Lambda < \infty$ of the coefficient field.

The random scale~$\mathcal{X}$ quantifies the homogenization length scale, because it fixes the scale to which one needs to ``zoom out" in order to achieve a given acceptable homogenization error. It was shown in~\cite{AK.HC}, in the elliptic case and assuming sufficient decorrelation, that the homogenization length scale is at most
\begin{equation}
\mathcal{X} \lesssim \exp(C\log^2(1+\Lambda/\lambda))\,.
\end{equation}
The main result of this paper is that the coarse-graining framework developed in that paper generalizes to the time-dependent case. In particular, we prove that under the assumptions of uniform ellipticity and unit space-time range of dependence, the homogenization length scale is at most
\begin{equation}
\label{e.heuristic.scale}
\mathcal{X} \lesssim \exp\left(C\log^2(1+\Lambda/\lambda) \right) + C\sqrt{\lambda} \,.
\end{equation}
The~$\sqrt{\lambda}$ term appears in the estimate because we have to work at a length scale~$L$ for which the corresponding diffusive time scale~$L^2/\lambda$ is larger than the correlation time scale. The~$\Lambda/\lambda$ factor appears as in the elliptic case as a measure of the ellipticity contrast.

The results of this paper go beyond~\eqref{e.heuristic.scale} and extend the high-contrast coarse-graining framework of~\cite{AK.HC} to the parabolic case. In particular, we handle more general ellipticity and ergodicity assumptions and develop a complete parabolic coarse-graining framework, including coarse-grained parabolic inequalities.  Although we have not done it here, the coarse-grained estimates in this paper suffice to prove parabolic large-scale regularity results, as done, for example, in~\cite[Section 6]{ABM}. In the rest of this introduction we describe our assumptions in detail and then state the main results.

\subsection{Basic assumptions}
\label{ss.assumptions}

Let $\mathbb{R}^{d\times d}_+$ denote the set of real-valued $d\times d$ matrices~$A\in\mathbb{R}^{d\times d}$ such that $e\cdot Ae > 0$ for all $e\in\Rd\setminus\{0\}$. If~$\a:\R \times \Rd \to \mathbb{R}^{d\times d}_+$ is a matrix-valued function, define the symmetric part $\s$ and the skew-symmetric part $\k$ by
\begin{equation}
\label{e.sk.def}
\s(t,x) := \frac{1}{2}(\a(t,x) + \a^t(t,x)) \ \text{ and } \ \k(t,x) := \frac{1}{2}(\a(t,x) - \a^t(t,x)), \ \forall (t,x) \in \R_+ \times \Rd\,,
\end{equation}
where $A^t$ denotes the transpose of a matrix $A\in \mathbb{R}^{d\times d}$. The set of symmetric matrices is denoted by~$\mathbb{R}^{d\times d}_{\mathrm{sym}}$, while the set of skew-symmetric matrices is denoted by~$\mathbb{R}^{d\times d}_{\mathrm{skew}}$. We introduce the minimal qualitative assumption that our fields belong to the space
\begin{equation}
\label{e.Omega.def}
\Omega := \{\text{Measurable } \a:\R\times \Rd \to \R^{d\times d}_+: \s,\s^{-1} \in L^1_{\mathrm{loc}}(\mathbb{R} \times \Rd)\,, \s^{-\nicefrac12}\k\s^{-\nicefrac12} \in L^\infty_{\mathrm{loc}}(\mathbb{R} \times \Rd) \}\,,
\end{equation}
and we show in Section~\ref{s.coarse.graining} that under this assumption the associated Cauchy-Dirichlet problem is well-posed. The canonical element of~$\Omega$ is~$\a(\cdot,\cdot)$, with~$\s(\cdot,\cdot)$ and~$\k(\cdot,\cdot)$ always taken to be the random fields defined in~\eqref{e.sk.def}, and we will often suppress the explicit dependence on~$t$ and~$x$. It is convenient to describe the field~$\a$ in terms of a~$\R^{2d\times 2d}_{\mathrm{sym}}$-valued field
\begin{equation}
\label{e.bfA.def}
\bfA(t,x)
:= 
\begin{pmatrix} 
( \s + \k^t\s^{-1}\k )(t,x) 
& -(\k^t\s^{-1})(t,x) 
\\ - ( \s^{-1}\k )(t,x) 
& \s^{-1}(t,x) 
\end{pmatrix}\,,
\end{equation}
which arises when considering the variational formulation of parabolic equations -- see~\cite[Appendix A]{ABM}. Instead of viewing~$\a$ as the canonical element of~$\Omega$ we may instead view~$\bfA$ as the canonical element, with~$\bfA\in \Omega$ implying that~$\bfA,\bfA^{-1}\in L^1_{\mathrm{loc}}(\R \times \Rd;\R^{2d\times 2d}_{\mathrm{sym}})$.

\smallskip

For any Borel subsets~$U\subseteq\Rd$ and~$I\subseteq\R$ define the~$\sigma$-field~$\mathcal{F}(I\times U)$ to be the~$\sigma$-field generated by the random variables
\begin{equation*}
\a \mapsto \int_{\R\times \Rd} e' \cdot \a e \,\varphi \quad \mbox{for fixed}\, e,e'\in\Rd \, \mbox{and}\, \varphi \in C_c^\infty(I\times U)\,,
\end{equation*}
and let~$\mathcal{F} := \mathcal{F}(\R \times \Rd)$. We will consider throughout the paper a probability measure~$\P$ on~$(\Omega, \mathcal{F})$, satisfying the three basic assumptions of stationary~\ref{a.stationarity}, ellipticity~\ref{a.ellipticity}, and ergodicity~\ref{a.CFS}.

\begin{enumerate}[label=(\textrm{P\arabic*})]
\setcounter{enumi}{0}
\item \emph{Stationarity with respect to~$\Z\times \Zd$--translations:}
\label{a.stationarity}
For every~$(s,y)\in \mathbb{Z}\times \mathbb{Z}^d$, let~$T_{s,y}: \Omega \to \Omega$ be the translation operator given by~$T_{s,y}\a := \a(\cdot + s, \cdot + y)$. Then
\begin{equation}
\label{e.P.stationary}
\P \circ T_{s,y} = \P, \quad \forall (s,y) \in \Z\times\Zd\,.
\end{equation}
\end{enumerate}

\smallskip

We will see in Section~\ref{ss.bfA.def} that the qualitative assumption~$\a\in\Omega$ is sufficient to define the coarse-grained matrices~$\s(I\times U),\s_*(I\times U)$ and~$\k(I\times U)$ which are the central objects of study in this paper. These matrices are collected together into a double-variable random field~$\bfA(I\times U)$ which represents a coarse-graining of the field given in~\eqref{e.bfA.def}. The coarse-grained matrices depend only on the restriction~$\a|_{I\times U}$ of the field~$\a$ to~$I\times U$. It is convenient to carry out the coarse-graining in the parabolic cubes defined for~$n\in\Z$ by
\begin{equation*}
I_n: = \bigg( -\frac{3^{2n}}{2},\frac{3^{2n}}{2}\bigg)\,, \quad \cu_n: = \bigg( -\frac{3^n}{2},\frac{3^n}{2}\bigg)^d\,, \quad \mbox{and} \quad \cudot_n := I_n \times \cu_n\,.
\end{equation*}
We will also use the notation
\begin{equation}
\label{e.Z.def}
\mathcal{Z}_n: = 3^{2n}\Z\times 3^n\Zd
\end{equation}
to refer to the standard space-time lattice and for~$s\in (0,1)$ define the coarse-grained ellipticity constants
\begin{align}
\label{e.lambda.infty.def}
\left\{
\begin{aligned}
& \Lambda_{s,\infty}(\cudot_n) \coloneqq \sup_{k \leq n} 3^{2s(k-n)} \max_{z\in \mathcal{Z}_k \cap \cudot_n} |(\s + \k^t\s_*^{-1}\k)(z+\cudot_k)|\,, \\
& \lambda_{s,\infty}(\cudot_n) \coloneqq  \biggl(\sup_{k \leq n} 3^{2s(k-n)} \max_{z\in \mathcal{Z}_k \cap \cudot_n} |\s_*^{-1}(z+\cudot_k)| \biggr)^{-1} \,. \\
\end{aligned}
\right.
\end{align}
Here~$|A|$ denotes the spectral norm of a square matrix~$A$; that is, the square root of the largest eigenvalue of~$A^tA$. Our ellipticity assumption states that the coarse-grained ellipticity constants are bounded by a deterministic constant above a large (random) scale.

\begin{enumerate}[label=(\textrm{P\arabic*})]
\setcounter{enumi}{1}
\item \emph{Ellipticity above a minimal scale.} \label{a.ellipticity}
There exist constants~$0 < \lambda_0 \leq \Lambda_0 < \infty$, an exponent~$\gamma \in [0,1)$, an increasing function~$\Psi_\S:\R_+ \to [1,\infty)$ and a constant~$K_{\Psi_\S}\in (1,\infty)$ satisfying the growth condition
\begin{equation}
\label{e.Psi.S.growth}
t \Psi_\S(t) \leq \Psi_\S(K_{\Psi_\S}  t), \quad \forall t \in [1,\infty)\,,
\end{equation}
and a nonnegative random variable~$\S$ satisfying the bound
\begin{equation}
\label{e.S.integrability}
\P \bigl[ \S > t   \bigr]
\leq 
\frac{1}{\Psi_\S(t)}
\,, 
\quad \forall t\in (0,\infty) \,,
\end{equation}
such that for every~$m\in\mathbb{Z}$,
\begin{equation}
\label{e.ellip.assumption}
3^m \geq \mathcal{S} \implies \Lambda_{\nicefrac{\gamma}{2},\infty}(\cudot_m) \leq \Lambda_0 \quad \mbox{and} \quad \lambda_0 \leq \lambda_{\nicefrac{\gamma}{2},\infty}(\cudot_m)\,.
\end{equation}
\end{enumerate}

One way to state the classical assumption of uniform ellipticity is to assume that there exist constants~$0 < \lambda \leq \Lambda < \infty$ such that
\begin{equation}
\label{e.uniform.ellipticity}
\s^{-1} \leq \lambda^{-1} \Id \quad \mbox{and} \quad \s + \k^t\s^{-1}\k \leq \Lambda\Id\,,
\end{equation}
with the inequality in the sense of the Loewner partial ordering.
If the coefficient field is uniformly elliptic then the coarse-grained matrices~$(\s + \k^t\s_*^{-1}\k)(z+\cudot_k)$ and~$\s_*^{-1}(z+\cudot_k)$ are controlled by the uniform ellipticity constants (see~\eqref{e.a.bounds}), and this implies that the coarse-grained ellipticity constants~$\Lambda_{0,\infty}(\cudot_m)$ and~$\lambda^{-1}_{0,\infty}(\cudot_m)$ are controlled by the uniform ellipticity constants for every~$m\in\mathbb{Z}$. However, the ellipticity assumption~\ref{a.ellipticity} is more general than a uniform ellipticity condition in that it permits degenerate and/or singular coefficient fields, provided that the degeneracy, parametrized by~$\gamma$, is not too strong. Examples of time-independent fields satisfying this assumption can be found in~\cite[Appendix D]{AK.HC}, and we note that one of the motivations for the coarse-grained ellipticity assumption is that it is renormalizable, in the sense of Lemma~\ref{l.renormalize.ellipticity}.

\smallskip

In order to state our space-time ergodicity assumption we first make one definition. Given an~$\mathcal{F}$-measurable random variable~$X$ on~$\Omega$ and a Borel subset~$V\subseteq \R \times \Rd$, let
\begin{align}
\label{e.malliavin.def}
|D_V &  X|(\bfA) \nonumber \\
& := \limsup_{t\to 0} \frac{1}{2t}\sup \biggl\{ X(\bfA_1)-X(\bfA_2) : \bfA_1,\bfA_2 \in \Omega\,, |\bfA^{-\nicefrac12}\bfA_i\bfA^{-\nicefrac12} - \Itwod| \leq t\indc_V\,, \forall i\in\{1,2\} \biggr\}\,,
\end{align}
where~$\bfA\in\Omega$.

\begin{enumerate}[label=(\textrm{P\arabic*})]
\setcounter{enumi}{2}
\item \emph{Concentration for sums ($\CFS$):}
\label{a.CFS}
There exist~$\beta \in \left[0,1\right)$,~$\nu\in (\gamma,\frac{d+2}{2}]$, an increasing function~$\Psi:\R_+ \to [1,\infty)$ and a constant~$K_\Psi\in [3,\infty)$ satisfying the growth condition
\begin{equation}
\label{e.Psi.growth}
t \Psi(t) \leq \Psi(K_\Psi  t), \quad \forall t \in [1,\infty)\,,
\end{equation}
such that,
for every~$m,n\in\N$ with $\beta m < n<m$ and  family $\{ X_z \,:\, z\in \mathcal{Z}_n \cap \cudot_m\}$ of random variables satisfying, for every~$z\in \mathcal{Z}_n \cap \cudot_m$,
\begin{equation} 
\label{e.CFS.ass}
\left\{
\begin{aligned}
& \E\left[ X_z \right]=0 \,, 
\\ & 
\left| X_z \right| \leq 1\,,
\\ &   
\left| D_{z+\cudot_n} X_z\right| \leq 1 \,,
\\
& X_z \ \ \mbox{is $\mathcal{F}(z+\cudot_n)$--measurable} \,, \\
\end{aligned}
\right.
\end{equation}
we have the estimate 
\begin{equation} 
\label{e.CFS}
\P \Biggl[ 
\biggl| \,
\avsum_{z\in \mathcal{Z}_n \cap \cudot_m}  X_z
\biggr| 
\geq 
t3^{-\nu(m-n)}
\Biggr]
\leq
\frac{1}{\Psi(t)}
\,, \quad \forall t\in [1,\infty) \,.
\end{equation}
\end{enumerate}

The standard example we have in mind is a field with finite space-time range of dependence. However, we note that~\ref{a.CFS} is much more general, and many examples of time-independent fields satisfying this condition are given explicitly in~\cite[Chapter 3]{AK.Book}. It is convenient to have notation to refer to estimates of the form~\eqref{e.CFS}, so given an increasing function~$\Psi:\mathbb{R}_+ \to [1,\infty)$ we write~$X \leq \mathcal{O}_{\Psi}(A)$ as shorthand for
\begin{equation*}
\mathbb{P}[ X > tA] \leq \frac{1}{\Psi(t)}\,, \quad \forall t \in [1,\infty)\,.
\end{equation*}
Inequalities of this type are discussed further in~\cite[Appendix C]{AK.HC}.

Our assumptions are stated for general ellipticity and ergodicity parameters. To illustrate how these apply in a specific context, suppose that~$\a$ is a uniformly elliptic field satisfying~\eqref{e.uniform.ellipticity} with constants~$0<\lambda \leq \Lambda < \infty$, and with finite range of dependence~$L$ in space and~$T$ in time. If we define the dimensionless variables
\begin{equation}
\label{e.rescale.variables}
t' = \frac{\lambda t}{L^2}\,, \quad x' = \frac{x}{L} \,, \quad \mbox{and} \quad \tilde{\a}(t',x') = \lambda^{-1}\a(t,x)\,,
\end{equation}
then the new coefficient field~$\tilde{\a}$ has uniform ellipticity lower bound~$1$, upper ellipticity bound~$\Lambda/\lambda$, range of dependence 1 in space, and range of dependence $\lambda T/L^2$ in time. Up to rescaling, this reduces the problem to the two dimensionless parameters~$\Lambda/\lambda$ (appearing in~\ref{a.ellipticity}) and~$\lambda T / L^2$ (appearing in~\ref{a.CFS}) -- see Section~\ref{ss.rescaling}.

\smallskip

\subsection{Main Results}

In this sub-section we state two main results. The first is a bound on the coarse-grained matrices, while the second is a homogenization statement for the Dirichlet problem. Convergence of the coarse-grained matrices can be viewed as the fundamental object of study because bounds on the homogenization error at the level of the coarse-grained matrices imply, deterministically, homogenization of the Dirichlet problem, as explained in Section~\ref{ss.homogenize}.

We first introduce some notation. For every scale~$3^n$ we define in~\eqref{e.all.homs.def} the symmetric, deterministic matrices~$\shom(\cudot_n),\shom_*(\cudot_n)$ and a deterministic matrix~$\khom(\cudot_n)$ which satisfy
\begin{equation}
\left\{
\begin{aligned}
(\shom +\khom^t\shom_*^{-1}\khom)(\cudot_n) &= \mathbb{E}[(\s +\k^t\s_*^{-1}\k)(\cudot_n)]\,,  \\
\shom_*(\cudot_n) &= \mathbb{E}[ \s_*^{-1}(\cudot_n)]^{-1}\,. \\
\end{aligned}
\right.
\end{equation}
There exist a symmetric, deterministic matrix~$\shom$ and a skew-symmetric, deterministic matrix~$\khom$ (the homogenized matrices, defined in Section~\ref{ss.homogenized.mat}) such that~$\shom(\cudot_n),\shom_*(\cudot_n) \to \shom$ and~$\khom(\cudot_n) \to \khom$ as~$n\to\infty$, and the full homogenized matrix is defined by~$\ahom = \shom + \khom$. We also denote~$\bhom = \shom + \khom^t \shom^{-1}\khom$, which is an upper ellipticity bound for the homogenized matrix.

The degree to which we have homogenized by scale~$3^n$ is characterized by the ratio of the ellipticity upper bound~$(\shom+\khom^t\shom_*^{-1}\khom)(\cudot_n)$ to the lower bound~$\shom_*(\cudot_n)$, modulo a ``centring" operation which we describe in Section~\ref{ss.bfA.def}. This is quantified by
\begin{equation}
\Theta_n := 
\min_{\h_0 \in \R^{d\times d}_{\mathrm{skew}}}
\bigl| ( \shom_{*}^{-\nicefrac12}(\cudot_n) (\shom(\cudot_n) + (\khom(\cudot_n) - \h_0)^t \shom_*^{-1}(\cudot_n) (\khom(\cudot_n) - \h_0) ) \,\shom_{*}^{-\nicefrac12}(\cudot_n)  )  \bigr|
\,,
\end{equation}
which converges monotonically downwards to 1 as~$n\to \infty$. It is one of the fundamental assertions in~\cite{AK.HC} that the quantity~$\Theta_n - 1$ is a good quantifier of the homogenization error at scale~$3^n$ and can be iterated to obtain quantitative convergence estimates. For this reason our first main result in~\eqref{e.Theta.conv.thm} is stated in terms of~$\Theta_n$. The particular focus of the following theorem is the dependence of the homogenization length scale on ellipticity, for which we define
\begin{equation}
\Pi_{\mathrm{par}} : = \max\bigg\{\frac{\Lambda_0}{\lambda_0},\lambda_0^{-1},\lambda_0\bigg\}\,,
\end{equation}
with~$\Lambda_0$ and~$\lambda_0$ as in~\eqref{e.lambdas.def}.

\begin{theorem}[Convergence of the coarse-grained matrices]
\label{t.theoremB}
Suppose that~$\mathbb{P}$ satisfies~\ref{a.stationarity},~\ref{a.ellipticity} and~\ref{a.CFS}. There exists a constant~$c(d)\in (0,\nicefrac{1}{4}]$ and exponents
\begin{equation}
\alpha := (\min\{\nu,1\} - \gamma)(1-\beta) \quad \mbox{and} \quad \kappa := \min\{c,c\alpha\}
\end{equation}
such that the following statements hold:
\begin{itemize}
\item \textbf{Estimate of the homogenization length scale:} There exists a constant~$C(d)<\infty$ and length scale
\begin{equation}
L := \exp \biggl( \frac{C}{\alpha}\log\biggl( \frac{K_{\Psi_{\S}}K_{\Psi}\Pi_{\mathrm{par}}}{\alpha}\biggr)\log(1+\Lambda_0/\lambda_0)  \biggr)
\end{equation}
such that
\begin{equation}
\label{e.Theta.conv.thm}
\Theta_n - 1 \leq \left(\frac{L}{3^n}\right)^{\kappa} \,.
\end{equation}
\item \textbf{Quenched convergence of the coarse-grained matrices:} For any~$\delta >0$ and~$\gamma' \in (\gamma,1)$ there exist a constant~$C = C(d,K_{\Psi},\gamma'-\gamma,\kappa,\delta)$, an exponent~$\theta:= \frac{1}{8}\min\{\kappa,\gamma'-\gamma\}$ and a random minimal scale~$\mathcal{Y}_{\delta,\gamma'}$ satisfying
\begin{equation}
\mathcal{Y}_{\delta,\gamma'}^{(\nu-\gamma)(1-\beta)} = \O_{\Psi} \bigl( C L^{\nicefrac{d}{\kappa}} \bigr)
\end{equation}
such that if~$3^m \geq \max\{\mathcal{Y}_{\delta,\gamma'}, \mathcal{S}\}$ then for every integer~$k\leq m$,
\begin{align}
\label{e.b.quenched}
\b(z+\cudot_k) \leq \biggl( 1 + \delta 3^{\gamma'(m-k)} \biggl(\frac{\max\{\mathcal{Y}_{\delta,\gamma'},\S\}}{3^m} \biggr)^{\theta} \biggr) \bhom \qquad \forall z\in \mathcal{Z}_k \cap \cudot_m \,,
\end{align}
and
\begin{align}
\label{e.s.quenched}
\s_*^{-1}(z+\cudot_k) \leq \biggl( 1 + \delta 3^{\gamma'(m-k)} \biggl(\frac{\max\{\mathcal{Y}_{\delta,\gamma'},\S\}}{3^m} \biggr)^{\theta} \biggr) \shom^{-1} \qquad \forall z\in \mathcal{Z}_k \cap \cudot_m \,.
\end{align}
\end{itemize}
\end{theorem}

The statement of Theorem~\ref{t.theoremB} simplifies if we fix a particular setting. To illustrate this, we consider in Corollary~\ref{c.frd} a uniformly elliptic coefficient field with finite-range of dependence, and apply Theorem~\ref{t.theoremB} along with the rescaling~\eqref{e.rescale.variables}. For simplicity we state only the quenched convergence of the coarse-grained matrices.

\begin{corollary}
\label{c.frd}
Suppose that~$\a$ is a coefficient field with law~$\mathbb{P}$ satisfying~\ref{a.stationarity}, the uniform ellipticity assumption~\eqref{e.uniform.ellipticity} with constants~$0<\lambda \leq \Lambda <\infty$, and with finite range of dependence 1 in space and $T$ in time: that is, given Borel subsets $U,V \subset \mathbb{R}^d$ and~$I,J\subset \mathbb{R}$
\begin{equation*}
\dist(U,V) \geq 1 \quad \mbox{or} \quad \dist(I,J) \geq T \implies \mathcal{F}(I\times U) \mbox{ and } \mathcal{F}(J\times V) \mbox{ are } \mathbb{P}\mbox{-independent}.
\end{equation*}
There exist a constant~$c(d) \in (0,\nicefrac{1}{4}]$, a constant~$C(d) < \infty$, an exponent~$\theta >0$, a length scale
\begin{equation}
\label{e.L.frd}
L \coloneqq \exp( C\log^2(1+\Lambda/\lambda) ) + C\sqrt{\lambda T}\,,
\end{equation}
and a random minimal scale~$\mathcal{X}$ satisfying
\begin{equation}
\mathcal{X} = \O_{\Gamma_2}(CL^C)\,, \mbox{ where } \Gamma_2(t) = e^{\frac{t^2}{2}} -1\,,
\end{equation}
such that if~$3^m \geq \mathcal{X}$ then for every integer~$k\leq m$ and every~$z\in (\lambda^{-1}3^{2k}\mathbb{Z} \times 3^k\mathbb{Z}^d) \cap (\lambda^{-1} 3^{2m}\mathbb{Z} \times 3^m \mathbb{Z}^d)$,
\begin{align}
\label{e.frd.quenched}
\left\{
\begin{aligned}
& \b(z+ (\lambda^{-1}I_k \times \cu_k) ) \leq \biggl( 1 + 3^{c(m-k)} \biggl(\frac{\mathcal{X}}{3^m} \biggr)^{\theta} \biggr) \bhom \,, \\
& \s_*^{-1}(z+ (\lambda^{-1}I_k \times \cu_k) ) \leq \biggl( 1 + 3^{c(m-k)} \biggl(\frac{\mathcal{X}}{3^m} \biggr)^{\theta} \biggr) \shom^{-1} \,.
\end{aligned}
\right.
\end{align}
\end{corollary}

Corollary~\ref{c.frd} states that the homogenization length scale is at most on the order of the constant in~\eqref{e.L.frd}. The~$\sqrt{\lambda T}$ term ensures that the diffusive time scale~$L^2/\lambda$ is larger than the correlation time, so that averaging can occur in time as well as in space. Beyond this scale we see dependence in the ellipticity contrast of the form~$\exp(C\log^2(1+\Lambda/\lambda) )$, matching the estimate obtained in the elliptic case in~\cite{AK.HC}. As noted in the introduction to~\cite{AK.HC}, the optimal estimate on the homogenization length scale is expected to have a power law dependence on the ellipticity of the problem. This is not achieved here and remains a difficult open problem.

\smallskip

Our second main result is that convergence of the coarse-grained matrices implies, deterministically, homogenization of the associated Dirichlet problem. In Section~\ref{s.homogenize} we prove coarse-grained parabolic inequalities and a homogenization ``black box" theorem which controls the homogenization error of the PDE by a multiscale quantity measuring the homogenization error in the coarse-grained matrices. Combining this with quenched convergence of the coarse-grained matrices implies homogenization at an algebraic rate, with the homogenization length scale given as in Theorem~\ref{t.theoremB}. If we assume that our coefficient field is uniformly elliptic with finite range of dependence, as in Corollary~\ref{c.frd}, then the statement of Theorem~\ref{t.theoremA} holds with the parameters given in Corollary~\ref{c.frd}, because the homogenization statement is obtained by combining the convergence of the coarse-grained matrices with a deterministic coarse-graining estimate for the PDE.

Our homogenization theorem is stated in domains adapted to~$\ahom$, because in these coordinates~$\ahom$ looks like the identity and the dependence of constants on~$\ahom$ can be made explicit. These domains are defined by
\begin{equation}
\cus_0 = |\shom^{-\nicefrac12}|\shom^{\nicefrac12}( \cu_0)\,, J_0 = \left[-\frac{1}{2|\shom^{-1}|}, \frac{1}{2|\shom^{-1}|}\right]\,, \quad \mbox{ and } \quad \cusdot_0 =  J_0 \times  \cus_0\,,
\end{equation}
as in Section~\ref{ss.subadditivity}. To simplify the statement of the theorem, let~$t_0 = -\nicefrac{1}{(2|\shom^{-1}|)}$.

\begin{theorem}[Homogenization of the Dirichlet problem]
\label{t.theoremA}
Suppose that $\mathbb{P}$ satisfies assumptions\newline \ref{a.stationarity}, \ref{a.ellipticity}, and \ref{a.CFS}. Suppose that~$\nicefrac{\gamma}{2} < s < \nicefrac{1}{2}$, let~$\mathcal{Y}_{1,s+\nicefrac{\gamma}{2}}$ and~$\theta>0$ be the random scale and exponent given in Theorem~\ref{t.theoremB}, and for each~$\epsilon \in (0,1)$ define~$\a^\epsilon(t,x) =\a(\nicefrac{t}{\epsilon^2}, \nicefrac{x}{\epsilon})$. There exist
\begin{equation}
\rho = \frac{\theta s}{4s+\gamma} \quad \mbox{and} \quad \mathcal{X} = \max\{\mathcal{Y}_{1,s+\nicefrac{\gamma}{2}},\mathcal{S}\} ^{\nicefrac{\theta}{2\rho}}
\end{equation}
such that the following homogenization statement holds: given data~$\mathbf{f}\in B^{s}_{2,2}(\cusdot_0)^d$ and~$u_0\in L^2(\cus_0)$, if~$u^\epsilon$ and~$v$ are the unique solutions to
\begin{equation}
\left\{
\begin{aligned}
& \partial_t u^\epsilon -\nabla \cdot \a^\epsilon \nabla u^\epsilon = \nabla \cdot \mathbf{f} & \mbox{ in } \cusdot_0\\
& u^\epsilon = 0 & \mbox{ on } J_0 \times \partial \cus_0 \\
& u^\epsilon = u_0 & \mbox{ at } t = t_0 \\
\end{aligned}
\right.
\qquad 
\left\{
\begin{aligned}
& \partial_t v - \nabla \cdot \ahom \nabla v = \nabla \cdot \mathbf{f} & \mbox{ in } \cusdot_0 \\
& v = 0 & \mbox{ on } J_0 \times \partial \cus_0 \\
& v = u_0 & \mbox{ at } t = t_0 \,, \\
\end{aligned}
\right.
\end{equation}
then for every~$\epsilon^{-1} \geq \mathcal{X}$ we have
\begin{equation}
\label{e.L2.homogenize}
\| u^\epsilon - v\|_{\L^2(\cusdot_0)} \leq C(d,\gamma,s)  (\mathcal{X}\epsilon)^{\rho}  ( \||\shom^{-1}|\mathbf{f}\|_{\underline{B}^{s}_{2,2}(\cusdot_0)} + \|u_0\|_{\L^2(\cus_0)} )\,.
\end{equation}
\end{theorem}

We are also able to handle data on the spatial boundary, but we omit this in the theorem because it complicates the norm of the data appearing on the right-hand side of~\eqref{e.L2.homogenize} -- see Theorem~\ref{t.homogenize} and Remark~\ref{r.boundary}. Finally, the space~$B^s_{2,2}$ is defined in~\eqref{e.Besov.gen.def}.

\subsection{Outline of the paper}

The key objects in the paper are the coarse-grained matrices and coarse-grained ellipticity constants, which we define in Section~\ref{s.coarse.graining}. The definitions we make are equivalent to those in~\cite{ABM}, but with definitions and additional properties that allow us to parallel the coarse-graining theory of~\cite{AK.HC}. The coarse-graining properties of Section~\ref{s.coarse.graining} are the input for the high-contrast homogenization proof in Section~\ref{s.hc} and the small contrast iteration in Section~\ref{s.smallcontrast}. These sections follow closely the proof in the elliptic case in~\cite{AK.HC}, up to technical details and the use of parabolic functional inequalities. In Section~\ref{s.homogenize} we prove parabolic coarse-grained inequalities, including Poincar\'e and Caccioppoli estimates, and prove a black box coarse-grained homogenization statement.

\section{The coarse-grained diffusion matrices}
\label{s.coarse.graining}

\subsection{Sobolev space framework}
In this section we show that the Cauchy-Dirichlet problem is well-posed on bounded domains for coefficient fields~$\a\in\Omega$. Given a coefficient field~$\a$ we define
\begin{equation}
\label{e.def.sk}
\s(t,x) = \frac{\a(t,x) + \a^t(t,x)}{2} \quad \mbox{and} \quad \k(t,x) = \frac{\a(t,x) - \a^t(t,x)}{2}\,,
\end{equation}
and the assumption~$\a\in\Omega$ states that
\begin{equation}
\s,\s^{-1} \in L^1_{\mathrm{loc}}(\mathbb{R} \times \Rd) \quad \mbox{and} \quad \s^{-\nicefrac12}\k\s^{-\nicefrac12} \in L^\infty_{\mathrm{loc}}(\mathbb{R} \times \Rd)\,.
\end{equation}
For any finite interval~$I$ and bounded Lipschitz domain~$U$, define
\begin{equation}
\label{e.weird.para.def}
\big\| u \bigr\|_{W_\s^1(I\times U)}
\coloneqq
\Bigl( \int_I \int_U |u(t,x)|^2 \,dxdt + \int_I\int_U \nabla u(t,x) \cdot \s(t,x) \nabla u(t,x) \, dx dt
\Bigr)^{\nicefrac12} 
\,,
\end{equation}
and let~$W_{\s}^1(I\times U)$ be the completion of~$C^\infty(I\times U)$ with respect to this norm. Since~$\s,\s^{-1}\in L^1(I\times U)$, the space~$W_\s^1(I\times U)$ is a complete Hilbert space by~\cite[Theorem 1.11]{KO84}. By H\"older's inequality, 
\begin{equation}
\label{e.Omega.imp}
u\in W_\s^1(I\times U)
\implies 
\nabla u , \,\a\nabla u \in L^1(I\times U)
\,,
\end{equation}
so in particular~$W_{\s}^1(I\times U) \hookrightarrow L^1(I; W^{1,1}(U))$. If~$u\in W_{\s}^1(I\times U)$ then for almost every~$t$ the function~$u(t,\cdot)$ will belong to the space~$H_{\s(t,\cdot)}^1(U)$, defined as the completion of~$C^\infty(U)$ with respect to the norm
\begin{equation}
\label{e.Hs.def}
\|u\|_{H_{\s(t,\cdot)}^1(U)} \coloneqq \biggl( \int_U |u(t,x)|^2\, dx  + \int_U \nabla u(t,x) \cdot \s(t,x) \nabla u(t,x)\, dx \biggr)^{\nicefrac12}\,.
\end{equation}
The standard~$W^{1,1}$ trace operator is a continuous operator from~$H^1_{\s(t,\cdot)}(U) \to L^1(\partial U)$ for almost every~$t$; we therefore define~$W^1_{\s,0}(I\times U)$ to be the closed subspace of~$W_{\s}^1(I\times U)$ with zero trace at almost every time, which coincides with the closure of~$C^\infty_c (I\times U)$ with respect to the norm~\eqref{e.weird.para.def}. The dual to this space is denoted~$W_{\s}^{-1}(I\times U)$ and equipped with the dual norm
\begin{equation}
\label{e.dual.norm}
\| f \|_{W_{\s}^{-1}(I\times U)} \coloneqq \sup \big\{ \langle f,g \rangle : \|g\|_{W_{\s,0}^{1}(I\times U)} \leq 1 \bigr\}\,,
\end{equation}
where~$\langle , \rangle$ denotes the duality pairing. As in~\cite[Lemma 2.1]{CS1984_1} and~\cite[Lemma 40.2]{Treves},
\begin{equation}
\label{e.continuous.in.time}
u \in W_{\s,0}^1(I\times U) \quad \mbox{and} \quad \partial_t u \in W_{\s}^{-1}(I\times U) \implies u \in C(I; L^2(U))\,,
\end{equation}
which is the sense in which initial data will be understood.

Given~$f \in W_{\s}^{-1}(I\times U)$ and~$u_0 \in L^2(U)$, we now consider the Cauchy-Dirichlet problem
\begin{equation}
\label{e.H1a.Dirichlet}
\left\{
\begin{aligned}
& \partial_t u -\nabla \cdot \a\nabla u = f & \mbox{ in } &  \ (0,T)\times U\,,
\\
& u = 0 & \mbox{ on } & (0,T] \times \partial U
\\
& u = u_0 & \mbox{ at } & t = 0\,.
\end{aligned}
\right.
\end{equation}
Here the equation is understood to hold as an equality in~$W_{\s}^{-1}((0,T)\times U)$, the spatial boundary data holds in the sense that~$u \in W_{\s,0}^1((0,T)\times U)$, and the initial condition~$\lim_{t\to 0} u(t,\cdot) = u_0(\cdot)$ is understood as an~$L^2(U)$ limit, in view of~\eqref{e.continuous.in.time}. We will proceed as in~\cite[Chapters 40 and 41]{Treves}, using as our main tool the following statement of the Lions-Lax-Milgram lemma, reproduced from~\cite[Lemma 41.2]{Treves}.

\begin{lemma}[Lions-Lax-Milgram Lemma]
\label{l.lions.lax}
Suppose that~$H$ is a Hilbert space,~$\Phi$ is a linear subspace of~$H$ and~$B:H \times \Phi \to \R$ is a bilinear form such that for each~$\varphi \in \Phi$,~$B[\cdot,\varphi]$ is a continuous linear functional on~$H$, and there exists~$c>0$ such that
\begin{equation}
c\|\varphi\|_H^2 \leq B[\varphi,\varphi] \quad \forall \varphi \in \Phi\,.
\end{equation}
Then for every continuous linear functional~$F$ on~$H$ there exists~$u\in H$ such that
\begin{equation}
B[u,\varphi] = F(\varphi) \quad \forall \varphi\in \Phi \,.
\end{equation}
Moreover,~$\|u\|_H \leq c^{-1} \|F\|_{H'}$.
\end{lemma}

We will actually apply this lemma to find a function~$v$ solving
\begin{equation}
\label{e.eqn.for.v}
\left\{
\begin{aligned}
& \partial_t v -\nabla \cdot \a\nabla v + v = e^{-t}f & \mbox{ in } &  \ (0,T)\times U\,,
\\
& v = 0 & \mbox{ on } & (0,T] \times \partial U
\\
& v = u_0 & \mbox{ at } & t = 0\,.
\end{aligned}
\right.
\end{equation}
and recover the solution to~\eqref{e.H1a.Dirichlet} by~$u(t,x) = e^t v(t,x)$. We will take as our Hilbert space~$H$ the set of all pairs~$(v,v_0)\in W_{\s}^1((0,T)\times U) \times L^2(U)$ equipped with the scalar product
\begin{equation}
((w,w_0),(v,v_0))_H \coloneqq  \int_0^T \int_U \bigl( wv + \nabla w \cdot \s\nabla v\bigr) + \int_U w_0 v_0\,.
\end{equation}
The linear subspace~$\Phi \subset H$ will be the set of pairs~$(\varphi, \varphi_0)$ such that~$\varphi \in C^\infty(I\times U)$,~$\varphi_0(\cdot) = \varphi(0,\cdot)$, and~$\varphi$ vanishes on~$(\{T\} \times U) \cup ((0,T] \times \partial U)$. Finally, our bilinear form is defined by
\begin{equation}
\label{e.bilinear}
B[v,\varphi] \coloneqq \int_0^T\int_U \bigl( - v \partial_t \varphi + v\varphi + \a\nabla v \cdot \nabla \varphi \bigr) \,.
\end{equation}
The conditions of Lemma~\ref{l.lions.lax} are verified immediately. Given~$f\in W_{\s}^{-1}((0,T) \times U)$ we then define a linear functional on~$H$ by
\begin{equation}
F((v,v_0)) = \langle e^{-t}f, v\rangle_{W_{\s}^1} + \int_U u_0 v_0\,,
\end{equation}
and apply the lemma to conclude that there exists~$(v,v_0) \in H$ such that for all~$(\varphi,\varphi_0) \in \Phi$,
\begin{equation}
\label{e.weak.form.eqn}
\int_0^T\int_U \bigl( - v \partial_t \varphi + v\varphi + \a\nabla v \cdot \nabla \varphi \bigr) = \langle e^{-t}f, \varphi \rangle + \int_U u_0 \varphi_0\,.
\end{equation}
Because~$\s^{-\nicefrac12}\k\s^{-\nicefrac12} \in L^\infty((0,T)\times U)$ we have that~$e^{-t}f +\nabla \cdot \a \nabla v - v \in W_{\s}^{-1}((0,T)\times U)$ and therefore by~\eqref{e.weak.form.eqn} we conclude that~$\partial_t v \in W_{\s}^{-1}((0,T)\times U)$ and that~$v$ is a solution to~\eqref{e.eqn.for.v}. In order to prove that the solution is unique we test the equation for~$v$ with itself and conclude that the only solution with~$f=0$ and~$u_0=0$ is identically zero.

% I hoped that this could handle the case where the stream matrix is in s-weighted H^1, but this doesn't seem to be possible.
% A solution can be obtain by approximation (cut off s^{-1} from above and k from above and below), passing to the limit in
% W_{\s}^1 \cap C([0,T];L^2(U)), but we can't conclude that the time derivative exists better than distributionally. This means
% we can't test the equation with itself (at least in the obvious way) and I don't want to deal with this

\smallskip

If~$\s^{-\nicefrac12}\mathbf{f}\in L^2(I\times U)^d$ and~$u_0\in L^2(U)$, the Neumann problem
\begin{equation}
\label{e.H1a.Neumann}
\left\{
\begin{aligned}
& \partial_t u -\nabla \cdot \a\nabla u = \nabla \cdot \mathbf{f}
& \mbox{ in } & (0,T)\times U \,,
\\ 
& \mathbf{n} \cdot (\a\nabla u + \mathbf{f}) = 0 & \mbox{ on } & (0,T] \times \partial U\,,
\\
& u = u_0 & \mbox{ at } & t=0\,,
\end{aligned}
\right.
\end{equation}
can be solved similarly. The weak formulation of the equation is
\begin{equation*}
\int_I \int_U -u \partial_t \varphi + \a\nabla u \cdot \nabla \varphi = \int_I \int_U -\mathbf{f} \cdot \nabla \varphi + \int_U u_0\varphi_0\,, \quad \forall \varphi \in C^\infty(I\times U) : \varphi(T,\cdot) \equiv 0\,,
\end{equation*}
and we obtain the existence of a unique solution~$u\in W_{\s}^1(I\times U)$ such that~$\partial_t u \in \hat{W}_{\s}^{-1}(I\times U))$, where~$\hat{W}_{\s}^{-1}(I\times U)$ is defined as the dual to~$W_{\s}^1(I\times U)$.

\smallskip

\subsection{The coarse-grained matrices: definitions and basic properties}
\label{ss.bfA.def}

The above discussion indicates that the parabolic Cauchy-Dirichlet and Neumann problems are well-posed for coefficients~$\a\in \Omega$.

We introduce the (non-empty) solution space
\begin{equation}
\label{e.solution.space}
\A(I\times U) = \bigg\{u: \|\s^{\nicefrac12}\nabla u\|_{\L^2(I\times U)} < \infty\,, (u)_{I\times U} = 0\,, \partial_t u \in W_{\s}^{-1}(I\times U) \mbox{ and } \partial_t u = \nabla \cdot \a \nabla u \mbox{ in } I\times U \bigg\}\,,
\end{equation}
and the space of solutions to the adjoint equation
\begin{align}
\label{e.adjoint.solution}
\lefteqn{
\A^*(I\times U)
} \quad & \notag \\
& = \bigg\{u: \|\s^{\nicefrac12}\nabla u\|_{\L^2(I\times U)} < \infty\,, (u)_{I\times U} = 0\,, \partial_t u \in W_{\s}^{-1}(I\times U) \mbox{ and } \partial_t u = - \nabla \cdot \a^t \nabla u \mbox{ in } I\times U \bigg\}\,,
\end{align}
The space~$\mathcal{A}(I\times U)$ is a Hilbert space under the norm~$\|u\| \coloneqq \|\s^{\nicefrac12}\nabla u\|_{L^2(I\times U)}$. That this defines a norm follows from Proposition~\ref{p.dumb.poincare}, and the closure of the space follows from the weak formulation of the equation and the fact that if~$u\in \mathcal{A}(I\times U)$ then for a constant~$C$ depending on norms of~$\a$ but independent of~$u$,
\begin{equation*}
\|\partial_t u \|_{W_{\s}^{-1}(I\times U)} = \|\nabla \cdot \a\nabla u\|_{W_{\s}^{-1}(I\times U)} \leq C \|\s^{\nicefrac12}\nabla u\|_{L^2(I\times U)}\,.
\end{equation*}

For every realization of the coefficients~$\a\in\Omega$, bounded Lipschitz domain~$U\subseteq \Rd$, and finite time interval~$I\subseteq\R$ we define, for every~$p,q\in\Rd$, the quantity
\begin{equation}
\label{e.J.def}
J(I\times U,p,q) 
: =
\sup_{u \in  \A(I\times U)}
\fint_I \fint_{U} \biggl( - \frac12 \nabla u \cdot \s \nabla u - p \cdot \a \nabla u + q \cdot \nabla u \biggr) 
\,.
\end{equation} 
This is a well-posed variational problem, using the results of the previous subsection. The maximization is over the Hilbert space~$\mathcal{A}(I\times U)$, and the functional which is being maximized is upper-semi-continuous, strictly concave, and coercive. Therefore, by~\cite[Chapter II, Propositions 1 and 2]{Temam} we obtain the existence of a unique maximizer, denoted~$v(\cdot,\cdot,I\times U,p,q)$. By carrying out the first variation, the maximizer is a linear function of~$(p,q)$. It follows that the mapping~$(p,q) \mapsto J(I\times U,p,q)$ is quadratic. In fact, there exist positive-definite symmetric matrices~$\s_*(I\times U)$ and~$\s(I\times U)$ and a matrix~$\k(I\times U)$ (all~$\mathcal{F}(I \times U)$--measurable) such that
\begin{equation}
\label{e.J.mat}
J(I\times U,p,q) =
\frac 12p \cdot \s(I\times U)p 
+ \frac 12 (q+\k(I\times U) p) \cdot \s_*^{-1}(I\times U) (q+\k(I\times U) p) 
- p \cdot q \,.
\end{equation}
We also define
\begin{equation}
\label{e.b.def}
\b(I\times U) := (\s + \k^t\s_*^{-1}\k)(I\times U).
\end{equation}

The following properties, and their proofs, are identical to those in the elliptic case, and follow directly from the variational formulation in~\eqref{e.J.def}.

\begin{lemma}[Properties of the coarse-grained coefficients]
\label{l.J.basicprops}
For any finite interval~$I$, bounded Lipschitz domain~$U$, and~$p,q\in\Rd$, the following holds:

\begin{itemize} 
\item The coarse-grained matrices satisfy the bounds
\begin{align}
\label{e.a.bounds}
\biggl( \fint_{I} \fint_{U} \s^{-1} (t,x)\, dt dx \biggr)^{\!\!-1} 
\leq 
\s_*(I\times U) 
\quad \mbox{and} \quad
\b(I\times U)
\leq 
\fint_I \fint_U \bigl ( \s + \k^t \s_*^{-1} \k \bigr )(t,x)\, dt dx\,.
\end{align}

\item The first variation states that for every~$w\in \A(I\times U)$
\begin{equation}
\label{e.firstvar}
q\cdot \fint_I \fint_U \nabla w - p \cdot \fint_I \fint_U \a \nabla w 
=
\fint_I \fint_U \nabla w \cdot \s \nabla v(I\times U,p,q)\, . 
\end{equation}

\item The second variation states that for every~$w\in \A(I\times U)$ 
\begin{multline}
\label{e.quadresp}
J(I\times U,p,q) - \fint_I \fint_U \Bigl  ( -\frac12 \nabla w \cdot \s\nabla w -p\cdot \a\nabla w+ q\cdot \nabla w   \Bigr  )
\\
=
\fint_I \fint_U \frac12 \bigl ( \nabla v(I\times U,p,q) - \nabla w \bigr )\cdot \s\bigl ( \nabla v(I\times U,p,q) - \nabla w \bigr )\,.
\end{multline}

\item The value of~$J(I\times U,p,q)$ is given by the energy of the maximizer
\begin{equation}
\label{e.Jenergyv}
J(I\times U,p,q) = \fint_I \fint_U \frac12 \nabla v(I\times U,p,q) \cdot \s \nabla v(I\times U,p,q).
\end{equation}

\item The space-time averages of the gradient and flux of maximizers are given by
\begin{equation}
\label{e.a.formulas}
\left\{
\begin{aligned}
& \fint_I \fint_U  \nabla v(I\times U,p,q)
=
-p +
\s_*^{-1} (I\times U)
\bigl ( q + \k(I\times U)p \bigr )
\,,\\
&
\fint_I \fint_U  \a \nabla v(I\times U,p,q)
= 
\bigl (\Id -\k^t  \s_*^{-1}\bigr ) (I\times U)  q
-
\bigl (\s +\k^t  \s_*^{-1} \k   \bigr )(I\times U) p
\,.
\end{aligned}
\right.
\end{equation}

\item Subadditivity: for every disjoint partition~$\{I_i \times U_i\}_{i=1}^N$ of~$I\times U$ we have
\begin{equation}
\label{e.subaddJ}
J (I\times U, p, q)
\leq 
\sum_{i=1}^N \frac{|I_i\times U_i|}{|I\times U|} J(I_i\times U_i,p,q)
\end{equation}

\item We have the following coarse-graining inequalities: for every~$u\in \A(I\times U)$
\begin{align}
\label{e.fluxmaps}
\biggl | \fint_I\fint_{U} \bigl ( p \cdot \a \nabla u - q \cdot \nabla u \bigr ) \biggr |
& =
\biggl | \fint_I\fint_U \nabla u 
\cdot  \s \nabla v\bigl (I\times U, p,q \bigr )  \biggr |
\notag \\ & \leq
(2J \bigl (I\times U, p,q \bigr ) )^{\frac12}
\biggl( \fint_I\fint_U \nabla u \cdot \s \nabla u \biggr)^{\frac12}
\,.
\end{align}
%This inequality is useful when~$J(I\times U,p,q)$ is small, which requires~$q$ and~$p$ to be related and the gap between~$\s(I\times U)$ and~$\s_*(I\times U)$ to be small. For instance, choosing~$q=(\s_*-\k^t)(I\times U)p$ and taking the supremum over~$|p|=1$ yields, in view of~\eqref{e.diagonalset.nosymm}, 
%\begin{equation*}
%\biggl| \fint_I\fint_{U} \a \nabla w - (\s_*-\k^t)(I\times U) \fint_I\fint_U \nabla w \biggr |
%\leq
%2^{\frac12}
%\bigl| (\s-\s_*)(I\times U) \bigr|^{\nicefrac12} 
%\biggl( \fint_I\fint_U \nabla w \cdot \s \nabla w \biggr)^{\frac12}
%\,.
%\end{equation*}
and
\begin{align}
\label{e.energymaps}
\frac12\left( \fint_I\fint_U \nabla u \right) \cdot \s_*(I\times U) \left(\fint_I \fint_U \nabla u \right)
\leq
\fint_I\fint_U \frac12 \nabla u \cdot \s\nabla u \\
\label{e.energymaps.flux}
\frac12\left( \fint_I\fint_U \a \nabla u \right) \cdot \b^{-1} (I\times U) \left( \fint_I\fint_U \a \nabla u \right)
\leq
\fint_I\fint_U \frac12 \nabla u \cdot \s\nabla u 
\,.
\end{align}
\end{itemize}
\end{lemma}
\begin{proof}
Given the well-posedness of the variational problem~\eqref{e.J.def}, these properties follow exactly as in~\cite[Lemma 5.1]{AK.Book}.
\end{proof}

Inspired by the variational formulation of the parabolic problem, as in~\cite[Appendix A]{ABM}, we need to consider the adjoint operator and a double-variable quantity which considers both solutions to the parabolic equation and solutions to the adjoint problem. We first define
\begin{equation}
\label{e.Jstar.def}
J^*(I\times U,p,q') 
: =
\sup_{u \in  \A^*(I\times U)}
\fint_I \fint_{U} \biggl( - \frac12 \nabla u \cdot \s \nabla u - p \cdot \a^t \nabla u + q' \cdot \nabla u \biggr).
\end{equation}
All the properties of Lemma~\ref{l.J.basicprops} hold for~$J^*(I\times U,p,q')$, with the exception that the coarse-grained matrices will be the coarse-grained matrices of the reversed-in-time adjoint operator; we identify these matrices in~\eqref{e.J.mat.star} below. In order to define the double-variable quantities we introduce, for each pair~$(v,v^*) \in \mathcal{A}(I\times U) \times \mathcal{A}^*(I\times U)$, the notation
\begin{equation}
\label{e.X.def}
X(v,v^*) = \begin{pmatrix}
\nabla v + \nabla v^* \\ \a\nabla v - \a^t\nabla v^*
\end{pmatrix}\,,
\end{equation}
and define, for every~$P,Q \in \mathbb{R}^{2d}$,
\begin{align}
\label{e.bfJ.var}
\bfJ(I\times U, P,Q )  =
\sup_{\substack{v\in\mathcal{A}(I\times U) \\ v^* \in \mathcal{A}^*(I\times U)} }
\fint_I \fint_U
\biggl( -\frac12 X(v,v^*) \cdot \bfA X(v,v^*) -  P \cdot \bfA X(v,v^*)  + Q\cdot X(v,v^*)
\biggr)
\,.
\end{align}
Recall that~$\bfA$ is defined in~\eqref{e.bfA.def}. In view of the equality
\begin{align}
\label{e.double.equiv}
\lefteqn{
-\frac12 X(v,v^*) \cdot \bfA X(v,v^*) -  P \cdot \bfA X(v,v^*)  + Q\cdot X(v,v^*)
} \qquad \qquad & \notag \\
&  = -\frac{1}{2}\nabla v \cdot \a \nabla v - (p-p^*)\cdot \a\nabla v + (q^* - q) \cdot \nabla v
\notag \\ & \qquad
-\frac{1}{2}\nabla v^* \cdot \a \nabla v^* - (p^*+p)\cdot \a\nabla v^* + (q^* + q) \cdot \nabla v^*
\end{align}
it is clear that the functional in~\eqref{e.bfJ.var} is strictly concave, upper-semi-continuous and coercive over the product space~$\mathcal{A}(I\times U) \times \mathcal{A}^*(I\times U)$. By the same reasoning as for the variational problem in~\eqref{e.J.def}, this implies the existence of a unique maximizer~$(v,v^*)$; by~\eqref{e.double.equiv} we see that~$v$ is the maximizer in~\eqref{e.J.def} with parameters~$p-p^*$ and~$q^*-q$, while~$v^*$ is the maximizer in~\eqref{e.Jstar.def} with parameters~$p^*+p$ and~$q^*+q$. The well-posedness of the double-variable variational problem allows us to introduce the double-variable matrices and prove non-obvious facts about them. In view of~\cite[Lemma 2.6]{ABM} our definition~\eqref{e.bfJ.var} is equivalent to the~$J$ quantity in~\cite[Lemma 2.3]{ABM}. It follows that there exist symmetric, positive-definite matrices~$\bfA(I\times U)$ and~$\bfA_*(I\times U)$ such that for all~$p,p^*,q,q^* \in \Rd$,
\begin{equation}
\label{e.Jsplitting}
\bfJ \biggl(I\times U,
\begin{pmatrix}
p \\ q
\end{pmatrix}
,
\begin{pmatrix}
q^* \\ p^*
\end{pmatrix}
\biggr) =
\frac{1}{2} \begin{pmatrix}
p \\ q
\end{pmatrix} \cdot \bfA(I\times U) \begin{pmatrix}
p \\ q
\end{pmatrix}
+ \frac{1}{2}\begin{pmatrix}
q^* \\ p^*
\end{pmatrix}  \cdot \bfA_*^{-1}(I\times U)\begin{pmatrix}
q^* \\ p^*
\end{pmatrix} 
- \begin{pmatrix}
p \\ q
\end{pmatrix} \cdot \begin{pmatrix}
q^* \\ p^*
\end{pmatrix}\,.
\end{equation}

The following lemma collects the properties of the double-variable coarse-grained matrices. These properties follow from the well-posedness of the variational problem~\eqref{e.bfJ.var} and the representation~\eqref{e.Jsplitting}, using a combination of~\cite[Lemma 5.2]{AK.Book} and~\cite[Section 2B]{ABM}. Note that the quantity~$\mu(V,X)$ defined in~\cite{ABM} is equal to~$\frac{1}{2} X\cdot \bfA(I\times U)X$ and~$\mu_*(V,X^*)$ is equal to~$\frac{1}{2}X^* \cdot \bfA^{-1}_*(I\times U) X^*$.

\begin{lemma}[Further properties of the coarse-grained coefficients]
\label{l.J.further.props}

For every finite interval~$I$,\newline bounded Lipschitz domain $U$, and~$p,q,p^*,q^*\in\Rd$, the following holds:
\begin{itemize}	
\item The double-variable matrices have the representation
\begin{equation}
\label{e.bigA.def}
\bfA(I\times U)
:= 
\begin{pmatrix} 
( \s + \k^t\s_*^{-1}\k )(I\times U) 
& -(\k^t\s_*^{-1})(I\times U) 
\\ - ( \s_*^{-1}\k )(I\times U) 
& \s_*^{-1}(I\times U) 
\end{pmatrix}\,.
\end{equation}
and
\begin{equation}
\label{e.bigAstar.def}
\bfA_*^{-1}(I\times U):=
\begin{pmatrix}
\s_*^{-1}(I\times U) & -(\s_*^{-1}\k)(I\times U) \\
-(\k^t\s_*^{-1})(I\times U) & (\s+\k^t\s_*^{-1}\k)(I\times U)
\end{pmatrix}\,.
\end{equation}
\item[•] The double-variable matrices have the ordering
\begin{equation*}
\bigg(\fint_I\fint_U \bfA^{-1}(t,x)\, dtdx \bigg)^{-1} \leq \bfA_*(I\times U) \leq \bfA(I\times U) \leq \fint_I\fint_U \bfA(t,x)\, dtdx\,,
\end{equation*}
and consequently~$\s_*(I\times U) \leq \s(I\times U)$.
\item The adjoint quantity has the matrix representation
\begin{equation}
\label{e.J.mat.star}
J^*(I\times U,p,q) =
\frac 12p \cdot \s(I\times U)p 
+ \frac 12 (q-\k(I\times U) p) \cdot \s_*^{-1}(I\times U) (q-\k(I\times U) p) 
- p \cdot q \,.
\end{equation}
\item[•] The matrices~$\bfA(I\times U)$ and~$\bfA_*^{-1}(I\times U)$ are subadditive: for every disjoint partition~$\{I_i \times U_i\}_{i=1}^N$ of~$I\times U$ we have
\begin{equation}
\label{e.subadditivity}
\bfA (I\times U)
\leq 
\sum_{i=1}^N \frac{|I_i\times U_i|}{|I\times U|} \bfA (I_i\times U_i)
\quad \mbox{and} \quad 
\bfA_{*}^{-1} (I\times U)
\leq 
\sum_{i=1}^N \frac{|I_i\times U_i|}{|I\times U|} \bfA_{*}^{-1} (I_i\times U_i)
\,.
\end{equation}
\item[•] The quantity~$\k(I\times U)$ is not symmetric in general, but its symmetric part is controlled by the gap between~$\s(I\times U)$ and~$\s_*(I\times U)$:
\begin{equation}
\label{e.symm.part.k}
(\k + \k^t)(I\times U) \leq (\s - \s_*)(I\times U) \quad \mbox{and} \quad -(\k+\k^t)(I\times U) \leq (\s-\s_*)(I\times U).
\end{equation}
\end{itemize}
Moreover, the following useful algebraic identities hold:
\begin{itemize}
\item We have
\begin{equation}
\label{e.bfJ.def}
\frac{1}{2} \bfJ\bigg(I\times U,
\begin{pmatrix}
p \\ q
\end{pmatrix},
\begin{pmatrix}
q^* \\ p^*
\end{pmatrix}\bigg)
= J(I\times U,p-p^*,q^*-q) + J^*(I\times U, p^* + p, q^* + q).
\end{equation}
\item Both~$J(I\times U,p,q)$ and~$J^*(I\times U,p,q)$ can be represented in terms of the double-variable matrix as
\begin{align}
\label{e.Jaas.matform}
J(I\times U,p,q) 
=
\frac 12 
\begin{pmatrix} 
-p \\ q
\end{pmatrix}
\cdot \bfA(I\times U)
\begin{pmatrix} 
-p \\ q
\end{pmatrix}
-p\cdot q
\notag \\
J^*(I\times U,p,q) 
=
\frac 12 
\begin{pmatrix} 
p \\ q
\end{pmatrix}
\cdot \bfA(I\times U)
\begin{pmatrix} 
p \\ q
\end{pmatrix}
-p\cdot q\,.
\end{align}
\item By direct computation
\begin{equation}
\label{e.bigA.formulas.inv}
\left\{
\begin{aligned}
& 
\bfA^{-1}(I\times U)
= 
\begin{pmatrix} 
\s^{-1}(I\times U) 
& (\s^{-1}\k^t )(I\times U) 
\\ (\k \s^{-1})(I\times U) 
& (\s_* + \k \s^{-1}\k^t )(I\times U) 
\end{pmatrix}\,,
\\ & 
\bfA_*(I\times U)
= 
\begin{pmatrix} 
 (\s_* + \k \s^{-1}\k^t )(I\times U) 
& (\k \s^{-1})(I\times U) 
\\ (\s^{-1}\k^t )(I\times U) 
& \s^{-1}(I\times U)
\end{pmatrix}
\,.
\end{aligned}
\right.
\end{equation}
and for every~$\eta >0$,
\begin{equation}
\label{e.bigA.diag.bounds}
\left\{
\begin{aligned}
& 
\bfA(I\times U)
\leq 
\begin{pmatrix} 
(\s +(1+\eta^{-1})\k^t\s_*^{-1}\k)(I\times U)
& 0
\\ 0
& (1+\eta)\s_*^{-1}(I\times U) 
\end{pmatrix}\,,
\\ & 
\bfA^{-1}(I\times U)
\leq  
\begin{pmatrix} 
 (1+\eta)\s^{-1}(I\times U)
& 0 
\\ 0
& (\s_* + (1+\eta^{-1})\k \s^{-1}\k^t )(I\times U) 
\end{pmatrix}
\,.
\end{aligned}
\right.
\end{equation}
\item Introducing
\begin{equation}
\label{e.R.def}
\mathbf{R}:= \begin{pmatrix}
0 & \Id \\ \Id & 0
\end{pmatrix}\,,
\end{equation}
the two equations~\eqref{e.a.formulas} can be written
\begin{equation}
\label{e.all.averages}
\fint_I\fint_U \begin{pmatrix}
\nabla v \\ \a\nabla v
\end{pmatrix}
(\cdot,\cdot,I\times U,p,q) =
(\mathbf{R}\bfA(I\times U) + \Itwod)\begin{pmatrix}
-p \\ q
\end{pmatrix}\,.
\end{equation}
\item The two inequalities~\eqref{e.energymaps} and~\eqref{e.energymaps.flux} can be written
\begin{equation}
\label{e.energymaps.double}
\frac{1}{2}(X(v,v^*))_{I\times U} \cdot \bfA_*(I\times U) (X(v,v^*))_{I\times U} \leq \frac{1}{2}\fint_I\fint_U X(v,v^*)\cdot \bfA X(v,v^*)\,,
\end{equation}
for all~$(v,v^*) \in \mathcal{A}(I\times U) \times \mathcal{A}^*(I\times U)$.
\end{itemize}
\end{lemma}

Although the double-variable quantities can be algebraically expressed in terms of the coarse-grained matrices~$\s(I\times U)$,~$\s_*(I\times U)$ and~$\k(I\times U)$, the variational formulation of~\eqref{e.bfJ.var} yields new information. For example, the ordering~$\s_*(I\times U) \leq \s(I\times U)$ cannot easily be deduced otherwise. We also note here that~$\bfA(I\times U)$,~$\bfA^{-1}_*(I\times U)$,~$ \b(I\times U)$ and~$\s^{-1}_*(I\times U)$ are all subadditive because they are defined directly from variational problems, but there is no sense in which~$\s(I\times U)$ and~$\k(I\times U)$ are subadditive.

%Variational formulation ======================================
%
%The larger matrix~$\bfA(I\times U)$ can be thought of as a coarse-graining of the matrix in~\eqref{e.bfA.def}, and it has a variational interpretation (see~\cite[Lemma 5.3]{AK.Book}), which gives an alternative way of defining the coarse-grained matrices. To see this, begin by defining~$L^2_{\s,\mathrm{par},0}(I\times U)$ to be the closure of
%\begin{equation*}
%\{(\nabla v,\h)\in (C_c^\infty((I_-,I_+)\times U))^2: \nabla \cdot \h = \partial_t v \text{ in } \hat{W}^{-1}_\s(I\times U)\}
%\end{equation*}
%with respect to the norm~$(\nabla v,\h) \mapsto (\fint_I\fint_U (\nabla v,\h)\cdot \bfA (\nabla v, \h))^{\nicefrac12}$. Note that
%\begin{equation*}
%(\nabla v, \h) \in L^2_{\s,\mathrm{par},0}(I\times U) \implies v \in W^1_{\s,0}(I\times U), \, \s^{-\nicefrac12}\h \in L^2(I\times U), \, \text{ and }\s^{-\nicefrac12}\k\nabla v \in L^2(I\times U).
%\end{equation*}
%The coarse-grained double variable matrix then satisfies
%\begin{equation}
%\label{e.J.P0.Dirichlet}
%\frac12 P \cdot \bfA(I\times U) P 
%=
%\inf\biggl\{ 
%\fint_{I}\fint_{U} 
%\frac12 (X + P) \cdot \bfA (X + P)
%\, : \, 
%X \in  L^2_{\s,\mathrm{par},0}(I\times U)  
%\biggr\}
%\,.
%\end{equation}

The algebraic structure of the double-variable quantities is also very useful. If we define, for any matrix~$\h \in \mathbb{R}^{d\times d}$,
\begin{equation}
\label{e.Gk.def}
\mathbf{G}_{\h}:= \begin{pmatrix} 
\Id
& 0
\\ \h
& \Id
\end{pmatrix}\,,
\end{equation}
then
\begin{equation}
\mathbf{G}_{\h_1}\mathbf{G}_{\h_2} = \mathbf{G}_{\h_1+\h_2} \qquad \forall \h_1,\h_2 \in \mathbb{R}^{d\times d}\,,
\end{equation}
and the double-variable matrices have the form
\begin{equation}
\bfA(I\times U) = \mathbf{G}_{-\k(I\times U)}^t \begin{pmatrix}
\s(I\times U) & 0 \\ 0 & \s_*^{-1}(I\times U)\,.
\end{pmatrix}
\mathbf{G}_{-\k(I\times U)}
\end{equation}
Conjugation by any invertible matrix preserves partial ordering. In particular, for~$n\leq m$ the means of the coarse-grained matrices (defined in~\eqref{e.bfA.hom.def}) satisfy~$\bfAhom(\cudot_m) \leq \bfAhom(\cudot_n)$, so conjugating with~$\mathbf{G}_{\khom(\cudot_n)}$ and comparing the diagonal entries we obtain
\begin{equation}
\label{e.monotone.s}
\shom(\cudot_m) \leq \shom(\cudot_n) 
\quad \mbox{and} \quad
\shom_{*}(\cudot_m) \geq \shom_{*}(\cudot_n) \,,
\end{equation}
which is not obvious from the definitions in~\eqref{e.all.homs.def}.

%As another application, if we take~$n\in\N$ large enough that~$3^n \geq \S$ then~$\bfA(\cudot_n) \leq \bfE$ so conjugating with~$\mathbf{G}_{\k_0}$ we obtain
%\begin{equation*}
%\s_{*,0} \leq \s_*(\cudot_n) \quad \mbox{and} \quad \s(\cudot_n) + (\k(\cudot_n) - \k_0)^t\s_{*}^{-1}(\cudot_n)(\k(\cudot_n) - \k_0) \leq \s_0\,,
%\end{equation*}
%recalling the definitions of~$\s_0,\k_0$ and~$\s_{*,0}$ in~\eqref{e.E0.blocks.def}. Since~$\s_{*,0}$ is positive-definite, and~$\s_*(\cudot_n) \leq \s(\cudot_n)$ we obtain that~$\s_0$ is positive-definite. \red{which should really just been assumed in the first place?}

Conjugation by an invertible matrix also leaves the eigenvalues of ratios of pairs of coarse-grained matrices unchanged. That is, for any~$\h\in\R^{d\times d}$ (not necessarily skew symmetric) and pair of symmetric matrices~$\mathbf{D},\mathbf{E}\in\R^{2d\times 2d}$ such that~$\mathbf{D}$ is positive definite, if we define
\begin{equation}
\label{e.eigen.unchanged}
\mathbf{D}_\h:= \mathbf{G}_\h^t \mathbf{D}\mathbf{G}_\h \quad \mbox{and} \quad \mathbf{E}_\h:= \mathbf{G}_\h^t \mathbf{E}\mathbf{G}_\h\,,
\end{equation} 
then~$\mathbf{D}^{-\nicefrac12}\mathbf{E}\mathbf{D}^{-\nicefrac12}$ and~$\mathbf{D}_\h^{-\nicefrac12}\mathbf{E}_\h\mathbf{D}_\h^{-\nicefrac12}$ have the same eigenvalues. This conjugation operation has a specific application if~$\h$ is a constant skew-symmetric matrix,\footnote{The matrix~$\h$ may depend on time, but we will not use this.} because the solutions to the parabolic equation
\begin{equation*}
\partial_t u - \nabla \cdot \a \nabla u =0
\end{equation*}
remain the same if~$\a$ is replaced by~$\a-\h$. This invariance is expressed in the coarse-grained quantities, as noted in~\cite[Section 2.5]{AK.HC}: if~$\a$ is a coefficient field with coarse-grained coefficient matrix~$\bfA(I\times U)$,~$\mathbf{h}$ a constant skew-symmetric matrix, and~$\bfA_{\mathbf{h}}(I\times U)$ denotes the coarse-grained matrix associated to the coefficient field~$\a-\h$ then
\begin{align}
\label{e.Ak.formula}
\bfA_{\h} (I\times U)
&
=
\mathbf{G}_{\h}^{t}
\bfA(I\times U)
\mathbf{G}_{\h}
\end{align}
Comparing~\eqref{e.Ak.formula} to~\eqref{e.bigA.def} we see that subtraction of an anti-symmetric matrix depending only on time ``commutes" with the coarse-graining operation in the sense that it simply subtracts~$\h$ from~$\k(I\times U)$. We similarly define
\begin{equation}
\label{e.b.centered.def}
\b_{\h}(\cudot_n) = \s(\cudot_n) + (\k(\cudot_n) - \h)^t\s_{*}^{-1}(\cudot_n)(\k(\cudot_n) - \h)\,. 
\end{equation}

The double-variable matrices are convenient to work with and appear very naturally. For this reason we rewrite the ellipticity assumption~\ref{a.ellipticity} in a double-variable formulation.

\begin{enumerate}[label=(\textrm{P\arabic*\textdagger})]
\setcounter{enumi}{1}
\item \emph{Coarse-grained ellipticity on large scales.} \label{a.ellipticity.dagger}
There exist a symmetric, positive-definite matrix~$\mathbf{E}_0$, an exponent~$\gamma \in [0,1)$, an increasing function~$\Psi_\S:\R_+ \to [1,\infty)$, a constant~$K_{\Psi_\S}\in (1,\infty)$ satisfying the growth condition
\begin{equation}
\label{e.Psi.S.growth.one}
t \Psi_\S(t) \leq \Psi_\S(K_{\Psi_\S}  t), \quad \forall t \in [1,\infty)\,,
\end{equation}
and a nonnegative random variable~$\S$ which satisfies the bound
\begin{equation}
\label{e.S.integrability.one}
\P \bigl[ \S > t   \bigr]
\leq 
\frac{1}{\Psi_\S(t)}
\,, 
\quad \forall t\in (0,\infty) \,,
\end{equation}
such that, for every~$m,n\in\mathbb{Z}$ with~$n\leq m$ we have
\begin{equation}
\label{e.ellipticity}
3^m \geq \S 
\implies
\bfA(z+\cudot_n)
\leq
3^{\gamma(m-n)}\bfE \qquad \forall z \in \mathcal{Z}_n \cap \cudot_m \,.
\end{equation}
\end{enumerate}

The inequality in~\eqref{e.ellipticity} is in the sense of partial ordering of matrices, namely that for~$A,B\in\mathbb{R}^{d\times d}_{\mathrm{sym}}$ we write~$A\leq B$ when~$B-A$ has nonnegative eigenvalues. The only difference between~\ref{a.ellipticity.dagger} and~\ref{a.ellipticity} is that we have replaced the last line with~\eqref{e.ellipticity}. This is equivalent up to a factor of 2 because
\begin{equation}
\label{e.E0.blocks.def}
\mathbf{E}_0 =
\begin{pmatrix}
\mathbf{E}_{11} & \mathbf{E}_{12} \\
\mathbf{E}_{21} & \mathbf{E}_{22}
\end{pmatrix}
\quad \mbox{and} \quad
\left\{
\begin{aligned}
&\s_{*,0} := \mathbf{E}_{22}^{-1} \,, \\
&\k_0:= -\mathbf{E}_{22}^{-1}\mathbf{E}_{21} \,,\\
&\s_0 := \mathbf{E}_{11} - \mathbf{E}_{12}\mathbf{E}_{22}^{-1}\mathbf{E}_{21}\,, \\
&\b_0 := \mathbf{E}_{11}\,,
\end{aligned}
\right.
\end{equation}
implies that
\begin{align*}
\mathbf{E}_0  =
\begin{pmatrix}
\s_0 + \k_0^t\s_{*,0}^{-1}\k_0 & -\k_0^t\s_{*,0}^{-1} \\
- \s_{*,0}^{-1}\k_0 & \s_{*,0}^{-1}
\end{pmatrix}
 \leq 2 \begin{pmatrix}
\s_0 + \k_0^t\s_{*,0}^{-1}\k_0 & 0\\
0 & \s_{*,0}^{-1}
\end{pmatrix}\,.
\end{align*}
Therefore~\ref{a.ellipticity.dagger} implies~\ref{a.ellipticity} with constants
\begin{equation*}
\Lambda_0 = 2|\s_0 + \k_0^t\s_{*,0}^{-1}\k_0|\,, \quad \mbox{and} \quad \lambda_0 = |2\s_{*,0}^{-1}|^{-1}\,,
\end{equation*}
while conversely given~\ref{a.ellipticity} we may take
\begin{equation*}
\mathbf{E}_0 = \begin{pmatrix}
\Lambda_0 & 0 \\ 0 & \lambda_0^{-1}
\end{pmatrix}\,.
\end{equation*}

The reason we use~\ref{a.ellipticity.dagger} is that it is natural to take~$\mathbf{E}_0 = \mathbb{E}[\bfA(\cudot_n)]$ at some scale~$n$ and renormalize the ellipticity assumption as in Lemma~\ref{l.renormalize.ellipticity}. We define the ellipticity ratio~$\Theta$ by
\begin{equation}
\label{e.Theta.def}
\Theta := \min_{\h\in\R^{d\times d}_{\mathrm{skew}}} \big| \s_{*,0}^{-\nicefrac12} (\s_0 + (\k_0-\h)^t\s_{*,0}^{-1}(\k_0-\h))\s_{*,0}^{-\nicefrac12} \big|\,.
\end{equation}
The subtraction of a constant skew-symmetric matrix reflects the invariance of divergence form equations under this transformation, as explored in this section. We denote by~$\h_0$ the minimizer in~\eqref{e.Theta.def} and define the ellipticity constants~$0<\lambda \leq \Lambda < \infty$ by
\begin{equation}
\label{e.lambdas.def}
\lambda_0 := \big|\s_{*,0}^{-1}\big|^{-1} \quad \mbox{and} \quad  \Lambda_0 := \min_{\h\in\R^{d\times d}_{\mathrm{skew}}} \big|\s_0 + (\k_0-\h)^t\s_{*,0}^{-1}(\k_0-\h)\big|\,,
\end{equation}
and the aspect ratio
\begin{equation}
\Pi_0 \coloneqq \frac{\Lambda_0}{\lambda_0}\,.
\end{equation}

Finally we state a purely algebraic lemma which will be useful later.
\begin{lemma}
\label{l.bfE.bounds}
Suppose~$\s_1,\s_{*,1}\in\R^{d\times d}_{\mathrm{sym}}$ are symmetric matrices,~$\k_1\in\R^{d\times d}$,
\begin{equation*}
\mathbf{E}_{1} :=
\begin{pmatrix}
\s_{1} + \k_1^t\s_{*,1}^{-1}\k_1 & -\k_1\s_{*,1}^{-1} \\
\s_{*,1}^{-1}\k_1 & \s_{*,1}^{-1}
\end{pmatrix}
\quad \mbox{and} \quad
\mathbf{E}_{*,1} :=
\begin{pmatrix}
\s_{*,1} + \k_1\s_1^{-1}\k_1^t & \k_1\s_1^{-1} \\
\s_1^{-1}\k_1^t & \s_1^{-1}
\end{pmatrix}\,,
\end{equation*}
and
\begin{equation*}
\mathbf{E}_{*,1} \leq \mathbf{E}_1\,.
\end{equation*}
Then for
\begin{equation*}
\tilde{\Theta}:= |\s_{*,1}^{-\nicefrac12}\s_1\s_{*,1}^{-\nicefrac12}|
\end{equation*}
we have
\begin{equation}
\label{e.symm.k.controlled}
|\s_{*,1}^{-\nicefrac12} (\k_1 + \k_1^t)\s_{*,1}^{-\nicefrac12} | \leq \tilde{\Theta} - 1\,,
\end{equation}
and
\begin{equation}
\label{e.bfA.E.diff}
|\mathbf{E}_{*,1}^{-\nicefrac12}\mathbf{E}_1\mathbf{E}_{*,1}^{-\nicefrac12} - \Itwod| \leq 6(\tilde{\Theta} - 1)\,.
\end{equation}
\end{lemma}
\begin{proof}
This is established in~\cite[Section 2.7]{AK.HC}.
\end{proof}

%In particular, the inequality~$\bfAhom_*(\cudot_n) \leq \bfAhom(\cudot_n)$ implies that
%\begin{equation*}
%|\shom_{*}(\cudot_n)^{-\nicefrac12}\bigl( \khom(\cudot_n) + \khom(\cudot_n)^t)\shom_{*}(\cudot_n)^{-\nicefrac12}| \leq \Theta_n - 1\,,
%\end{equation*}
%and
%\begin{equation*}
%\bfAhom(\cudot_n) - \bfAhom_*(\cudot_n) \leq (2+\Theta_n^{\nicefrac12})(\Theta_n - 1)\bfAhom_*(\cudot_n)\,,
%\end{equation*}
%where~$\Theta_n$ is defined below in~\eqref{e.Theta.n.def}.

\subsection{Stochastic bounds on the coarse-grained matrices}
\label{ss.stochastic}
It is a consequence of~\eqref{e.bfA.crude.moment}, subadditivity and the inequality (see~\cite[Lemma C1]{AK.HC}) 
\begin{equation}
\label{e.E.Opsi}
X \leq \O_{\Psi}(a) \implies \mathbb{E}[X^p] \leq 2pa^p K_{\Psi}^{1+ \lceil\frac{1}{2}p(p+1)\rceil} \quad \forall p\geq 1\,,
\end{equation}
that all finite moments of~$\bfA(I\times U)$ are bounded, for any finite interval~$I\subseteq \R$ and bounded Lipschitz domain~$U\subseteq \Rd$. We therefore define
\begin{equation}
\label{e.bfA.hom.def}
\bfAhom(I\times U) := \E[\bfA(I\times U)]\,.
\end{equation}
Similarly, we define~$\shom(I\times U),\shom_*(I\times U),\khom(I\times U)$ and~$\bhom(I\times U)$ as the deterministic matrices satisfying
\begin{equation}
\label{e.all.homs.def}
\left\{
\begin{aligned}
\shom_*(I\times U) & := \E \bigl[ \s_*^{-1}(I\times U) \bigr]^{-1} \,, 
\\ 
\khom(I\times U) 
& :=  \shom_*(I\times U)\E \bigl[ \s_*^{-1}(I\times U) \k(I\times U) \bigr] \,, 
\\ 
\bhom(I\times U) &:= \shom(I\times U) + \khom^t(I\times U) \shom_*^{-1}(I\times U) \khom(I\times U)
\\ &
=
\E \bigl[ \s(I\times U) + \k^t(I\times U) \s_*^{-1}(I\times U) \k(I\times U) \bigr]\,.
\end{aligned}
\right.
\end{equation}
As a consequence of these definitions,
\begin{equation}
\label{e.bfAhom.mat}
\bfAhom(I\times U) 
= 
\begin{pmatrix} 
( \shom + \khom^t\shom_*^{-1}\khom)(I\times U)
& - (\khom^t \shom_*^{-1})(I\times U)
\\ -(\shom_*^{-1}\khom)(I\times U)
& \shom_*^{-1}(I\times U)
\end{pmatrix}\,,
\end{equation}
and
\begin{equation}
\label{e.bfAhom.star.mat}
\bfAhom_*(I\times U):= \mathbb{E} [\bfA_*^{-1}(I\times U)]^{-1} = 
\begin{pmatrix}
(\shom_* + \khom\shom^{-1}\khom^t)(I\times U) & (\khom\shom^{-1})(I\times U) \\
(\shom^{-1}\khom^t)(I\times U) & \shom^{-1}(I\times U)
\end{pmatrix}
\,.
\end{equation}
Taking the expectation of~\eqref{e.Jsplitting} with~$Q = \bfAhom_*(I\times U)P$ for any~$P\in\R^{2d}$, we get
\begin{equation*}
0 \leq \mathbb{E}[\bfJ(I\times U,P,\bfAhom_*(I\times U)P)] = \frac{1}{2}P\cdot \bigl( \bfAhom(I\times U) - \bfAhom_*(I\times U) \bigr)P\,,
\end{equation*}
so that we have the ordering~$\bfAhom_*(I\times U) \leq \bfAhom(I\times U)$, and consequently~$\shom_*(I\times U) \leq \shom(I\times U)$.

\begin{lemma}[Stochastic bounds on the coarse-grained matrices]
\label{l.bfA.bounds}
Assume that~$\P$ satisfies~\ref{a.stationarity},\newline \ref{a.ellipticity}, and~\ref{a.CFS}. Then the following holds:
\begin{itemize}

\item Improving ellipticity on large mesoscales: for every~$h\in\N$ there exists a random scale~$\S_h$ satisfying
\begin{equation}
\label{e.Sh.O}
\S_h \leq \O_{\Psi_{\S}} \bigl ( K_{\Psi_{\S}}^{4(d+3)} 3^h \bigr)
\end{equation}
such that, for every~$m\in\Z$ and~$n\in\Z \cap (-\infty,m]$,
\begin{equation}
\label{e.ellipticity.mesogrid}
3^m \geq \S_h \implies \bfA(z+\cudot_n) \leq 3^{\gamma(m-h-n)_+}\bfE \qquad \forall z\in \mathcal{Z}_n \cap \cudot_m\,.
\end{equation}

\item Upper bounds for coarse-grained matrices: for every~$m\in\N, n\in\Z$ with~$n\leq m$ and~$z\in\mathcal{Z}_n \cap \cudot_m$
\begin{equation}
\label{e.bfA.crude.moment}
|\bfE^{-\nicefrac12}\bfA(z + \cudot_n)\bfE^{-\nicefrac12}| \leq 3^{\gamma(m-n)} (1+\O_{\Psi_{\S}}(3^{\gamma-m})) \,.
\end{equation}
In particular, for every~$n\in \Z$,
\begin{equation}
\label{e.bfA.crude.moment.E}
\bigl|\mathbf{E}_0^{-\nicefrac12}\bfA(\cudot_n) \mathbf{E}_0^{-\nicefrac12}\bigr| \leq 1+\O_{\Psi_{\S}}(3^{\gamma-m})\,.
\end{equation}

\item Upper and lower bounds on the means: for every~$n\in\N$
\begin{equation}
\label{e.bfAhom.by.E0}
(1+3^{3-n}K_{\Psi_\S}^2)^{-1} \bfAhom(\cudot_n) \leq \mathbf{E}_0 \leq 2\big(1 + 32(\Theta - 1)\big)\bfAhom(\cudot_n)\,.
\end{equation}
% We don't need the 3^{-n}K_S factor on the upper bound since we can take very large n and then use monotonicity of the \bfAhom

\item Sensitivity and locality of~$\bfA$: for any finite interval~$I\subseteq\R$ and bounded Lipschitz~$U\subseteq\Rd$,
\begin{equation}
\label{e.bfA.malliavin}
\big| D_U (P \cdot \bfA(I\times U)P)\big| \leq P \cdot \bfA(I\times U)P\,, \forall P\in\R^{2d}
\end{equation}
and
\begin{equation}
\label{e.bfA.local}
\bfA(I\times U) \quad \mbox{is $\mathcal{F}(I\times U)$--measurable.}
\end{equation}

\item Concentration for sums of~$\bfA$'s: for every~$k,m,n\in \N$ with~$\beta k < n < k \leq m$ and~$z\in\mathcal{Z}_k \cap \cudot_m$,
\begin{equation}
\label{e.bfA.CFS.upper}
\avsum_{z'\in\mathcal{Z}_n \cap (z+\cudot_k)} \bfE^{-\nicefrac12}\big(\bfA(z' + \cudot_n) - \bfAhom(\cudot_n)\big)\bfE^{-\nicefrac12} \indc_{\{\S \leq 3^m\}} \leq \O_{\Psi} \bigg( 4 \cdot 3^{\gamma(m-n)} 3^{-\nu(k-n)} \bigg)\,.
\end{equation}

\end{itemize}
\end{lemma}
\begin{proof}

The proofs are straightforward generalizations of the elliptic case, following~\cite[Section 2.8]{AK.HC}.
\end{proof}

\subsection{Renormalization of the ellipticity assumption}
\label{ss.renormalization}

As in the elliptic case, the assumption that~$\P$ satisfies~\ref{a.stationarity},~\ref{a.ellipticity.dagger} and~\ref{a.CFS} can be renormalized. 
To formalize this, we introduce the mapping~$D_{n_0}:\Omega \to \Omega$ given by dilation by~$3^{n_0}$,
\begin{equation}
\label{e.dilation.def}
(D_{n_0} \a )(t,x) = \a (3^{2n_0}t, 3^{n_0}x)
\end{equation}
and we define~$\P_{n_0}$ by
\begin{equation}
\label{e.Pn0}
\P_{n_0}:= \mbox{\,the pushforward of~$\P$ under~$D_{n_0}$.}
\end{equation}
The measure~$\P_{n_0}$ satisfies (almost) the same assumptions as~$\P$, but with the ellipticity matrix~$\mathbf{E}_0$ replaced by~$\bfAhom(\cudot_{n_0-l_0})$, where the scale separation~$l_0$ is sufficiently large enough. However, we expect the ellipticity ratio for~$\bfAhom(\cudot_{n_0-l_0})$ to be much smaller than for~$\bfE$. It is natural to define, for each~$n\in\N$, the renormalized ellipticity ratio~$\Theta_n \in [1,\infty)$ at scale~$3^n$, which is the ellipticity ratio for~$\bfAhom(\cudot_n)$. In view of~\eqref{e.Theta.def} and~\eqref{e.bigA.def}, we define it by 
\begin{equation}
\label{e.Theta.n.def}
\Theta_n := 
\min_{\h_0 \in \R^{d\times d}_{\mathrm{skew}}}
\bigl| ( \shom_{*}^{-\nicefrac12} \bhom_{\h_0}\,\shom_{*}^{-\nicefrac12}  ) (\cudot_n) \bigr|
\,.
\end{equation}
Note that~$n\mapsto \Theta_n$ is monotone decreasing, as a consequence of the subadditivity of~$\b$ and~$\s_*^{-1}$.  For convenience, we define an exponent~$\mu$, used throughout the rest of the paper, by
\begin{equation}
\label{e.def.mu}
\mu:= (\nu-\gamma)(1-\beta)\,.
\end{equation}

\begin{lemma}[Renormalization of the ellipticity]
\label{l.renormalize.ellipticity}
Let~$\gamma < \rho < 1$ and~$\delta > 0$. Suppose that~$l_0\in\N$ satisfies
\begin{equation*}
l_0 \geq \frac{1}{\rho - \gamma}\big( 1 + \frac{d+2}{\mu}\big)\big(9 + \log(\delta^{-1}\Theta)\big) + \frac{6}{\mu}\log K_{\Psi}\,.
\end{equation*}
Then for every~$n \in \N$ with~$n-l_0 \geq 2\log K_\Psi$, there exists a minimal scale~$\S'\geq \S$ satisfying
\begin{equation*}
\S' = \O_{\Psi_{\S'}}(3^n) \quad \mbox{with} \quad \Psi_{\S'}(t):= \frac{1}{2}\min\big\{\Psi_\S(3^nt),\Psi(t^\mu)\big\}\,,
\end{equation*}
such that for every~$m\in\N$ with~$m\geq n$ and every~$k\leq m$
\begin{equation*}
3^m \geq \S' \implies \sup_{z\in \mathcal{Z}_k \cap \cudot_m} \bfA(z+\cudot_k) \leq ( 1+ \delta 3^{\rho(m-k)}) \bfAhom(\cudot_{n-l_0})\,.
\end{equation*}
\end{lemma}
\begin{proof}
The proof is a straightforward generalization of the elliptic case in~\cite[Lemma 2.12]{AK.HC}, up to the factor of~$d+2$ instead of~$d$.
\end{proof}

\begin{proposition}[Renormalization of the assumptions]
\label{p.renormalization.P}
Suppose~$\P$ satisfies~\ref{a.stationarity},~\ref{a.ellipticity.dagger} and\newline \ref{a.CFS}.
Let~$\rho \in (\gamma, \min\{ \nu,1\})$ and~$\delta>0$. 
Suppose that~$l_0\in\N$ satisfies
\begin{equation}
\label{e.l0.condition}
l_0 \geq \frac{1}{\rho - \gamma}\big( 1 + \frac{d+2}{\mu}\big)\big(9 + \log(\delta^{-1}\Theta)\big) + \frac{6}{\mu}\log K_{\Psi}\,.
\end{equation}
For every~$n_0\in\N$ with~$n_0 \geq l_0 + 2\log K_\Psi$, the pushforward~$\P_{n_0}$ of~$\P$ under the dilation map given in~\eqref{e.dilation.def} satisfies the assumptions~\ref{a.stationarity},~\ref{a.ellipticity.dagger} and~\ref{a.CFS}, where the parameters~$(\gamma,\Psi_\S,\bfE)$ in assumption~\ref{a.ellipticity.dagger} are replaced by~$(\rho,\Psi_{\S'},(1+\delta) \bfAhom(\cu_{n_0-l_0}))$ and~$\Psi_{\S'}$ is defined by
\begin{equation}
\label{e.new.Psi.S}
\Psi_{\S'}(t):=\frac12 \min\bigl\{ \Psi_\S(3^{n_0}t) ,\Psi  ( t^{\mu })\bigr\}
\,.
\end{equation}
\end{proposition}
\begin{proof}
The conditions~\ref{a.stationarity} and~\ref{a.CFS} for~$\P_{n_0}$ are immediate from their validity for~$\P$, and~\ref{a.ellipticity.dagger} is checked in Lemma~\ref{l.renormalize.ellipticity}.
\end{proof}

The function~$\Psi_{\S'}$ satisfies~$t\Psi_{\S'}(t) \leq \Psi_{\S'}(K_{\Psi_{\S'}}t)$ for all~$t\geq 1$ with~$K_{\Psi_{\S'}}$ given by
\begin{equation}
\label{e.new.K.Psi.S}
K_{\Psi_{\S'}} := 
\max\big\{ K_{\Psi_{\S}} , K_{\Psi}^{\lceil \nicefrac1\mu \rceil} \bigr\}\,.
\end{equation}
This follows from the definition of~$\Psi_{\S'}$ in~\eqref{e.new.Psi.S} and~\cite[Appendix C]{AK.HC}. The new value of~$\Pi$ is at most~$256(1+\delta)^2\Pi$ by~\eqref{e.bfAhom.by.E0} and~$n_0-l_0\geq 2\log K_{\Psi}$, while the new value of~$\Theta$ is~$(1+\delta)^2\Theta_{n_0-l_0}\leq (1+\delta)^2\Theta$.

\subsection{Parabolic adapted geometry}
\label{ss.subadditivity}

The high-contrast homogenization proof requires the geometry to be adapted to the coefficient matrices, while maintaining parabolic scaling of the domains. We introduce the (metric) geometric mean of the matrices~$\b_0$ and~$\s_{*,0}$, denoted by
\begin{equation}
\label{e.m0.def}
\m_0 : = (\s_0 + \k_0^t\s_{*,0}^{-1}\k_0)\#\s_{*,0} \quad \mbox{and} \quad \mathbf{M}_0 =
\begin{pmatrix}
\m_0 & 0 \\ 0 & \m_0^{-1}
\end{pmatrix}\,,
\end{equation}
The definition of geometric mean is given in Appendix~\ref{aa.matrices}. We define
\begin{equation}
\label{e.m0.lambdas}
\lambda_{\m_0} := |\m_0^{-1}|^{-1}\,, \quad \Lambda_{\m_0} := |\m_0|\,, \quad \mbox{and} \quad \Pi_{\m_0} := \frac{\Lambda_{\m_0}}{\lambda_{\m_0}}\,.
\end{equation}
Note that the definition of~$\m_0$ is not invariant under the addition of a constant skew-symmetric matrix as considered in Section~\ref{ss.bfA.def}. We will however, make an appropriate centering assumption such that~$\m_0$ is the correct quantity, under which we will see that
\begin{equation*}
\Lambda_{\m_0} \leq \sqrt{8d}\Theta_m^{\nicefrac12}\Lambda\,,
\end{equation*}
while it is true under any centering that~$\lambda \leq \lambda_{\m_0}$.

\smallskip

We will work in domains adapted to~$\m_0$. For a large~$k_0\in\N$, to be selected below, define a matrix~$\q_0$ by
\begin{equation}
\label{e.q0.def}
(\q_0)_{ij} := 3^{-k_0}\lceil 3^{k_0}\lambda_{\m_0}^{-\nicefrac12}(\m_0^{\nicefrac12})_{ij}\rceil\,.
\end{equation}
Then every entry of~$\q_0$ belongs to~$3^{-k_0}\Z$,~$\q_0$ is symmetric, and
\begin{equation*}
|\q_0 -\lambda_{\m_0}^{-\nicefrac12}\m_0^{\nicefrac12}| \leq C(d)3^{-k_0}\,.
\end{equation*}
This implies that
\begin{equation*}
(1-C(d)3^{-k_0})\lambda_{\m_0}^{-\nicefrac12}\m_0^{\nicefrac12} \leq \q_0 \leq (1+ C(d)3^{-k_0})\lambda_{\m_0}^{-\nicefrac12}\m_0^{\nicefrac12}\,.
\end{equation*}
Choosing~$k_0$ sufficiently large, depending only on~$d$, we have
\begin{equation}
\label{e.q0.by.m0}
\frac{99}{100}\lambda^{-\nicefrac12}_{\m_0}\m_0^{\nicefrac12} \leq \q_0 \leq \frac{101}{100}\lambda^{-\nicefrac12}_{\m_0}\m_0^{\nicefrac12}\,,
\end{equation}
which implies that
\begin{equation}
\label{e.cus.in.cu}
\frac{99}{100}\cu_n \subseteq \q_0(\cu_n) \subseteq \frac{101}{100}\Pi_{\m_0}^{\nicefrac12}\cu_n\,.
\end{equation}
We round~$\lambda_{\m_0}$ up to
\begin{equation}
\label{e.adapted.cylinders.lambda.def}
\lambda_{r} : = \inf\{ 3^{2k_1} : k_1\in \mathbb{Z}\,, \lambda_{\m_0} \leq 3^{2k_1}\}\,,
\end{equation}
which is equivalent to~$\lambda_{\m_0}$ up to a factor of~$9$. As a consequence of the rounding, for the lattice defined by
\begin{equation}
\label{e.adapted.lattice.def}
\mathbb{L}_n : = 3^{2n}\mathbb{L}_{t} \times 3^n\mathbb{L}_{x} \quad \mbox{where} \quad \mathbb{L}_{t} := \lambda_{r}^{-1}\Z \,, \mathbb{L}_{x} := \q_0(\Zd)\,,
\end{equation}
we have~$\mathbb{L}_n \subseteq \Z^{d+1}$ when~$n\geq C\log (1+\lambda_{\m_0})$.
We introduce the adapted parabolic cubes 
\begin{equation}
\label{e.adapted.cyl.def}
\cusdot_n : =  J_n \times \cus_n\,, \quad \mbox{where} \quad  J_n := \bigg(-\frac{1}{\lambda_r}\frac{3^{2n}}{2},\frac{1}{\lambda_r} \frac{3^{2n}}{2}\bigg)\,, \quad \cus_n: = \q_0(\cu_n)\,.
\end{equation}
These are parallelepipeds in the spatial variable with the parabolic scaling in time, up to the rounding error in~\eqref{e.adapted.cylinders.lambda.def}. We again note that these domains are a function of the centring and will change throughout the paper. We will often use that for any~$n\in\mathbb{Z}$,
\begin{equation}
\label{e.adapted.scales.normal}
l \geq \log_3(9\max\{\Pi_{\m_0}^{\nicefrac12},\lambda_{\m_0}^{-\nicefrac12}\}) \implies \cusdot_n \subseteq \cudot_{n+l}\,,
\end{equation}
while
\begin{equation}
\label{e.normal.scales.adapted}
l \geq \max\{1, \log_3(9\lambda_{\m_0}^{\nicefrac12})\} \implies \cudot_n \subseteq \cusdot_{n+l}\,.
\end{equation}

\smallskip

We state here versions of the bounds on the coarse-grained matrices in adapted parabolic cubes. The lemmas in this section are generalizations of the elliptic case in~\cite[Section 2.10]{AK.HC}, but with parabolic geometry. We state the full proofs of these lemmas because they have an explicit ellipticity dependence which carries over into our main theorem on the homogenization length scale, and the ellipticity dependence (in particular the appearance of~$\lambda_{\m_0}$ as opposed to just~$\Pi_{\m_0}$) is parabolic in nature.

\begin{lemma}[Upper bounds for~$\bfA$ in adapted cylinders.]
\label{l.adapted.cylinders}
If~$\S_{h}$ is the random scale in Lemma~\ref{l.bfA.bounds}, we have for every~$n,m\in\N$ with~$n\leq m$, and every~$y\in \mathbb{L}_n$ such that~$y+\cusdot_n \subseteq \cusdot_m$
\begin{equation}
\label{e.adapted.cylinders}
3^m \geq \S_{h} \implies \bfA(y+\cusdot_n) \leq  \frac{C(d)}{1-\gamma}\max\{\Pi_{\m_0}^{\nicefrac{\gamma}{2}},\lambda_{\m_0}^{-\nicefrac{\gamma}{2}}\} \max\{1, \lambda_{\m_0}^{-\frac{1-\gamma}{2}}\} 3^{\gamma(m-h-n)_+}\mathbf{E}_0\,.
\end{equation}
\end{lemma}
\begin{proof}
Fix~$h\in \N$ and take~$m\in\N$ such that~$3^m \geq \S_{h}$, where~$\S_{h}$ is the minimal scale given by Lemma~\ref{l.bfA.bounds}. Choose~$l$ to be the smallest integer satisfying~\eqref{e.adapted.scales.normal} so that~$y+\cusdot_n \subseteq \cusdot_m \implies y+\cusdot_n \subseteq \cudot_{m+l}$. We will decompose~$y+\cusdot_n$ into the disjoint union (up to a null set) of families~$\{V_j(y):-\infty < j\leq n\}$ of sets such that each~$V_j(y)$ is the disjoint union of cubes~$z+\cudot_j$ for~$z\in\mathcal{Z}_j$, and apply Lemma~\ref{l.bfA.bounds} to each subcube.

Define first
\begin{equation*}
V_n(y): = \bigcup \big\{ z+ \cudot_n:z\in\mathcal{Z}_n\,, z+\cudot_n \subseteq y + \cusdot_n\big\}\,,
\end{equation*}
and then recursively,
\begin{equation*}
V_{j-1}(y) := \bigcup \big\{z+\cudot_{j-1}: z\in\mathcal{Z}_{j-1}\,, z+\cudot_{j-1} \subseteq (y+\cusdot_n)\setminus (V_n\cup \cdots \cup V_j)\big\}\,.
\end{equation*}
Recalling from~\eqref{e.adapted.cylinders.lambda.def} that~$\lambda_r = 3^{2k_1}$, the largest~$j$ such that~$V_j(y)$ is non-empty is~$j_{\max} \coloneqq n-1-\max\{k_1,0\}$. Our choice of rounded~$\lambda_r$ means that there will be no boundary layer in the time direction, because the size of the interval~$J_n$ is an integer multiple of~$3^{2j}$ for every~$j \leq j_{\max}$.

If~$x\in \cusdot_n \setminus (V_{j_{\max}}(y) \cup \cdots \cup V_{j+1}(y))$ then it is within distance~$C\sqrt{d}3^j$ of the spatial boundary, and therefore~$V_j(y)$ is contained in a volume bounded by this depth times the surface of the perpendicular surface of~$\cusdot_n$, summed over the faces of~$\cusdot_n$. We may then place an upper bound on the ratio~$\frac{|V_{j}|}{|\cusdot_n|}$ by
\begin{equation}
\label{e.cube.decomp.bound}
\frac{|V_{j}(y)|}{|\cusdot_n|} \leq Cd^{\nicefrac{3}{2}}3^{j-n} \quad \forall j < j_{\max}\,.
\end{equation}
By subadditivity, Lemma~\ref{l.bfA.bounds}, the above display, and
\begin{equation*}
(m+l-h-j)_+ \leq (m-h-n)_+ + (n+l-j)_+\,,
\end{equation*}
we have
\begin{align*}
\bfA(y+\cusdot_n) & \leq \sum_{j=-\infty}^n \frac{|V_j(y)|}{|\cusdot_n|}\bfA(V_j(y)) \\
& \leq \sum_{j=-\infty}^{j_{\max}} \frac{|V_j(y)|}{|\cusdot_n|}3^{\gamma(m+l-h-j)_+}\mathbf{E}_0 \\
& \leq \sum_{j=-\infty}^{n-1-\max\{k_1,0\}} C(d) 3^{j-n} 3^{\gamma(n+l-j)_+} 3^{\gamma(m-h-n)_+}\mathbf{E}_0 \\
& \leq C(d) 3^{\gamma l} 3^{-(1-\gamma)\max\{k_1,0\}} \sum_{j=-\infty}^{n-1-\max\{k_1,0\}} 3^{(1-\gamma)(j-n+\max\{k_1,0\})} 3^{\gamma(m-h-n)_+}\mathbf{E}_0 \\
& \leq \frac{C(d)}{1-\gamma} \max\{\Pi_{\m_0}^{\nicefrac{\gamma}{2}},\lambda_{\m_0}^{-\nicefrac{\gamma}{2}}\} \max\{1, \lambda_{\m_0}^{-\frac{1-\gamma}{2}}\} 3^{\gamma(m-h-n)_+}\mathbf{E}_0\,,
\end{align*}
which concludes the proof.
\end{proof}

\begin{lemma}[Concentration for adapted cylinders]
\label{l.CFS.adapted}
There exists a constant~$C(d) < \infty$ such that for every~$m,n\in\N$ with~$\beta m < n\leq m$,
\begin{align}
\lefteqn{
\biggl|  \avsum_{z\in\mathbb{L}_n \cap \cusdot_m} \mathbf{E}_0^{-\nicefrac12}\bigl(\bfA(z+\cusdot_n) - \bfAhom(z+\cusdot_n)\bigr)\mathbf{E}_0^{-\nicefrac12} \biggr|
} \qquad & \\
& \leq \frac{C d^{\nicefrac{3}{2}}K_{\Psi_{\S}}^{4d+14} }{1-\gamma} \max\{\Pi_{\m_0}^{\nicefrac{\gamma}{2}},\lambda_{\m_0}^{-\nicefrac{\gamma}{2}}\} \max\{1, \lambda_{\m_0}^{-\frac{1-\gamma}{2}}\}  3^{\gamma(m-n)-m} \\
& \quad + \O_{\Psi_{\S}}\biggl( \frac{C d^{\nicefrac{3}{2}}K_{\Psi_{\S}}^{4d+12}}{1-\gamma} \max\{\Pi_{\m_0}^{\nicefrac{\gamma}{2}},\lambda_{\m_0}^{-\nicefrac{\gamma}{2}}\} \max\{1, \lambda_{\m_0}^{-\frac{1-\gamma}{2}}\} 3^{\gamma(m-n)-m} \biggr) \\
&\quad  + \O_{\Psi}\biggl(C(d)K_{\Psi}^7 \max\{1,\lambda_{\m_0}\} \max\{\Pi_{\m_0}^{\frac{d+2}{2}},\lambda_{\m_0}^{-\frac{d+2}{2}}\} \max\{1, \lambda_{\m_0}^{-\frac{1-\gamma}{2}}\} 3^{-(\nu-\gamma)(m-n)} \biggr)\,.
\end{align}
\end{lemma}
\begin{proof}
Fix~$m,n\in\N$ such that~$\beta m < n \leq m$, and let~$n_0\in\N$ be the smallest integer such that~$\cusdot_0 \subseteq \cudot_{n_0}$; it follows that~$3^{n_0} \leq 3\max\{\Pi_{\m_0}^{\nicefrac12},\lambda_{\m_0}^{-\nicefrac12}\}$. We will prove concentration for adapted cylinders by grouping them into ordinary parabolic cylinders and applying~\ref{a.CFS} to those domains.

For each~$z\in\R^{d+1}$, let~$[z]$ denote the nearest point of the lattice~$\mathcal{Z}_{n+n_0}$ to~$z$, with lexicographical ordering used as a tiebreaker if this point is not unique. We have then that
\begin{equation*}
z+\cusdot_n \subseteq [z] + \cudot_{n+n_0+1}\,, \quad \forall z\in \R^{d+1}\,.
\end{equation*}
For any~$x\in\mathcal{Z}_{n+n_0}$, the set of~$z+\cusdot_n$ such that~$[z]=x$ is a disjoint union of cubes which is contained in~$x+\cudot_{n+n_0+1}$.

Then by dividing the volumes, there are at most~$C(d) 3^{(d+2)n_0}(1 + \lambda_{\m_0})$ points~$z\in \mathbb{L}_n$ such that~$[z]=x$. We can only apply~\ref{a.CFS} to bounded random variables, so select a smooth cutoff function~$\varphi:\R_+ \to [0,1]$ and for
\begin{equation}
T = \frac{C(d)}{1-\gamma}\max\{\Pi_{\m_0}^{\nicefrac{\gamma}{2}},\lambda_{\m_0}^{-\nicefrac{\gamma}{2}}\} \max\{1, \lambda_{\m_0}^{-\frac{1-\gamma}{2}}\} 3^{\gamma(m-n)}\,,
\end{equation}
which is the constant appearing on the right-hand side of~\eqref{e.adapted.cylinders}, define
\begin{equation*}
\indc_{[0,T]} \leq \varphi \leq \indc_{[0,2T]}\,, \quad |\varphi'| \leq 2T^{-1}\,,
\end{equation*}
and for each~$x\in\mathcal{Z}_{n+n_0}\cap \cudot_{m+n_0+1}$,
\begin{equation*}
X_x := \sum_{z\in\mathbb{L}_n \cap \cusdot_m,[z]=x} \varphi\big(|\mathbf{E}_0^{-\nicefrac12}\bfA(z+\cusdot_n)\mathbf{E}_0^{-\nicefrac12}|\big)\mathbf{E}_0^{-\nicefrac12}\bfA(z+\cusdot_n)\mathbf{E}_0^{-\nicefrac12}\,.
\end{equation*}
There are at most~$C(d)3^{(d+2)n_0}(1+\lambda_{\m_0}) \leq  C(d)\lambda_{\m_0}\max\{\Pi_{\m_0}^{\frac{d+2}{2}}, \lambda_{\m_0}^{-\frac{d+2}{2}}\}$ elements in the sum, so
\begin{equation*}
|X_x| \leq C(d)T (1+\lambda_{\m_0})\max \{\Pi_{\m_0}^{\frac{d+2}{2}}, \lambda_{\m_0}^{-\frac{d+2}{2}}\}\,.
\end{equation*}
We may now proceed exactly as in the elliptic case to conclude the proof: to briefly summarize, on the event~$\{\S_0 \leq 3^m\}$ we have~$\varphi\big(|\mathbf{E}_0^{-\nicefrac12}\bfA(z+\cusdot_n)\mathbf{E}_0^{-\nicefrac12}|\big) = 1$ and we can apply~\ref{a.CFS} between scales~$n+n_0+1$ and~$m+n_0+1$, and on the event~$\{ \S_0 > 3^m\}$ we use a more brutual bound using Lemma~\ref{l.adapted.cylinders}.

\end{proof}

\begin{lemma}[Means in adapted cylinders]
\label{l.adapted.means}
There exists a constant~$C(d)<\infty$ such that for all $y\in\mathbb{L}_n$ and~$k,n,m\in\N$ such that~$\cudot_k \subseteq \cusdot_n \subseteq \cudot_m$, and~$\mathbb{L}_n \subseteq \mathbb{Z}^{d+1}$,
\begin{equation}
\label{e.adapted.mean.by.Ahom}
\bfAhom(y+\cusdot_n) \leq \bfAhom(\cudot_k) + \frac{C(d) K_{\Psi_{\S}}^9}{1-\gamma}\max\{\Pi_{\m_0}^{\nicefrac{\gamma}{2}},\lambda_{\m_0}^{-\nicefrac{\gamma}{2}}\} 3^{-(1-\gamma)(n-k)}\mathbf{E}_0\,
\end{equation}
and
\begin{equation}
\label{e.Ahom.by.adapted.mean}
\bfAhom(\cudot_m) \leq \bfAhom(\cusdot_n) + \frac{C(d)K_{\Psi_{\S}}^9}{1-\gamma} \Pi_{\m_0}^{\frac{1-\gamma}{2}} 3^{-(1-\gamma)(m-n)}\mathbf{E}_0\,.
\end{equation}
\end{lemma}
\begin{proof}
Fix~$k,n\in\N$ with~$n-k$ satisfying~\eqref{e.normal.scales.adapted} so that~$\cudot_k \subseteq \cusdot_n$. Define the interior
\begin{equation*}
V: = \bigcup \big\{ z+\cudot_k : z\in\mathcal{Z}_k\,, z+\cudot_{k}\subseteq \cusdot_n\big\}\,.
\end{equation*}
Define recursively, for each~$j < k$,
\begin{equation}
V_j := \bigcup \big\{ z+\cudot_j : z\in\mathcal{Z}_j\,, z + \cudot_j \subseteq \cusdot_n \setminus (V \cup V_{k-1} \cup \cdot \cup V_{j+1})\big\}\,,
\end{equation}
and note that the estimate~\eqref{e.cube.decomp.bound} holds for every~$j < k$. Using subadditivity,
\begin{align}
\label{e.subadd.means}
\bfA(\cusdot_n) & \leq \frac{|V|}{|\cusdot_n|}\bfA(V) + \sum_{j=-\infty}^{k-1}\frac{|V_j|}{|\cusdot_n|}\bfA(V_j) \notag \\
& \leq \avsum_{z\in\mathcal{Z}_k,z+\cudot_k\subseteq V} \bfA(z+\cudot_k) + C\sum_{j=-\infty}^{k-1}\avsum_{\underset{z\in\mathcal{Z}_j}{z+\cudot_j\subseteq V_j}} d^{\nicefrac{3}{2}}3^{j-n} \bfA(z+\cudot_j)\,.
\end{align}
Let~$l\in\N$ be the minimum integer satisfying~\eqref{e.adapted.scales.normal} so that~$z+\cudot_j \subseteq \cusdot_n$ implies that~$z+\cudot_j \subseteq \cudot_{n+l}$. We will control the boundary layers using~\eqref{e.bfA.crude.moment} in the form
\begin{equation}
\label{e.take.E.bound}
\mathbb{E}[ | |\mathbf{E}_0^{-\nicefrac12}\bfA(z+\cudot_j)\mathbf{E}_0^{-\nicefrac12}| ] \leq C(d)K_{\Psi_{\s}}^7 3^{\gamma(n+l-j)}\,.
\end{equation}
From this it follows that
\begin{align*}
\lefteqn{
\sum_{j=-\infty}^{k-1} \avsum_{\underset{z\in\mathcal{Z}_j}{z+\cudot_j\subseteq V_j}} C(d) 3^{j-n} \mathbb{E}[|\mathbf{E}_0^{-\nicefrac12}\bfA(z+\cudot_j)\mathbf{E}_0^{-\nicefrac12}|]
}\qquad & \\
& \leq \sum_{j=-\infty}^{k-1} C(d)K_{\Psi_{\S}}^9 3^{j-n} 3^{\gamma(l+n-j)} \leq C(d)K_{\Psi_{\S}}^9 3^{\gamma l} 3^{-(1-\gamma)(n-k)} \sum_{j=-\infty}^{k-1} 3^{(1-\gamma)(j-k)} \\
& \leq \frac{C(d)K_{\Psi_{\S}}^7}{1-\gamma} \max\{\Pi_{\m_0}^{\nicefrac{\gamma}{2}},\lambda_{\m_0}^{-\nicefrac{\gamma}{2}}\}3^{-(1-\gamma)(n-k)}\,.
\end{align*}
Taking an expectation of~\eqref{e.subadd.means} and substituting in the above proves~\eqref{e.adapted.mean.by.Ahom}.

To get a bound in the opposite direction we need to partition~$\cudot_m$ into cubes of the form~$y'+\cusdot_n$ for~$y'\in\mathbb{L}_n$, plus a boundary layer. Define the interior
\begin{equation*}
W:= \bigcup\big\{ y' + \cusdot_n : y'\in\mathbb{L}_n\,, y'+\cusdot_n \subseteq \cudot_m\big\}\,,
\end{equation*}
and define recursively
\begin{equation*}
W_j := \bigcup \big\{ z + \cudot_j: z \in \mathcal{Z}_j\,, z + \cudot_j \subseteq \cudot_m \setminus (W \cup W_{j_{\max}} \cup \cdots \cup W_{j+1}) \big\}\,.
\end{equation*}
Here~$j_{\max}$ is the largest~$j$ such that~$W_j$ is non-empty. Since we only need to worry about the spatial direction this satisfies~$3^{j_{\max}} \leq \Pi_{\m_0}^{\nicefrac12}3^n$.

From the definitions each~$W_j$ is at least distance~$\sqrt{d}3^{j+1}$ from the spatial boundary of~$W \cup \partial \cudot_m$. The perimeter of~$W$ is bounded by a constant (depending only on~$d$) times the perimeter of~$\cudot_m$, so we have the bound
\begin{equation}
\frac{|W_j|}{|\cudot_m|} \leq C(d)3^{j-m}\,.
\end{equation}
Subadditivity then gives
\begin{align*}
\bfA(\cudot_m) & \leq \frac{|W|}{|\cudot_m|}\bfA(W) + \sum_{j=-\infty}^{n_0} \frac{|W_j|}{|\cudot_m|}\bfA(W_j) \\
& \leq \avsum_{\underset{y'+\cusdot_n\subseteq \cudot_m}{y'\in \mathbb{L}_n}}\bfA(y' + \cusdot_n) + \sum_{j=-\infty}^{n_0}\avsum_{\underset{z+\cudot_j\subseteq W_j}{z\in \mathcal{Z}_j}} C(d)3^{j-m} \bfA(z + \cudot_j)\,.
\end{align*}
Using again~\eqref{e.take.E.bound} but this time comparing scale~$j$ to scale~$m$
\begin{align*}
\lefteqn{
\sum_{j=-\infty}^{n_0} \avsum_{\underset{z+\cudot_j\subseteq W_j}{z\in \mathcal{Z}_j}} 3^{j-m} \mathbb{E}[|\mathbf{E}_0^{-\nicefrac12}\bfA(z + \cudot_j)\mathbf{E}_0^{-\nicefrac12}|]
} \qquad & \\
& \leq \sum_{j=-\infty}^{j_{\max}} C(d)K_{\Psi_{\S}}^9 3^{j-m}3^{\gamma(m-j)}\big) \leq \frac{C(d)K_{\Psi_{\S}}^9}{1-\gamma} 3^{-(1-\gamma)(m-j_{\max})} \\
& \leq \frac{C(d)K_{\Psi_{\S}}^9}{1-\gamma} 3^{-(1-\gamma)(m-n)} \Pi_{\m_0}^{\frac{1-\gamma}{2}}\,,
\end{align*}
concluding the proof as before.
\end{proof}

\subsection{Parabolic rescaling}
\label{ss.rescaling}

In this section we describe the natural rescaling of the parabolic problem and how it fits into the coarse-graining framework. First, we state a simple lemma describing the effect of a change of variables at the level of the coarse-grained matrices. Fix a positive-definite, symmetric matrix~$\q_0$ and constant~$\lambda_r$ and define the adapted parabolic cylinders as in Section~\ref{ss.subadditivity}. If~$\partial_t u = \nabla \cdot \a\nabla u$ in~$\cusdot_n$ then
\begin{equation}
\label{e.transform.eq}
\left.
\begin{aligned}
& \tilde{u}(s,y) \coloneqq u(\lambda_r^{-1} s, \q_0(y)) \\
& \tilde{\a}(s,y) \coloneqq (\lambda_r^{\nicefrac12}\q_0)^{-1}\a(\lambda_r^{-1} s, \q_0(y))(\lambda_r^{\nicefrac12}\q_0)^{-1} \\
\end{aligned}
\right\}
\implies \partial_s \tilde{u} - \nabla \cdot \tilde{\a} \nabla \tilde{u} = 0 \quad \mbox{ in } \cudot_n\,.
\end{equation}
The symmetric and skew-symmetric parts of~$\tilde{\a}$ are respectively
\begin{equation}
\tilde{\s}(y,s) \coloneqq (\lambda_r^{\nicefrac12}\q_0)^{-1}\s(\lambda_r^{-1} s, \q_0(y))(\lambda_r^{\nicefrac12}\q_0)^{-1}\,,
\end{equation}
and
\begin{equation}
\tilde{\k}(y,s) \coloneqq (\lambda_r^{\nicefrac12}\q_0)^{-1}\k(\lambda_r^{-1} s, \q_0(y))(\lambda_r^{\nicefrac12}\q_0)^{-1}\,.
\end{equation}
We then let~$J_{\a}(I\times U,p,q)$ be the quantity defined in~\eqref{e.J.def}, but with explicit reference to the coefficient field, and similarly for~$\b_{\a}(I\times U)$ and~$\s_{*,\a}(I\times U)$. The following lemma states this change of variables at the level of the coarse-grained matrices.

\begin{lemma}
\label{l.change.variables}
Suppose that~$z+\cusdot_k$ is the image of~$y + \cudot_k$ under the transformation~$(x,t) \mapsto (\q_0(x), \lambda_r^{-1}t)$. Then
\begin{equation}
J_{\tilde{\a}}(y+\cudot_k,p,q) = \frac{1}{\lambda_r}J_{\a}(z+\cusdot_k,\q_0^{-1}p, \lambda_r \q_0 q)
\end{equation}
and in particular
\begin{equation}
\b_{\tilde{\a}}(y+\cudot_k) = \lambda_r^{-1}\q_0^{-1} \b_{\a}(z+\cusdot_k)\q_0^{-1} \quad \mbox{and} \quad \s_{*,\tilde{\a}}^{-1}(\cudot_k) = \lambda_r \q_0 \s_{*,\a}^{-1}(\cusdot_k) \q_0\,.
\end{equation}
\end{lemma}
\begin{proof}
Identifying every solution~$u\in\mathcal{A}(z+\cusdot_k)$ with its transformation~$\tilde{u}$ as in~\eqref{e.transform.eq}
\begin{align*}
J_{\tilde{\a}}(\cudot_k,p,q) & = \sup_{\partial_s \tilde{u} = \nabla \cdot \tilde{\a}\nabla \tilde{u}} \fint_{\cudot_k} \biggl( -\frac{1}{2}\nabla \tilde{u}\cdot \tilde{\s}\nabla \tilde{u} - p \cdot \tilde{\a}\nabla \tilde{u} + q \cdot \nabla \tilde{u} \biggr) \\
& = \frac{1}{\lambda_r}\sup_{\partial_t u = \nabla \cdot \a \nabla u} \fint_{\cusdot_k} \biggl( -\frac{1}{2}\nabla u \cdot \s \nabla u - \q_0^{-1}p \cdot \a \nabla u + \lambda_r\q_0 q \cdot \nabla u \biggr) \\
& = \frac{1}{\lambda_r}J_{\a}(\cusdot_k, \q_0^{-1}p, \lambda_r\q_0 q)\,.
\end{align*}
It follows that for all~$p\in \mathbb{R}^d$,
\begin{equation}
\frac{1}{2}p \cdot \b_{\tilde{\a}}(\cudot_k)p = J_{\tilde{\a}}(\cudot_k,p,0) = \frac{1}{\lambda_r}J_{\a}(\cusdot_k,\q_0^{-1}p,0) = \frac{1}{2\lambda_r}p \cdot \q_0^{-1} \b_{\a}(\cusdot_k)\q_0^{-1} p\,,
\end{equation}
so that~$\b_{\tilde{\a}}(\cudot_k) = \lambda_r^{-1}\q_0^{-1} \b_{\a}(\cusdot_k)\q_0^{-1}$. By setting~$p=0$ we obtain $\s_{*,\tilde{\a}}^{-1}(\cudot_k) = \lambda_r \q_0 \s_{*,\a}^{-1}(\cusdot_k) \q_0$.
\end{proof}

For a simple application of Lemma~\ref{l.change.variables}, suppose that~$\a$ is a uniformly elliptic field satisfying~\eqref{e.uniform.ellipticity} with constants~$0<\lambda \leq \Lambda < \infty$, and with finite range of dependence~$L$ in space and~$T$ in time. This suggests that we work in the dimensionless variables
\begin{equation}
t' = \frac{\lambda t}{L^2} \quad \mbox{and} \quad x' = \frac{x}{L}\,,
\end{equation}
because
\begin{equation}
\left.
\begin{aligned}
\partial_t \tilde{u} = \nabla \cdot \tilde{\a} \nabla \tilde{u} \quad \mbox{ in } \quad \cudot_0 \\
\tilde{u}(t',x') = u(\nicefrac{L^2 t'}{\lambda}, Lx') \\
\tilde{\a}(t',x') = \lambda^{-1}\a(\nicefrac{L^2 t'}{\lambda}, Lx') \\
\end{aligned}
\right\}
\implies
\partial_t u = \nabla \cdot \a \nabla u \mbox{ in } \left(-\frac{L^2}{2\lambda},\frac{L^2}{2\lambda}\right) \times \left(-\frac{L}{2},\frac{L}{2}\right)^d\,.
\end{equation}
The new coefficient field~$\tilde{\a}$ has uniform ellipticity lower bound~$1$, upper ellipticity bound~$\Lambda/\lambda$, range of dependence 1 in space, and range of dependence $\lambda T/L^2$ in time. This reduces the problem to the two dimensionless parameters~$\Lambda/\lambda$ and~$\lambda T / L^2$, and Lemma~\ref{l.change.variables} states that we can recover the coarse-grained matrices for~$\a$ in diffusively scaled domains from the coarse-grained matrices for~$\tilde{\a}$ in domains of the form~$z+\cudot_n$ with~$z\in\mathcal{Z}_n$. We state formally in the next proposition the constants for which the coefficient field~$\tilde{\a}$ satisfies the assumptions~\ref{a.ellipticity} and~\ref{a.CFS}. It follows that our main theorems, such as Theorem~\ref{t.theta.rate}, apply to~$\tilde{\a}$ with these parameters, and apply to~$\a$ after a change of variables.

\begin{proposition}
\label{p.rescaling}
Suppose that~$\a$ is a coefficient field with law~$\mathbb{P}$ satisfying~\ref{a.stationarity}, the uniform ellipticity condition~\eqref{e.uniform.ellipticity} with constants~$0 < \lambda \leq \Lambda < \infty$, and with finite range of dependence~$L$ in time and~$T$ in space: that is, given Borel subsets $U,V \subset \mathbb{R}^d$ and~$I,J\subset \mathbb{R}$
\begin{equation*}
\dist(U,V) \geq L \quad \mbox{or} \quad \dist(I,J) \geq T \implies \mathcal{F}(I\times U) \mbox{ and } \mathcal{F}(J\times V) \mbox{ are } \mathbb{P}\mbox{-independent}.
\end{equation*}
Assume without loss of generality that~$L,T\in\mathbb{N}$ and there exists~$k\in\mathbb{Z}$ such that~$\lambda = 3^{2k}$. If~$\tilde{\a}(t,x) = \lambda^{-1}\a(\nicefrac{L^2 t}{\lambda}, Lx)$ then for any~$n_0 \in \mathbb{N}$ satisfying
\begin{equation}
\label{e.first.scale.sep}
n_0 \geq \frac{1}{2}\log_3\left(\frac{3\lambda T}{L^2}\right)
\end{equation}
the pushforward measure~$\P_{n_0}$ of~$\tilde{\a}$ satisfies~\ref{a.stationarity}, satisfies the uniform ellipticity condition~\eqref{e.uniform.ellipticity} with lower ellipticity constant~$1$ and upper ellipticity constant~$\Lambda/\lambda$, and satisfies~\ref{a.CFS} with parameters~$\beta = 0$, $\nu = \frac{d+2}{2}$, and constant~$K_{\Psi}$ independent of~$L,T,\lambda$ and~$\Lambda$.
\end{proposition}
\begin{proof}
Suppose that~$\a$ is a coefficient field with law~$\mathbb{P}$ satisfying~\ref{a.stationarity}, the uniform ellipticity condition~\eqref{e.uniform.ellipticity} with constants~$0 < \lambda \leq \Lambda < \infty$, and with finite range of dependence~$L$ in time and~$T$ in space. Define~$\tilde{\a}(t,x) = \lambda^{-1}\a(\nicefrac{L^2t}{\lambda},Lx)$. Since~$\lambda = 3^{2k}$ for some~$k\in\mathbb{Z}$, if~$n_0\in\mathbb{N}$ satisfies~\eqref{e.first.scale.sep} then~$3^{2n_0}\frac{L^2}{\lambda}$ is an integer and it follows from this and~$\mathbb{Z}\times \mathbb{Z}^d$ stationarity of~$\a$ that~$D_{n_0}\tilde{\a}$ is~$\mathbb{Z}\times \mathbb{Z}^d$ stationary. Similarly, the uniform ellipticity bound for~$D_{n_0}\tilde{\a}$ follows immediately from the ellipticity bounds for~$\a$.

Choosing~$n_0$ to satisfy~\eqref{e.first.scale.sep} implies that the field~$D_{n_0}\tilde{\a}$ has space-time range of dependence 1. It is proved in~\cite[Section 3.2]{AK.Book} that a time-independent field with range of dependence 1 satisfies~\ref{a.CFS} with parameters~$\beta = 0$,~$\nu = \nicefrac{d}{2}$ and function~$\Psi(\cdot) = \Gamma_2(c\cdot)$, where $\Gamma_2(t) = e^{\nicefrac{t^2}{2}} - 1$. In the time-dependent case the exact same proof applies, up to the averaging factor given by the number of cubes in the family~$\{X_z : z\in\mathcal{Z}_n \cap \cudot_m\}$. Tracking this constant through the proof (effectively replacing~$d$ with~$d+2$) in~\cite[Section 3.2.1]{AK.Book} we obtain that~$D_{n_0}\tilde{\a}$ satisfies~\ref{a.CFS} with the stated parameters.
\end{proof}

\subsection{Function spaces}
\label{ss.besov}
For each~$s\in (0,1), p\in [1,\infty), q\in [1,\infty)$ and~$n\in\N$, we define a volume-normalized Besov seminorm in the parabolic cube~$\cudot_n$
\begin{equation}
\label{e.Besov.def}
[g]_{\underline{B}^s_{p,q}(\cudot_n)} := \biggl( \sum_{k=-\infty}^n \bigl( 3^{-spk}\avsum_{z\in \mathcal{Z}_{k-1}, z+\cudot_k \subseteq \cudot_n} \norm{g-(g)_{z+\cudot_k}}_{\L^p(z+\cudot_k)}^p \bigr)^{\nicefrac{q}{p}} \biggr)^{\nicefrac{1}{q}}\,.
\end{equation}
For every~$z\in\mathcal{Z}_{k-1}$ we integrate over the parabolic cube~$z+\cudot_{k}$, so each cube will overlap with~$3^{d+1}$ neighbouring cubes. This allows the semi-norm to detect discontinuity across the cubes, which would otherwise be an artefact of the cube decomposition. If~$s\in [0,1],p\in [1,\infty)$ then we define the~$q=\infty$ Besov seminorm by
\begin{equation}
\label{e.Besov.infty.def}
[g]_{\underline{B}^s_{p,\infty}(\cudot_n)} := \sup_{k\in (-\infty,n]\cap\Z} 3^{-sk} \biggl( \avsum_{z\in\mathcal{Z}_{k-1}, z+\cudot_k \subseteq \cudot_n} \norm{g-(g)_{z+\cudot_k}}_{\L^p(z+\cudot_k)}^p \biggr)^{\nicefrac{1}{p}}\,.
\end{equation}
The corresponding Besov norms are defined by
\begin{equation}
\label{e.Besov.norm.def}
\norm{g}_{\underline{B}_{p,q}^s(\cudot_n)}: = 3^{-sn}\|g\|_{\L^p(\cudot_n)} + [g]_{\underline{B}^s_{p,q}(\cudot_n)}\,,
\end{equation}
and the Banach space~$B_{p,q}^s(\cudot_n)$ is defined to be the closure of~$C^\infty(\cudot_n)$ with respect to~$\|\cdot\|_{\underline{B}_{p,q}^s(\cudot_n)}$. We use the Besov terminology because the three parameters~$p,q$ and~$s$ are respectively an integrability parameter, a scale parameter, and a regularity parameter. In the case~$q = p \in [1,\infty)$ and~$s\in (0,1)$ we have by Proposition~\ref{p.Besov.equiv}
\begin{equation*}
B_{p,p}^s(\cudot_n) = L^p(I_n; W^{s,p}(\cu_n)) \cap W^{\nicefrac{s}{2},p}(I_n;L^p(\cu_n))\,,
\end{equation*}
with an equivalence of norms. In particular, in the case~$p=2$ we obtain the spaces~$H^{s,\nicefrac{s}{2}}(\cudot_n)$ as defined, for example, in~\cite[Chapter 4, Section 2]{LMV2}. Another similar approach to defining Besov norms on finite domains can be found in~\cite[Section 1.10.3]{Triebel}. We also note that the semi-norm in~\eqref{e.Besov.def} is equivalent, for~$q=p\in[1,\infty)$ and~$s\in (0,1)$, to the integral
\begin{equation}
\label{e.fractional.integral}
\biggl( \fint_{\cudot_n} \int_{\cudot_n} \frac{|g(t,x) - g(s,y)|^p}{( |x-y| + |t-s|^{\nicefrac12} )^{d+2+sp} } \biggr)^{\nicefrac{1}{p}}\,,
\end{equation}
which is obtained by taking~\cite[Lemma A.4]{AK.HC} and replacing the partition of unity with a space-time, parabolically scaled partition of unity.

For~$s\in (0,1]$,~$p\in [1,\infty)$,~$q\in [1,\infty]$, and~$p',q'$ denoting the respective H\"older conjugates, define 
\begin{align}
\|f\|_{\underline{\hat{B}}_{p,q}^{-s}(\cudot_n)} := \sup \bigg\{ \fint_{\cudot_n} fg : g\in C^\infty(\cudot_n)\,, \|g\|_{\underline{B}_{p',q'}^s(\cudot_n)} \leq 1 \bigg\} \label{e.Besov.dual.hat.def}\,, \\
\|f\|_{\underline{B}_{p,q}^{-s}(\cudot_n)} := \sup \bigg\{ \fint_{\cudot_n} fg : g\in C_c^\infty(\cudot_n) \,, \|g\|_{\underline{B}_{p',q'}^s(\cudot_n)} \leq 1 \bigg\} \label{e.Besov.dual.def}\,,
\end{align}
and by Lemma~\ref{l.bound.dual.seminorm},
\begin{equation}
[f]_{\underline{B}_{p,q}^{-s}(\cudot_n)}  \leq [f]_{\mathring{\underline{B}}^{-s}_{p,q}(\cudot_n)} : = 3^{d+2+s} \biggl( \sum_{k=-\infty}^n \bigl( 3^{spk}\avsum_{z\in\mathcal{Z}_{k-1}, z+\cudot_k \subseteq \cudot_n} |(f)_{z+\cudot_k}|^p \bigr)^{\nicefrac{q}{p}}\biggr)^{\nicefrac{1}{q}} \label{e.Besov.dot.def}\,.
\end{equation}
These spaces appear naturally, for example in the parabolic multiscale Poincar\'e inequality (Lemma \ref{l.multiscale.poincare}), which states that if~$\partial_t u = \nabla \cdot \g$ in~$\cudot_{n+1}$ then
\begin{equation*}
\| u- (u)_{\cudot_n}\|_{\L^2(\cudot_n)} \leq C(d)\bigl([\nabla u]_{\Bring_{2,1}^{-1}(\cudot_{n+1})} + [\g]_{\Bring_{2,1}^{-1}(\cudot_{n+1})}\bigr)\,.
\end{equation*}
Since spatial averages of solutions are controlled by the coarse-graining inequalities of Section~\ref{ss.bfA.def} we will obtain good control of solutions in these spaces.

The coarse-grained ellipticity constants represent the effective diffusivity at a given scale. Similarly to~\eqref{e.lambda.infty.def}, we define, for~$n,m\in\mathbb{Z}$,~$s\in [0,1]$ and~$q \in [1,\infty)$ such that~$n\leq m$, the quantities
\begin{equation}
\label{e.coarse.grained.ellipticity}
\left\{
\begin{aligned}
& \Lambda_{s,q}(\cudot_n) \coloneqq \biggl(\sum_{k=-\infty}^{n} 3^{sq(k-n)} \max_{z\in\mathcal{Z}_{k} \cap \cudot_n} |\b(z+\cudot_k)|^{\nicefrac{q}{2}} \biggr)^{\nicefrac{2}{q}}\,,\\
& \lambda_{s,q}(\cudot_n) \coloneqq \biggl(\sum_{k=-\infty}^n 3^{sq(k-n)} \max_{z\in\mathcal{Z}_{k} \cap \cudot_n} |\s_*^{-1}(z+\cudot_k)|^{\nicefrac{q}{2}} \biggr)^{-\nicefrac{2}{q}}\,.
\end{aligned}
\right.
\end{equation}
The coarse-grained ellipticity assumption~\ref{a.ellipticity.dagger} implies finiteness of the coarse-grained ellipticity constants for~$s > \nicefrac{\gamma}{2}$ because
\begin{align}
\label{e.finite.E}
\sup_{z\in\mathcal{Z}_{k-1} \cap \cudot_n} |\b(z+\cudot_k)| \leq \biggl( \frac{\S \vee 3^n}{3^{k}}\biggr)^{\gamma}|\b_0| \qquad \mbox{and} \qquad \sup_{z\in\mathcal{Z}_{k-1} \cap \cudot_n} |\s_*^{-1}(z+\cudot_k)| \leq \biggl( \frac{\S \vee 3^n}{3^k}\biggr)^{\gamma} |\s_{*,0}^{-1}|\,,
\end{align}
which implies that
\begin{align}
\label{e.bounded.cg}
\Lambda_{s,q}(\cudot_n) \leq C(2s-\gamma,q) \biggl( \frac{\S \vee 3^n}{3^n}\biggr)^{\gamma}|\mathbf{b}_0|\,, \quad \lambda_{s,q}^{-1}(\cudot_n) \leq C(2s-\gamma,q) \biggl( \frac{\S \vee 3^n}{3^n}\biggr)^{\gamma}|\mathbf{s}_{*,0}^{-1}|
\end{align}

%We will use the coarse-grained ellipticity constants defined in adapted parabolic cylinders, in which case the transformation in~\eqref{e.transformers} means we should consider
%\begin{equation}
%\label{e.adapt.cg.ellipticity}
%\left\{
%\begin{aligned}
%& \Lambda_{s,\infty}(\cusdot_n) := \sup_{k\in\Z, k\leq n} 3^{2s(k-m)} \max_{z\in\mathbb{L}_{k} \cap \cusdot_n} |\m_0^{-\nicefrac12}\b(z+\cudot_k)\m_0^{-\nicefrac12}| \,, \\
%& \lambda_{s,\infty}^{-1}(\cusdot_n) := \sup_{k\in\Z, k\leq n} 3^{2s(k-m)} \max_{z\in\mathbb{L}_{k} \cap \cusdot_n} |\m_0^{\nicefrac12}\s_*^{-1}(z+\cudot_k)\m_0^{\nicefrac12}| \,, \\
%& \Lambda_{s,q}(\cusdot_n) : = \biggl(\sum_{k=-\infty}^{n} 3^{sq(k-n)} \max_{z\in\mathbb{L}_{k} \cap \cusdot_n} |\m_0^{-\nicefrac12}\b(z+\cusdot_k)\m_0^{-\nicefrac12}|^{\nicefrac{q}{2}} \biggr)^{\nicefrac{1}{q}}\,,\\
%& \lambda_{s,q}^{-1}(\cusdot_n) : = \biggl(\sum_{k=-\infty}^n 3^{sq(k-n)} \max_{z\in\mathbb{L}_{k} \cap \cusdot_n} |\m_0^{\nicefrac12}\s_*^{-1}(z+\cudot_k)\m_0^{\nicefrac12}|^{\nicefrac{q}{2}} \biggr)^{\nicefrac{1}{q}}\,.
%\end{aligned}
%\right.
%\end{equation}

We next state some functional inequalities which we will use repeatedly throughout the paper.

\begin{lemma}
\label{l.mathringB.bounds}
If~$s\in [0,1]$ and~$u\in\mathcal{A}(\cudot_n)$ then
\begin{align}
\label{e.mathring.B}
\left\{
\begin{aligned}
& [\nabla u ]_{\Bring^{-s}_{2,1}(\cudot_{n})}  \leq C(d)3^{sn}\lambda_{s,1}^{-\nicefrac12}(\cudot_{n}) \|\s^{\nicefrac12}\nabla u\|_{\L^2(\cudot_{n})}\,, \\
& [ \a\nabla u ]_{\Bring^{-s}_{2,1}(\cudot_{n})}  \leq C(d) 3^{sn} \Lambda_{s,1}^{\nicefrac12}(\cudot_{n}) \|\s^{\nicefrac12}\nabla u\|_{\L^2(\cudot_{n})}
\end{aligned}
\right.
\end{align}
\end{lemma}
\begin{proof}
We obtain~\eqref{e.mathring.B} as in~\cite[Lemma 2.2]{AK.HC}, using the parabolic coarse-graining inequalities~\eqref{e.energymaps} and~\eqref{e.energymaps.flux}.
\end{proof}

\begin{lemma}[Coarse-grained Poincar\'e inequality]
\label{l.Besov.regularity.solns}
For every~$n\in\mathbb{Z}, u\in\A(\cudot_{n+1})$ and~$s\in [0,1]$
\begin{align}
\label{e.cg.poincare}
\|u -(u)_{\cudot_n}\|_{\underline{B}_{2,\infty}^s(\cudot_n)} & \leq C(d) \bigl( [\nabla u]_{\Bring_{2,1}^{s-1}(\cudot_{n+1})} + [\a \nabla u]_{\Bring_{2,1}^{s-1}(\cudot_{n+1})} \bigr) \nonumber \\
& \leq C(d)3^{(1-s)n}\bigl( \Lambda_{s,1}^{\nicefrac12}(\cudot_{n+1}) + \lambda_{s,1}^{-\nicefrac12}(\cudot_{n+1}) \bigr)\norm{\s^{\nicefrac12}\nabla u}_{\L^2(\cudot_{n+1})}\,.
\end{align}
\end{lemma}
\begin{proof}
The proof of the first inequality is exactly as in \cite[Lemma 2.3]{AK.HC}, substituting in our parabolic multiscale Poincar\'e inequality from Lemma~\ref{l.multiscale.poincare} with~$\g=\a\nabla u$. The second inequality then follows directly from Lemma~\ref{l.mathringB.bounds}.
\end{proof}

Our next lemma uses approximation to pass to a limit provided that certain Besov norms are finite. By Lemma~\ref{l.mathringB.bounds} this follows from finiteness of the coarse-grained ellipticity constants. Since~$\gamma< 1$ we may take~$s = \frac{1+\gamma}{4} \in ( \nicefrac{\gamma}{2},\nicefrac12)$ and note that the conditions of Lemma~\ref{l.testing.with.u} are satisfied by our remark below~\eqref{e.coarse.grained.ellipticity}, since the random minimal scale~$\S$ is almost surely finite.
\begin{lemma}
\label{l.testing.with.u}
Let~$n\in\Z, s\in (0,1), \epsilon \in (0,1-s)$ and suppose~$u\in\A(\cudot_{n+1})$ such that
\begin{equation*}
[u]_{\underline{B}^{s+\epsilon}_{2,\infty}(\cudot_n)} + [\a\nabla u]_{\underline{B}_{2,1}^{-s}(\cudot_n)} < \infty\,.
\end{equation*}
Then for every~$\varphi\in C_c^\infty(\cudot_n)$,
\begin{equation}
\label{e.integrate.by.parts}
\fint_{\cudot_n} \varphi \nabla u \cdot \s \nabla u + \fint_{\cudot_n} u\nabla \varphi \cdot \a \nabla u = \frac{1}{2}\fint_{\cudot_n} u^2 \partial_t \varphi\,.
\end{equation}
\end{lemma}
\begin{proof}
Assume~$u$ is as in the statement and without loss of generality assume that~$(u)_{\cudot_n}=0$. For~$k\in \N, k \geq 10$, let~$u_k:= (u\wedge k) \vee (-k)$ and fix any~$\varphi \in C_c^\infty(\cudot_n)$. Then since~$u_k\varphi \in W^1_{\s}(\cudot_n)$ we can test the equation for~$u$ to obtain
\begin{equation*}
\fint_{\cudot_n} \varphi \nabla u_k \cdot \a \nabla u + \fint_{\cudot_n} u_k\nabla \varphi \cdot \a \nabla u = -\fint_{\cudot_n}\varphi u_k \partial_t u\,.
\end{equation*}
By the same proof as in~\cite[Lemma 2.4]{AK.HC}, using also Lemma~\ref{l.multiply.in.besov.norm}, the terms on the left-hand side converge as~$k\to\infty$ to the respective terms with~$u$ instead of~$u_k$. For the term on the right we use that~$u\partial_t u_k = u_k \partial_t u_k$ so
\begin{align*}
-\fint_{\cudot_n} \varphi u_k \partial_t u & = \fint_{\cusdot_n} u_k u \partial_t \varphi + \fint_{\cudot_n}  \varphi u \partial_t u_k = \fint_{\cudot_n} u_k u \partial_t \varphi + \fint_{\cudot_n} \varphi u_k \partial_t u_k \\
& = \fint_{\cudot_n} u_k u \partial_t \varphi + \frac{1}{2}\fint_{\cudot_n} \varphi \partial_t (u_k^2)  = \fint_{\cudot_n} u_k u \partial_t \varphi - \frac{1}{2}\fint_{\cudot_n} u_k^2 \partial_t \varphi\,.
\end{align*}
Since~$\partial_t \varphi \in L^\infty(\cudot_n)$ and~$u_k \to u$ in~$L^2(\cudot_n)$ we can send~$k\to\infty$ and replace~$u_k$ with~$u$ in the last expression.
\end{proof}

\smallskip

All of the functional inequalities and definitions in this section can be transformed to the adapted cubes defined in Section~\ref{ss.subadditivity} by applying the transformation~\ref{e.transformers}. We make all of the analogous definitions with the natural substitutions~$\cudot_n \to \cusdot_n$ and~$\mathcal{Z}_n \to \mathbb{L}_n$. For example, the coarse-grained Poincar\'e inequality in adapted cubes states that for every~$n\in\mathbb{Z}, u\in\A(\cusdot_{n})$ and~$s\in [0,1]$
\begin{align}
\label{e.adapted.bound}
\left\{
\begin{aligned}
& \| \q_0 \nabla u\|_{\underline{\hat{B}}^{-s}_{2,1}(\cusdot_n)} \leq C(d)3^{sn} \lambda_{s,1}^{-\nicefrac12}(\cusdot_{n}) \lambda_{\m_0}^{-\nicefrac12} \norm{\s^{\nicefrac12}\nabla u}_{\L^2(\cusdot_{n})}\,, \\
& \|\q_0^{-1} \a\nabla u\|_{\underline{\hat{B}}^{-s}_{2,1}(\cusdot_n)} \leq C(d)3^{sn} \Lambda_{s,1}^{\nicefrac12}(\cusdot_{n}) \lambda_{\m_0}^{-\nicefrac12} \norm{\s^{\nicefrac12}\nabla u}_{\L^2(\cusdot_{n})}\,,
\end{aligned}\right.
\end{align}
with
\begin{equation}
\label{e.adapted.Lambda}
\Lambda_{s,q}(\cusdot_n) \coloneqq \biggl(\sum_{k=-\infty}^{n} 3^{sq(k-n)} \max_{z\in\mathbb{L}_{k} \cap \cusdot_n} |\m_0^{-\nicefrac12}\b(z+\cusdot_k)\m_0^{-\nicefrac12}|^{\nicefrac{q}{2}} \biggr)^{\nicefrac{1}{q}}\,,
\end{equation}
and
\begin{equation}
\label{e.adapted.lambda}
\lambda_{s,q}(\cusdot_n) \coloneqq \biggl(\sum_{k=-\infty}^{n} 3^{sq(k-n)} \max_{z\in\mathbb{L}_{k} \cap \cusdot_n} |\m_0^{\nicefrac12}\s_*^{-1}(z+\cusdot_k)\m_0^{\nicefrac12}|^{\nicefrac{q}{2}} \biggr)^{-\nicefrac{1}{q}}\,.
\end{equation}

We have defined the coarse-grained ellipticity in the adapted cubes such that they are dimensionless constants. Finally, we note that the key lemma~\cite[Lemma 2.16]{AK.HC}, estimating the gradient and fluxes of solutions in negative regularity norms, holds with the obvious modifications, because the properties of the coarse-grained matrices established in Section~\ref{ss.bfA.def} are exactly analogous to those in the elliptic case.

\section{Renormalization in high contrast}
\label{s.hc}

\subsection{Renormalization strategy}

In Section~\ref{s.coarse.graining} we saw that the parabolic coarse-grained diffusion matrices can be defined in exact analogy to the elliptic coarse-grained matrices, and that the relevant coarse-graining inequalities, adapted geometry, and parabolic function spaces can be developed along the same lines. The consequence of this is that by inserting our parabolic machinery into the proof of high contrast elliptic homogenization we obtain a proof of high contrast parabolic homogenization. To be precise, in this section we estimate the length scale at which a space-time coefficient field~$\a(\cdot,\cdot)$, with possibly large ellipticity ratio~$\Theta$, has homogenized to a low contrast problem.

Recall that the parameters~$\Theta$,~$\lambda_0$ and~$\Lambda_0$ are given by the ellipticity assumption~\ref{a.ellipticity.dagger}, and satisfy~$1 \leq \Theta \leq \Lambda_0/\lambda_0$. At each scale~$3^m$ we have an analogous ellipticity, defined in~\eqref{e.Theta.n.def} by
\begin{equation*}
\Theta_m : = \min_{\h_0\in\mathbb{R}^{d\times d}_{\mathrm{skew}}} |(\shom_*^{-\nicefrac12}\b_{\h_0}\shom_*^{-\nicefrac12})(\cudot_m)|\,,
\end{equation*}
which is monotone decreasing in~$m$ and satisfies the crude bound
\begin{equation*}
\Theta_m \leq (1+3^{3-m}K_{\Psi_{\S}}^2)^2\Theta
\end{equation*}
by~\eqref{e.bfAhom.by.E0}. Our main theorem is a bound on~$\Theta_m$, which depends on the quantity
\begin{equation}
\Pi_{\mathrm{par}} \coloneqq \max\left\{\frac{\Lambda_0}{\lambda_0},\lambda_0^{-1},\lambda_0\right\}\,.
\end{equation}

\begin{theorem}
\label{t.highcontrast}
There exists a constant~$C(d)<\infty$ such that if~$\alpha  = (\mathrm{min}\{\nu,1\}-\gamma)(1-\beta), \sigma \in (0,\frac{1}{2}\Theta]$ and~$m\in \N$ satisfy
\begin{equation}
\label{e.highcontrast.scales}
m \geq \frac{C}{\alpha \sigma^2}\biggl( \log(K_{\Psi}\Pi_{\mathrm{par}}) + \frac{1}{\alpha}\log\bigl( \frac{K_{\Psi_{\S}}\Pi_{\mathrm{par}}}{\alpha \sigma}\bigr) \biggr)\log(1+\Theta)\,,
\end{equation}
then the renormalized ellipticity ratio satisfies
\begin{equation}
\label{e.make.Theta.small}
\Theta_m - 1 \leq \sigma\,.
\end{equation}
\end{theorem}
If we ignore all parameters in the above theorem except for ellipticity, taking, for example,~$\sigma = \nicefrac{1}{100}$, then the theorem states that
\begin{equation*}
m\geq C\log^2(1+\Pi_{\mathrm{par}}) \implies \Theta_m - 1 \leq \nicefrac{1}{100}\,.
\end{equation*}
In other words, by length scale~$\exp(C\log^2(1+\Pipar) )$ the problem has homogenized to a low contrast problem.

The proof of Theorem~\ref{t.highcontrast} is an iteration procedure. The main step is finding a length scale such that zooming out to that scale reduces the ellipticity by a constant factor.

\begin{proposition}
\label{p.renormalize}
There exists a constant~$C(d)<\infty$ such that if $\alpha = (\mathrm{min}\{\nu,1\}-\gamma)(1-\beta), \sigma \in (0,\nicefrac12]$ and~$m\in \N$ satisfy
\begin{equation}
\label{e.m.is.verylarge}
m \geq \frac{C}{\alpha \sigma^2}\biggl( \log(K_{\Psi}\Pi_{\mathrm{par}}) + \frac{1}{\alpha}\log\bigl( \frac{K_{\Psi_{\S}}\Pi_{\mathrm{par}}}{\alpha \sigma}\bigr) \biggr)\,,
\end{equation}
then we have either
\begin{equation}
\label{e.improve}
\Theta_m - 1 \leq \sigma \Theta_0 \quad \mbox{ or } \quad \bigl(\det\bfAhom(\cudot_{m})\bigr)^{\frac{1}{d}} \leq \sigma \det \bigl( \bfAhom(\cudot_{0}) \bigr)^{\frac{1}{d}}\,.
\end{equation}
\end{proposition}

The proof of Proposition~\ref{p.renormalize}, and consequently Theorem~\ref{t.highcontrast}, relies on first finding a range of scales over which the coarse-grained matrices do not change much, and then showing that on these scales the problem must already have homogenized to a desired degree. We state this in the following lemma and proposition.

\begin{lemma}[Pigeonhole lemma]
\label{l.pigeon}
For every~$\delta_1,\sigma \in (0,\nicefrac12]$ and~$l,N\in\N$ satisfying,
\begin{equation}
\label{e.pigeonhole.scale}
N \geq \bigg\lceil \frac{2|\log \sigma|}{\delta_1}\bigg\rceil l\,,
\end{equation}
for every~$m_1 \in\N$ either
\begin{itemize}
\item[•]
$\bfAhom(\cudot_{m-l}) \leq (1+\delta_1) \bfAhom(\cudot_{m}) \quad \mbox{ for some } m\in [m_1 + l, m_1 + N]$
\end{itemize}
or
\begin{itemize}
\item[•]
$\bigl( \det\bfAhom(\cudot_{m_1 + N})\bigr)^{\frac{1}{d}} \leq \sigma \bigl(\det \bfAhom(\cudot_{m_1}) \bigr)^{\frac{1}{d}}$.
\end{itemize}
\end{lemma}
\begin{proof}
Exactly as in~\cite[Lemma 3.4]{AK.HC}.
\end{proof}

\begin{proposition}
\label{p.renormalize.reduce}
There exists a constant~$\delta_0(d)>0$ such that if~$\delta,\sigma \in (0,\nicefrac12]$,~$l,m\in \N$,~$m\geq 100l$,
\begin{equation}
\label{e.fix.scale.separation}
\max\biggl\{ \frac{C(d)K_{\Psi_{\S}}^{4d+14}\Pi_{\mathrm{par}}}{(1-\gamma)^2}3^{-(1-\gamma)l}, C(d)K_{\Psi_{\S}}^{18} \Pi_{\mathrm{par}}^{d+2}3^{-\frac{1}{2}(\nu-\gamma)l}, \frac{C(d)K_{\Psi_{\S}}^9\Pi_{\mathrm{par}}}{1-\gamma}3^{-\frac{1}{4}(1-\gamma)(1-\beta)l}\biggr\} \leq \delta \sigma^2\,,
\end{equation}
the matrix~$\bfE$ in~\ref{a.ellipticity.dagger} satisfies
\begin{equation*}
\bfAhom(\cudot_0) \leq \bfE \quad \mbox{and} \quad |\bfE^{\nicefrac12}\bfAhom^{-1}(\cudot_m)\bfE^{\nicefrac12} - \Itwod | \leq \delta \sigma^2\,,
\end{equation*}
and~$\delta \leq \delta_0$, then
\begin{equation*}
\Theta_m - 1 \leq \sigma \Theta\,.
\end{equation*}
\end{proposition}

The statements of these propositions are nearly the same as in the elliptic case, although the proof of Proposition~\ref{p.renormalize.reduce} has to be modified using Lemma~\ref{l.Jtilde.energy.bound}. For completeness we outline how these propositions imply Theorem~\ref{t.highcontrast}, but since the arguments are nearly the same as in the elliptic case the reader is referred to~\cite{AK.HC} for the technical details.

First, applying Lemma~\ref{l.pigeon} with suitable scale separation parameters~$l,N$ and small enough~$\delta$, we obtain either a range of scales~$k\in [m-l,m]$ such that~$\bfAhom(\cudot_k)$ doesn't change much (because~$\bfAhom(\cudot_{m-l}) \leq (1+\delta)\bfAhom(\cudot_m)$), or a contraction of the determinant,
\begin{equation*}
\det\bfAhom(\cudot_{m_1 + N}) \leq \sigma^d \det \bfAhom(\cudot_{m_1})\,.
\end{equation*}
In the latter case we obtain the second option in Proposition~\ref{p.renormalize} (noting the monotonicity of~$\bfAhom(\cudot_k)$). In the former case we apply the renormalization in Proposition~\ref{p.renormalization.P} (with appropriate parameters) to zoom out~$n_0$ scales, which implies that the renormalized measure~$\mathbb{P}_{n_0}$ satisfies the assumptions of Proposition~\ref{p.renormalize.reduce}. Thus for~$m$ large enough ($m$ from Proposition~\ref{p.renormalize.reduce} plus the scales~$n_0$ that we zoomed out) we get~$\Theta_m - 1 \leq \sigma\Theta_0$, which is the first option in Proposition~\ref{p.renormalize}.

Once we obtain Proposition~\ref{p.renormalize}, the proof of Theorem~\ref{t.highcontrast} is an iteration. For~$\sigma \leq \nicefrac12$
\begin{equation*}
\Theta_m - 1 \leq \sigma \Theta_0 \implies \Theta_m - (1+2\sigma) \leq \sigma(\Theta_0 - (1+2\sigma))\,,
\end{equation*}
so for~$m$ large enough,~\eqref{e.improve} and the monotonicity of~$\Theta_m$ and~$\det \bfAhom(\cudot_m)$ imply that
\begin{equation*}
\bigl(\det \bfAhom(\cudot_m) \bigr)^{\frac{1}{d}}(\Theta_m - (1+2\sigma) ) \leq \sigma\bigl(\det \bfAhom(\cudot_0) \bigr)^{\frac{1}{d}}(\Theta_0 - (1+2\sigma) )\,.
\end{equation*}
Since we get a contraction by a factor of~$\sigma$ each time, iterating this inequality~$\log(1+\Theta)$ many times (with the help of the renormalization in Proposition~\ref{p.renormalization.P}) and noting that~$\det \bfAhom(\cudot_m) = \det (\shom(\cudot_m)\shom_*^{-1}(\cudot_m)) \geq 1$ gives Theorem~\ref{t.highcontrast}. This reasoning is illustrated as follows:
\begin{equation*}
\left.
\begin{gathered}
\mbox{Lemma~\ref{l.pigeon}} \\
+ \\
\mbox{Proposition~\ref{p.renormalize.reduce}}
\end{gathered}
\right\}
\implies
\mbox{Proposition~\ref{p.renormalize}}
\implies
\mbox{Theorem~\ref{t.highcontrast}}\,.
\end{equation*}

\subsection{One renormalization step}
\label{ss.one.renorm.step}
In this subsection we present the proof of Proposition~\ref{p.renormalize.reduce}. Once we set up the argument and use the parabolic framework in the proof of Lemma~\ref{l.Jtilde.energy.bound}, the rest of the argument follows~\cite[Section 3]{AK.HC}. Throughout this section we fix the following parameters:

\begin{itemize}
\item $\delta_0\in (0,\nicefrac12]$ is a constant depending only on~$d$, to be selected at the end of the proof;
\item $\sigma \in (0,\nicefrac12]$ and~$\delta \in (0,\delta_0]$ are given constants;
\item We fix an integer~$l$ representing a mesoscopic scale, such that
\begin{equation}
\label{e.sizeof.l}
\biggl\{ \frac{C(d)K_{\Psi_{\S}}^{4d+14}\Pipar}{(1-\gamma)^2}3^{-(1-\gamma)l}, C(d)K_{\Psi}^{18} \Pipar^{d+2}3^{-\frac{1}{2}(\nu-\gamma)l}, \frac{C(d)K_{\Psi_{\S}}^9 \Pipar^3 }{1-\gamma}3^{-\frac{1}{4}(1-\gamma)(1-\beta)l}\biggr\} \leq \delta \sigma^2
\end{equation}
and note that by taking the constants large enough, with reference to Section~\ref{ss.subadditivity}, we have~$l\geq C(d)\log(1+ \lambda_0)$ so that~$\mathbb{L}_l \subseteq \mathbb{Z}^{d+1}$ and the stationarity assumption is valid in the adapted parabolic cubes.
\item Suppose that~$m\in\N$ with~$m\geq 100 l$ such that
\begin{equation}
\label{e.pigeonholing.assumption}
\bfAhom(\cudot_0) \leq \mathbf{E}_0 \quad \mbox{and} \quad |\bfAhom^{-\nicefrac12}(\cudot_m)\mathbf{E}_0\bfAhom^{-\nicefrac12}(\cudot_m) - \Itwod| \leq \delta \sigma^2\,.
\end{equation}
\end{itemize}
To simplify the presentation, throughout this section we work with the following notation and assumptions:
\begin{itemize}
\item For every~$j\in\N$ we define~$\bfAhom_j := \bfAhom(\cudot_j)\,, \shom_j := \shom(\cudot_j)\,, \shom_{*,j} := \shom_{*}(\cudot_j)\,,$ and~$\khom_j := \khom(\cudot_j)$. Similarly we define~$\bhom_j := \shom_j + \khom_j^t\shom_{*,j}^{-1}\khom_j$, and, given a constant matrix~$\h$ we set~$\bhom_{\h,j}:= \shom_j + (\khom_j - \h)^t \shom_{*,j}^{-1}(\khom_j - \h)$.
\item The coefficient is ``centered'' so that the anti-symmetric part of a certain annealed coarse-grained matrix vanishes. By subtracting the matrix~$\frac{1}{2}(\khom_m - \khom_m^t)$ from the coefficient field and recentering both~$\mathbf{E}_0$ and~$\bfA(I\times U)$ accordingly (as in section~\ref{ss.bfA.def}), we may assume that
\begin{equation}
\label{e.centering}
\khom_m = \khom_m^t\,.
\end{equation}
Note that the definitions of~$\Theta$ and~$\Pi$ are independent of the centering, so these quantities remain unchanged.
\item[•] Under the centering in~\eqref{e.centering} we now define the adapted cubes by taking~$\m_0$ as in~\eqref{e.m0.def} to be
\begin{equation*}
\m_0 := (\s_0 + \k_0^t\s_{*,0}^{-1}\k_0 ) \# \s_{*,0}\,.
\end{equation*}
\end{itemize}

Our first lemma is a technical statement controlling the effect of the centering in~\eqref{e.centering}.
\begin{lemma}
\label{l.centering.ok}
Under the choice of centering in~\eqref{e.centering} we have
\begin{equation}
\label{e.m0.controlled}
\lambda \Id \leq \m_0 \leq \sqrt{8d}\Theta_m^{\nicefrac12} \Lambda \Id\,,
\end{equation}
and
\begin{equation}
\label{e.centering.controls.ratio}
|\m_0^{-\nicefrac12}\b_0\m_0^{-\nicefrac12}| + |\s_{*,0}^{-\nicefrac12}\m_0\s_{*,0}^{-\nicefrac12}| \leq \sqrt{8d} \Theta_m^{\nicefrac12}.
\end{equation}
\end{lemma}
\begin{proof}
We first establish a series of facts that will be needed. We have by the pigeonholing~\eqref{e.pigeonholing.assumption} that
\begin{equation*}
\bfAhom_m \leq \bfAhom_0 \leq \bfE \leq 2\bfAhom_m\,.
\end{equation*}
By conjugation by~$\mathbf{G}_{\h}$, as in Section~\ref{ss.bfA.def}, we then have that for any~$\h\in\R^{d\times d}$,
\begin{align*}
\shom_m + (\khom_m-\h)^t\shom_{*,m}^{-1}(\khom_m-\h) & \leq \s_0 + (\k_0 - \h)^t \s_{*,0}^{-1}(\k_0 - \h) \nonumber \\
& \leq 2(\shom_m + (\khom_m-\h)^t\shom_{*,m}^{-1}(\khom_m-\h))\,,
\end{align*}
and
\begin{equation*}
\shom_{*,m}^{-1} \leq \s_{*,0}^{-1} \leq 2\shom_{*,m}^{-1}\,.
\end{equation*}
In particular, for~$\h=0$ we get
\begin{equation*}
\b_0 \leq 2(\shom_m + \khom_m^t\shom_{*,m}^{-1}\khom_m)\,.
\end{equation*}
Suppose generally that~$\tilde{\s}$ is any symmetric matrix and~$\tilde{\k}$ is a skew-symmetric matrix that minimizes
\begin{equation*}
\tr \bigl( \s_{*,0}^{-\nicefrac12}(\tilde{\s}+\tilde{\k})^t\s_{*,0}^{-1}(\tilde{\s}+\tilde{\k})\s_{*,0}^{-\nicefrac12} \bigr)\,.
\end{equation*}
By the first variation, for any skew-symmetric matrix~$\tilde{\h}$ we have
\begin{equation*}
\tr \bigl( \s_{*,0}^{-1}(\tilde{\s}+\tilde{\k})^t\s_{*,0}^{-1}\tilde{\h} \bigr) = 0\,,
\end{equation*}
from which it follows that~$\s_{*,0}^{-1}(\tilde{\s}+\tilde{\k})^t\s_{*,0}^{-1}$ is symmetric, and therefore that~$\tilde{\k}=0$. Similarly, if we fix a skew-symmetric~$\tilde{\k}$ then
\begin{equation*}
\min_{\tilde{\s}\in \R^{d\times d}_{\mathrm{sym}}} \tr \bigl( \s_{*,0}^{-\nicefrac12}(\tilde{\s}+\tilde{\k})^t\s_{*,0}^{-1}(\tilde{\s}+\tilde{\k})\s_{*,0}^{-\nicefrac12} \bigr) = \tr \bigl( \s_{*,0}^{-\nicefrac12}\tilde{\k}^t\s_{*,0}^{-1}\tilde{\k}\s_{*,0}^{-\nicefrac12} \bigr)\,.
\end{equation*}
We now turn to proving~\eqref{e.m0.controlled}; the lower bound always holds so we only need to prove the upper bound. We have
\begin{align*}
|\s_{*,0}^{-\nicefrac12}\b_0\s_{*,0}^{-\nicefrac12}| & \leq 2|\s_{*,0}^{-\nicefrac12}(\shom_m + \khom_m^t\shom_{*,m}^{-1}\khom_m)\s_{*,0}^{-\nicefrac12}| \leq 4\tr \bigl(\s_{*,0}^{-\nicefrac12}(\shom_m + \khom_m^t\s_{*,0}^{-1}\khom_m)\s_{*,0}^{-\nicefrac12} \bigr) \\
& = 2\inf_{\h\in\R^{d\times d}_{\mathrm{skew}}} \tr \bigl(\s_{*,0}^{-\nicefrac12}(\shom_m + (\khom_m-\h)^t\shom_{*,0}^{-1}(\khom_m-\h))\s_{*,0}^{-\nicefrac12} \bigr) \\
& \leq 2d \inf_{\h\in\R^{d\times d}_{\mathrm{skew}}} |\s_{*,0}^{-\nicefrac12}(\shom_m + (\khom_m-\h)^t\shom_{*,0}^{-1}(\khom_m-\h))\s_{*,0}^{-\nicefrac12}| \\
& \leq 8d \inf_{\h\in\R^{d\times d}_{\mathrm{skew}}} |\s_{*,m}^{-\nicefrac12}(\shom_m + (\khom_m-\h)^t\shom_{*,m}^{-1}(\khom_m-\h))\s_{*,m}^{-\nicefrac12}| \\
& = 8d \Theta_m\,.
\end{align*}
We then crudedly bound~$\s_{*,0} \leq \b_{\h_0} \leq \Lambda\Id$ to obtain
\begin{equation*}
\m_0 = \s_{*,0}^{\nicefrac12}(\s_{*,0}^{-\nicefrac12}\b_0\s_{*,0}^{-\nicefrac12})^{\nicefrac12}\s_{*,0}^{\nicefrac12} \leq |\s_{*,0}||\s_{*,0}^{-\nicefrac12}\b_0\s_{*,0}^{-\nicefrac12}|^{\nicefrac12} \leq \Lambda \sqrt{8d}\Theta_m^{\nicefrac12}\,.
\end{equation*}
To prove~\eqref{e.centering.controls.ratio} note that~$|\m_0^{-\nicefrac12}\b_0\m_0^{-\nicefrac12}| = |\m_0^{\nicefrac12}\s_{*,0}^{-1}\m_0^{\nicefrac12}| = |\s_{*,0}^{-\nicefrac12}\m_0\s_{*,0}^{-\nicefrac12}|$, and that by the definition of~$\m_0$, and the above,
\begin{equation*}
|\s_{*,0}^{-\nicefrac12}\m_0\s_{*,0}^{-\nicefrac12}| = |\s_{*,0}^{-\nicefrac12}\b_0\s_{*,0}^{-\nicefrac12}|^{\nicefrac12} \leq \sqrt{8d}\Theta_m^{\nicefrac12} \,.
\end{equation*}
\end{proof}

Let the antisymmetric part of~$\khom(I\times U)$ be denoted
\begin{equation}
\label{e.hhom.def}
\hhom(I\times U) := \frac{1}{2}(\khom - \khom^t)(I\times U)\,,
\end{equation}
and define
\begin{equation}
\label{e.thom.def}
\thom(I\times U) := \bhom_{\hhom(I\times U)} \# \shom_*(I\times U)\,.
\end{equation}
We define a variant of~$J$ by
\begin{equation}
\label{e.Jminusmeans}
\tilde{J}(I\times U,p,q) := J(I\times U,p,q) - \frac{1}{2}\E\biggl[ \fint_I\fint_U \nabla v(\cdot,\cdot,I\times U,p,q) \biggr]\E\biggl[ \fint_I\fint_U \a\nabla v(\cdot,\cdot,I\times U,p,q) \biggr]\,,
\end{equation}
With this notation we can now state Lemma~\ref{l.Jcontrols.Theta}; this lemma is valid if we replace~$\thom(I\times U)$ with any positive-definite symmetric matrix, but we will use~$\thom(I\times U)$ later.

\begin{lemma}
\label{l.Jcontrols.Theta}
For every bounded Lipschitz domain~$U\subseteq \Rd$, and interval~$I\subseteq \R$, if
\begin{align*}
p = \thom^{-\nicefrac12}(I\times U)e\,, \quad q = \thom(I\times U)p-\hhom(I\times U)p\,, \quad \mbox{and} \quad  q' = \thom(I\times U)p +\hhom(I\times U)p
\end{align*}
then
\begin{align}
\label{e.Jcontrols.Theta}
|(\shom_*^{-\nicefrac12}\bhom_{\hhom(I\times U)}\shom_*^{-\nicefrac12})(I\times U)-\Id|  \leq \sup_{|e|=1}\biggl(\E \bigl[ \tilde{J}(I\times U,p,q)\bigr]  + \E \bigl[ \tilde{J}^*(I\times U,p,q')\bigr] \biggr)
\end{align}
\end{lemma}
\begin{proof}
See~\cite[Lemma 3.6]{AK.HC}.
\end{proof}

Our centering assumption states that~$\hhom(\cudot_m) = 0$, so writing~$\thom_m := \thom(\cudot_m)$ if we define
\begin{equation}
\label{e.choice.pq}
e' := \mbox{maximizer in Lemma \ref{l.Jcontrols.Theta}} \,, \quad p' := \thom_m^{-\nicefrac12}e'\,, \quad \mbox{and} \quad q':= \thom_m^{\nicefrac12}e'
\end{equation}
then we have the bound
\begin{equation}
\label{e.Jscontrol.Theta}
\Theta_m - 1 \leq \E[\tilde{J}(\cudot_m,p',q') + \tilde{J}^*(\cudot_m,p',q')]\,.
\end{equation}
This bound can be formulated in terms of the adapted parabolic cubes. First, define~$P',Q',P^*,Q^*\in\Rd$ by
\begin{equation}
\label{e.PQ.choice}
\begin{pmatrix}
P' \\ Q'
\end{pmatrix}
: = \E\biggl[ \fint_{\cudot_m} \begin{pmatrix} \nabla v(\cdot,\cdot,\cudot_m,p',q') \\ \a\nabla v(\cdot,\cdot,\cudot_m,p',q') \end{pmatrix} \biggr]
\quad \mbox{and} \quad
\begin{pmatrix}
P^* \\ Q^*
\end{pmatrix}
: = \E\biggl[ \fint_{\cudot_m} \begin{pmatrix} \nabla v^*(\cdot,\cdot,\cudot_m,p',q') \\ \a^t\nabla v^*(\cdot,\cdot,\cudot_m,p',q') \end{pmatrix} \biggr]\,.
\end{equation}
By Lemma~\ref{l.adapted.means} we have, for~$n\leq m$ such that~$\cusdot_n \subseteq \cudot_m$,
\begin{align}
\label{e.Jbound.to.tildes}
\E\bigl[ & \tilde{J}(\cudot_m,p',q') + \tilde{J}^*(\cudot_m,p',q')\bigr] \nonumber \\
& = \E\bigl[J(\cudot_m,p',q') + J^*(\cudot_m,p',q')\bigr] -\frac{1}{2}P'\cdot Q' - \frac{1}{2}P^* \cdot Q^* \nonumber \\
& \leq \E\bigl[ J(\cusdot_n,p',q') + J^*(\cusdot_n,p',q')\bigr] -\frac{1}{2}P'\cdot Q' - \frac{1}{2}P^* \cdot Q^* \nonumber \\
& \quad \quad \quad + \frac{C(d)K_{\Psi_{\S}}^9}{1-\gamma}\Pi^{\frac{1-\gamma}{2}}\Pi^2 3^{-(1-\gamma)(m-n)}\,.
\end{align}
In the last line we have bounded
\begin{equation}
\label{e.bound.E0}
\begin{pmatrix}
p' \\ q' 
\end{pmatrix}
\cdot
\bfE
\begin{pmatrix}
p' \\ q'
\end{pmatrix}
\leq 2(\Lambda |\thom_m^{-1}| \vee \lambda^{-1}|\thom_m|) \leq 2d\Pi^2\,.
\end{equation}
This uses the choices in~\eqref{e.choice.pq} and the centering in the form of
\begin{align*}
|\shom_{*,m}^{-\nicefrac12}\bhom_m\shom_{*,m}^{-\nicefrac12}\shom_{*,m}^{-\nicefrac12}| & \leq \tr \bigl(\shom_{*,m}^{-\nicefrac12}(\shom_m + \khom_m\shom_{*,m}^{-1}\khom_m) \shom_{*,m}^{-\nicefrac12}\bigr) \\
& = \inf_{\h\in\R^{d\times d}_{\mathrm{skew}}}\tr \bigl(\shom_{*,m}^{-\nicefrac12}(\shom_m + (\khom_m-\h)^t\shom_{*,m}^{-1}(\khom_m-\h)) \shom_{*,m}^{-\nicefrac12}\bigr) \\
& \leq d \inf_{\h\in\R^{d\times d}_{\mathrm{skew}}} |\shom_{*,m}^{-\nicefrac12}(\shom_m + (\khom_m-\h)^t\shom_{*,m}^{-1}(\khom_m-\h)) \shom_{*,m}^{-\nicefrac12}| \\
& = d\Theta_m\,.
\end{align*}
This implies (crudely) that~$\bhom_m \leq d\Theta\shom_{*,m} \leq  \Theta \Lambda \Id$; then~$\thom_m \leq \bhom_m \leq d \Theta \Lambda$, and~$\thom_m^{-1} \leq \shom_{*,m}^{-1} \leq \lambda^{-1}\Id$, which completes the proof of~\eqref{e.bound.E0}.

The last term in~\eqref{e.Jbound.to.tildes} is an error term that will be controlled by~$\delta \sigma^2$, choosing~$n=m-l$ provided that~$l$ satisfies
\begin{equation}
\label{e.mesoscale.cond.three}
\frac{C(d)K_{\Psi_{\S}}^9}{1-\gamma}\Pi^{\frac{5-\gamma}{2}} 3^{-(1-\gamma)(m-n)} \leq \delta\sigma^2\,.
\end{equation}

Our next lemma, and the main part of the proof, is a quantitative div-curl type argument that says exactly that we can control the right-hand side of~\eqref{e.Jbound.to.tildes}. The key part of the lemma adapts naturally to the parabolic setting, using our parabolic functional inequalities; the rest of the lemma, once we input the properties of the parabolic coarse-grained matrices, follows exactly the elliptic case.

\begin{lemma}
\label{l.Jtilde.energy.bound}
There exists a constant~$C(d) < \infty$ such that, for~$n=m-l$, we have the estimate
\begin{align}
\label{e.Jtilde.energy.bound}
\bigl| \mathbb{E}[J(\cusdot_n,p',q') & - \frac{1}{2}P'\cdot Q'\bigr| \leq  C\delta^{\nicefrac12}\sigma\Theta^{\nicefrac12}\Theta_m^{\nicefrac12}\,.
\end{align}
\end{lemma}
\begin{proof}
Fix~$k_0 \geq C(d)\log(1+\lambda)$ so that for all scales above~$k_0$ the adapted parabolic cubes are stationary, as in Section~\ref{ss.subadditivity}. Let~$p',q'$ be as in~\eqref{e.choice.pq} and for any~$n > k+3 > k \geq k_0$, where~$k,n\in\N$, and any~$z\in\mathbb{L}_k$, we denote
\begin{equation*}
v_{z,k} := v(\cdot,\cdot,z+\cusdot_k,p',q')\,, \quad J(z+\cusdot_k) = J(z+\cusdot_k,p',q')\,, \quad \mbox{and} \quad \bar{\tau}_{n,k} := \E[J(\cusdot_k) - J(\cusdot_n)]\,.
\end{equation*}
We also let~$v_n := v(\cdot,\cusdot_n,p',q')$. Fix a nonnegative smooth test function~$\varphi\in C_c^\infty(\cusdot_n)$ with support in~$\cusdot_{n-1}$ such that
\begin{equation}
\label{e.varphi.bounds}
(\varphi)_{\cusdot_n} = 1\,, \quad 0 \leq \varphi \leq 2\,, \quad \lambda_r^{-1} \norm{\partial_t \varphi}_{L^\infty(\cusdot_n)} \leq C3^{-2n}\,,  \mbox{ and }  \norm{\q_0^j \nabla^j \varphi}_{L^\infty(\cusdot_n)} \leq C3^{-jn}\,,  j\in \{1,2\}\,.
\end{equation}
so we have chosen~$\varphi$ such that its oscillations are controlled. By rearranging terms,
\begin{align*}
\E[J(\cusdot_n)] - \frac{1}{2}P'\cdot Q' =
& \E \biggl[ \fint_{\cusdot_n} \frac{1}{2}\varphi (\nabla v_n - P') \cdot (\a\nabla v_n - Q')\biggr] \\
& \quad + \E\biggl[ J(\cusdot_n) - \fint_{\cusdot_n} \frac{1}{2}\varphi \nabla v_n \cdot \s \nabla v_n \biggr] \\
& \quad + \frac{1}{2}Q' \cdot \E[\bigl((\varphi-1)\nabla v_n\bigr)_{\cusdot_n}] + \frac{1}{2}P' \cdot \E[\bigl((\varphi-1)\a \nabla v_n\bigr)_{\cusdot_n}] \\
& \quad + \frac{1}{2}Q' \cdot \bigl( \E[(\nabla v_n)_{\cusdot_n}] - P' \bigr) + \frac{1}{2}P' \cdot \bigl( \E[(\a\nabla v_n)_{\cusdot_n}] - Q' \bigr)\,.
\end{align*}
The last three lines will all be small up to a scale separation and therefore controlled by the right-hand side of~\eqref{e.Jtilde.energy.bound}. The bounds are exactly as in~\cite[Lemma 3.7]{AK.HC} and we do not include the proof here.

For the first line, we let~$\ell_{P'}(x) \coloneqq (v_n)_{\cusdot_n} + P'\cdot x$, and integrate by parts using Lemma~\ref{l.testing.with.u} to obtain
\begin{align*}
\frac{1}{2}\fint_{\cusdot_n} \varphi(\nabla v_n - P') \cdot (\a\nabla v_n - Q') = -\frac{1}{2}\fint_{\cusdot_n} (v_n - \ell_{P'})\nabla \varphi \cdot (\a\nabla v_n - Q') + \frac{1}{4}\fint_{\cusdot_n} (v_n - \ell_{P'})^2\partial_t \varphi\,.
\end{align*}
The first term is bounded using~\eqref{e.div.curl.lemma.eq}, noting that~$\varphi$ is supported in~$\cusdot_{n-1}$,
\begin{align*}
\lefteqn{
\biggl| \fint_{\cusdot_n} (v_n - \ell_{P'})\nabla \varphi \cdot (\a\nabla v_n - Q') \biggr|
} \: & \\
&  \leq \|\m_0^{\nicefrac12}\nabla \varphi (v_n - \ell_{P'})\|_{\underline{B}_{2,\infty}^{\nicefrac12}(\cusdot_n)} [\m_0^{-\nicefrac12}(\a\nabla v_n - Q') ]_{\underline{B}_{2,1}^{-\nicefrac12}(\cusdot_n)} \\
& \leq C3^{-n}\bigl([\m_0^{\nicefrac12}(\nabla v_n - P')]_{\underline{B}_{2,1}^{-\nicefrac12}(\cusdot_n)} + [\m_0^{-\nicefrac12}(\a\nabla v_n - Q')]_{\underline{B}_{2,1}^{-\nicefrac12}(\cusdot_n)}\bigr)[\m_0^{-\nicefrac12}(\a\nabla v_n - Q') ]_{\underline{B}_{2,1}^{-\nicefrac12}(\cusdot_n)}\,.
\end{align*}
Similarly we bound
\begin{align*}
\fint_{\cusdot_n} (v_n - \ell_{P'})^2\partial_t \varphi \leq \|\lambda_r^{-1}\partial_t \varphi\|_{L^\infty(\cusdot_n)} \bigl( [\m_0^{\nicefrac12}(\nabla v_n - P')]_{\underline{B}_{2,1}^{-1}(\cusdot_n)} + [\m_0^{-\nicefrac12}(\a\nabla v_n - Q')]_{\underline{B}_{2,1}^{-1}(\cusdot_n)}\bigr)^2\,,
\end{align*}
using~\eqref{e.parabolic.multiscale.poincare} and again noting that on the left-hand side we have nonzero contributions only from the domain~$\cusdot_{n-1}$. Therefore the first term contributes
\begin{equation*}
\biggl| \E \biggl[ \fint_{\cusdot_n} \frac{1}{2}\varphi (\nabla v_n - P') \cdot (\a\nabla v_n - Q')\biggr] \biggr| \leq C3^{-n} \mathbb{E}\biggl[ \biggl[\mathbf{M}_0^{\nicefrac12}
\begin{pmatrix}
\nabla v(\cusdot_n,p',q') - P' \\ \a\nabla v(\cusdot_n,p',q') - Q'\\
\end{pmatrix}
\biggr]^2_{\Bring_{2,1}^{-\nicefrac12}(\cusdot_{n})}\biggr]\,.
\end{equation*}
The right-hand side is bounded using the parabolic analogue of~\cite[Lemma 2.16]{AK.HC} (see the remark at the end of Section~\ref{ss.besov}), using the scale separation in~\eqref{e.sizeof.l}.

\end{proof}

%Finally, we are able to estimate the remaining term in Lemma~\ref{l.Jtilde.energy.bound}.
%\begin{lemma}
%\label{l.bound.the.rest}
%There exists~$C(d)<\infty$ such that for~$n=m-l$ we have
%\begin{equation}
%\label{e.Hminus.s}
%3^{-n} \mathbb{E}\biggl[ \biggl[\mathbf{M}_0^{\nicefrac12}
%\begin{pmatrix}
%\nabla v(\cusdot_n,p',q') - P \\ \a\nabla v(\cusdot_n,p',q') - Q\\
%\end{pmatrix}
%\biggr]^2_{\Bring_{2,1}^{-\nicefrac12}(\cusdot_{n})}\biggr] \leq C\delta \sigma^2\Theta^{\nicefrac12}\Theta_m^{\nicefrac12}\,.
%\end{equation}
%\end{lemma}
%\begin{proof}
%The proof is as in the elliptic case using~\cite[Lemma 2.14]{AK.HC} (which generalizes immediately to the parabolic case) provided we have a scale separation such that
%\begin{equation}
%\label{e.mesoscale.cond.five}
%\frac{\Pi^4K_{\Psi_{\S}}^{4d+15}}{(1-\gamma)^5}3^{-(m-l)} \leq \delta \sigma^2 \quad \mbox{and} \quad \frac{3^{-(1-\gamma)l}}{1-\gamma} \leq \delta \sigma^2.
%\end{equation}
%\end{proof}

\begin{proof}[{Proof of Proposition~\ref{p.renormalize.reduce}}]
This is exactly as in~\cite[Section 3.2]{AK.HC}, using the results in this section as replacements for the necessary inequalities, but we sketch the argument here. With the choice of parameters~$\delta,\sigma,l,m$ as in the proposition statement, taking~$n=m-l$ and combining~\eqref{e.Jscontrol.Theta},~\eqref{e.Jbound.to.tildes}, and Lemma~\ref{l.Jtilde.energy.bound}, we get
\begin{align*}
\Theta_m - 1 & \leq \mathbb{E}[ \tilde{J}(\cudot_m,p',q') + \tilde{J}^*(\cudot_m,p',q')]  \leq C(d)\delta^{\nicefrac12}\sigma \Theta^{\nicefrac12}\Theta_m^{\nicefrac12}\,.
\end{align*}
Hence if~$\delta$ is small enough, depending only on~$d$, then
\begin{equation*}
\Theta_m - 1\leq \frac{\sigma}{2}\Theta^{\nicefrac12}\Theta_m^{\nicefrac12} \leq \frac{\sigma}{4}\Theta + \frac{\sigma}{4}\Theta_m\,,
\end{equation*}
and solving for~$\Theta_m$ we obtain~$\Theta_m - 1 \leq \sigma \Theta$.
\end{proof}

\section{Renormalization in small contrast}
\label{s.smallcontrast}

Theorem~\ref{t.highcontrast} implies that~$\Theta_n \to 1$, with a quantitative rate. However the main purpose of Section~\ref{s.hc} is to bound the length scale~$3^n$ by which~$\Theta_n - 1$ is small; once we are in the small contrast regime an adaptation of standard homogenization arguments allows us to prove that~$\Theta_n - 1$ decays algebraically in the length scale. Our main theorem is the following:

\begin{theorem}
\label{t.theta.rate}
There exist constants~$C(d)<\infty$ and~$c(d) \in (0,\nicefrac{1}{4}]$ such that
\begin{equation}
\label{e.algebraic.theta}
\Theta_m - 1 \leq 3^{-\kappa (m-m_*)} \quad \mbox{for } m \geq m_*
\end{equation}
where
\begin{align}
\label{e.theta.rate.params}
\left\{
\begin{aligned}
& \alpha = (\min\{\nu,1\}-\gamma)(1-\beta) \\
& \kappa = \min\{c,c\alpha\} \\
& m_* = \biggl \lceil \frac{C}{\kappa} \log\bigl( \frac{K_{\Psi_{\S}}K_{\Psi}\Pi_{\mathrm{par}}}{\kappa}\bigr)\log(1+\Theta)  \bigg \rceil
\end{aligned}
\right.
\end{align}
\end{theorem}
Theorem~\ref{t.theta.rate} is a combination of Theorem~\ref{t.highcontrast} and an iteration starting from small contrast. Our next theorem formalizes the iteration step, which is independent of the arguments of Section~\ref{s.hc}. Recall that~$\Theta$ is the ellipticity constant from~\ref{a.ellipticity.dagger}, while~$\Theta_0$ is the ellipticity constant of~$\bfAhom(\cudot_0)$. Since we will be using the renormalization as in Proposition~\ref{p.renormalization.P} they will be essentially the same, up to technical details.

\begin{theorem}
\label{t.small.contrast.rate}
There exist constants~$C(d),\sigma_0(d),c(d) \in (0,\infty)$ such that if~$\Theta \leq 2$ then
\begin{equation}
\label{e.small.contrast.rate.one}
\Theta_0 - 1 \leq \sigma_0(d) \implies \Theta_{m} - 1 \leq 3^{-\kappa(m-m_0)} \quad \mbox{for } m\geq m_0\,,
\end{equation}
where
\begin{align}
\label{e.small.c.params}
\left\{
\begin{aligned}
& \kappa = c(d)\min\{1,1-\gamma,(\nu-\gamma)(1-\beta)\}\,, \\
& m_0 = \bigg\lceil \frac{C(d)}{\kappa}\log\biggl(\frac{C(d)\Pi_{\mathrm{par}}K_{\Psi_{\S}}K_{\Psi}}{\kappa} \biggr) \bigg\rceil\,.
\end{aligned}
\right.
\end{align}
\end{theorem}

\begin{proof}[{Proof of Theorem~\ref{t.theta.rate} using Theorem~\ref{t.small.contrast.rate}}]
By Theorem~\ref{t.highcontrast} there exists a scale~$n_0$ satisfying
\begin{equation*}
n_0 \leq \frac{C(d)}{\alpha}\biggl( \log(K_{\Psi}\Pi_{\mathrm{par}}) + \frac{1}{\alpha}\log\bigl( \frac{K_{\Psi_{\S}}\Pi_{\mathrm{par}}}{\alpha}\bigr) \biggr)\log(1+\Theta)\,,
\end{equation*}
such that if~$\sigma_0(d)$ is the constant in Theorem~\ref{t.small.contrast.rate} then
\begin{equation*}
\Theta_{n_0} - 1 \leq \sigma_0\,.
\end{equation*}
For~$l_0$ given as in~\eqref{e.l0.condition} we may assume generally that~$n_0 \geq 3l_0$. We then apply Proposition~\ref{p.renormalization.P} to~$n_0+l_0$, choosing the parameters~$\rho = \frac{\min\{\nu,1\}+\gamma}{2}$ and~$\delta = 1$, to obtain that the pushforward measure~$\mathbb{P}_{n_0+l_0}$, defined in~\eqref{e.Pn0}, satisfies the assumptions~\ref{a.stationarity},~\ref{a.ellipticity.dagger}, and~\ref{a.CFS} with the new parameters
\begin{align*}
\left\{
\begin{aligned}
& \mathbf{E}_{\mathrm{new}}  = 2\bfAhom(\cudot_{n_0}) \\
& \gamma_{\mathrm{new}}  = \rho \\
& K_{\Psi,\mathrm{new}} = K_{\Psi} \\
& K_{\Psi_{\S},\mathrm{new}} = \max\{K_{\Psi_{\S}},K_{\Psi}^{\lceil \nicefrac{1}{\mu}\rceil}\}
\end{aligned}
\right.
\end{align*}
Applying Theorem~\ref{t.small.contrast.rate} to~$\mathbb{P}_{n_0+l_0}$ (with~$m_0$ as in that statement) gives
\begin{equation*}
\Theta_{m,\mathrm{new}} - 1 \leq 3^{-\kappa_{\mathrm{new}} (m-m_0)}
\end{equation*}
with~$\kappa_{\mathrm{new}}$ given as in~\eqref{e.small.c.params} with~$\gamma_{\mathrm{new}}$ in place of~$\gamma$. In view of the identity~$\Theta_{m,\mathrm{new}} = \Theta_{m+n_0+l_0}$ and the estimate
\begin{equation*}
\kappa_{\mathrm{new}} \geq \frac{\kappa}{2}\,,
\end{equation*}
we conclude the proof after a relabelling of parameters.
\end{proof}

Combining quantitative convergence of the means with quantitative ergodicity, in the form of~\ref{a.CFS}, leads to a quenched convergence result. We omit the proof since it is identical to the proof in~\cite[Section 4.2]{AK.HC}, once we are able to take Theorem~\ref{t.small.contrast.rate} as an input. The homogenized matrix~$\bfAhom$ is defined in~\eqref{e.bfAhom.def}.

\begin{theorem}
\label{t.quenched}
Suppose that~$\Theta_m - 1 \leq \Upsilon 3^{-\kappa m}$. Then for each~$\delta > 0, \gamma' \in (\gamma,1)$ there exists a random variable~$\mathcal{Y}_{\delta,\gamma'}$ satisfying
\begin{equation}
\label{e.Y.integrability}
\mathcal{Y}_{\delta,\gamma'}^{(\nu-\gamma)(1-\beta)} = \mathcal{O}_{\Psi}\biggl( C K_{\Psi}^{4 + \frac{4d^2}{(\gamma' - \gamma)^2}} \biggl(\frac{\Upsilon}{\kappa \delta} \biggr)^{\nicefrac{d}{\kappa}}\biggr)\,,
\end{equation}
such that for~$\theta: = \frac{1}{8}\min\{\kappa,\gamma' - \gamma\}$ and every~$m\in \N$,
\begin{equation}
3^m \geq \mathcal{Y}_{\delta,\gamma'} \vee \mathcal{S} \implies \bfA(z+\cudot_k) \leq \biggl( 1 + \delta 3^{\gamma'(m-k)} \biggl( \frac{\mathcal{Y}_{\delta,\gamma'}\vee \mathcal{S}}{3^m}\biggr)^{\theta} \biggr) \bfAhom\,, \quad \forall k\in \Z \cap (-\infty,m], z\in \mathcal{Z}_k \cap \cudot_m\,.
\end{equation}
\end{theorem}

In the rest of this section we define the homogenized matrix~$\bfAhom$, introduce the relevant adapted cubes for the proofs, and prove a straightforward but useful result comparing the ellipticity ratios in normal and adapted domains.

\subsection{Homogenized matrix and adapted domains}
\label{ss.homogenized.mat}

For any~$e\in\R^d$ the sequences~$n\mapsto e\cdot \shom(\cudot_n) e$ and~$n\mapsto e \cdot \shom_*^{-1}(\cudot_n)e$ are non-increasing and bounded, by~\eqref{e.monotone.s}. Since the matrices are symmetric we therefore obtain the qualitative limits~$\shom(\cudot_n) \to \shom$ and~$\shom_*(\cudot_n) \to \shom_*$. The definition of~$\Theta_n$ implies that for any~$n\in\N$ we have~$\shom(\cudot_n) \leq \Theta_n \shom_*(\cudot_n)$ so it follows that
\begin{equation*}
\shom_* \leq \shom \leq \shom(\cudot_n) \leq \Theta_n \shom_*(\cudot_n) \leq \Theta_n \shom_*\,,
\end{equation*}
and the limit~$\Theta_n \to 1$ gives~$\shom = \shom_*$. Taking~$n\leq m$, conjugating with~$\mathbf{G}_{\khom(\cudot_n)}$ and using~$\bfAhom(\cudot_m) \leq \bfAhom(\cudot_n)$ we obtain
\begin{equation*}
\shom_*(\cudot_m) \leq \shom(\cudot_m) + (\khom(\cudot_m) - \khom(\cudot_n))^t\shom_*^{-1}(\cudot_m)(\khom(\cudot_m)-\khom(\cudot_n)) \leq \shom(\cudot_n) \leq \Theta_n\shom_*(\cudot_n) \leq \Theta_n \shom_*(\cudot_m)\,,
\end{equation*}
from which we obtain
\begin{equation}
\label{e.control.kdiff.theta}
(\khom(\cudot_m) - \khom(\cudot_n))^t\shom_*^{-1}(\cudot_m)(\khom(\cudot_m)-\khom(\cudot_n)) \leq (\Theta_n - 1)\shom_*(\cudot_m) \leq (\Theta_n -1)\shom\,.
\end{equation}
Then~$\Theta_n \to 1$ implies that the sequence~$\khom(\cudot_n)$ is Cauchy and converges to a limit~$\khom$. By~\eqref{e.symm.k.controlled} we have
\begin{equation*}
|\shom^{-\nicefrac12} (\khom(\cudot_n) + \khom^t(\cudot_n))\shom^{-\nicefrac12} | \leq \Theta_n - 1\,,
\end{equation*}
so we see by sending~$n\to\infty$ that~$\khom$ is skew-symmetric. We then define the homogenized matrix~$\bfAhom$ by
\begin{equation}
\label{e.bfAhom.def}
\bfAhom := \begin{pmatrix}
\shom + \khom^t\shom^{-1}\khom & -\khom^t\shom^{-1} \\
-\shom^{-1}\khom & \shom^{-1}
\end{pmatrix}\,,
\end{equation}
and define~$\bhom$ as the top left block
\begin{equation}
\label{e.bhom.def}
\bhom := \shom + \khom^t\shom^{-1}\khom\,.
\end{equation}
We see that~$\lim_{n\to\infty} \bfAhom(\cudot_n) = \bfAhom$ and~$\lim_{n\to\infty} \bhom(\cudot_n) = \bhom$.

In Section~\ref{s.hc} we used a pigeonholing lemma to find a sequence of scales where the coarse-grained matrices~$\bfAhom(\cudot_n)$ did not change much. We now state a simple lemma which proves that in the small contrast regime all of our double-variable matrices are essentially equivalent. Recall that
\begin{equation*}
\mathbf{M}_0 := \begin{pmatrix}
\m_0 & 0 \\ 0 & \m_0^{-1}
\end{pmatrix}\,,
\end{equation*}
where in this section we define
\begin{equation}
\label{e.ch.four.m0}
\m_0 \coloneqq \s_0\,.
\end{equation}

\begin{lemma}
\label{l.small.theta}
Suppose that
\begin{equation*}
n \geq C(d)\log\bigl(C(d)K_{\Psi_{\S}}\bigr)
\end{equation*}
and~$\Theta \leq 2$. Then
\begin{equation*}
\bfAhom(\cudot_n) \leq 2\bfE \leq C(d)\mathbf{M}_0 \leq C(d)\bfAhom(\cudot_n)\,.
\end{equation*}
\end{lemma}
\begin{proof}
We have
\begin{align*}
|\s_0^{-\nicefrac12}\b_0\s_0^{-\nicefrac12}| & \leq 16\mathrm{tr}(\s_0^{-\nicefrac12}\b_0\s_0^{-\nicefrac12})  = 16\inf_{\h\in\mathbb{R}^{d\times d}_{\mathrm{skew}}} \mathrm{tr}(\s_0^{-\nicefrac12}\b_{0,\h}\s_0^{-\nicefrac12}) \\
& \leq 16d \inf_{\h\in\mathbb{R}^{d\times d}_{\mathrm{skew}}}|\s_0^{-\nicefrac12}\b_{0,\h}\s_0^{-\nicefrac12}| = 16d\Theta\,,
\end{align*}
and therefore applying~\eqref{e.bfAhom.by.E0}, our lower bound on~$n$ and the definition of~$\Theta$,
\begin{align*}
\bfAhom(\cudot_n) \leq 2\bfE \leq 4\begin{pmatrix}
\b_0 & 0 \\ 0 & \s_{*,0}^{-1}
\end{pmatrix} \leq 48d\Theta \mathbf{M}_0 \leq 48d\Theta \bfE \leq 96d\Theta(1+32(\Theta - 1)) \bfAhom(\cudot_n)\,.
\end{align*}
\end{proof}

In order to prove Theorem~\ref{t.theta.rate} we need to work in adapted parabolic cylinders. In this section the adapted parabolic cubes are defined by~$\m_0$ as in~\eqref{e.ch.four.m0}, for which~$\Lambda_{\m_0} \leq \Lambda$. We make the standing assumption that we work above the scale
\begin{equation}
k_0 := C(d)\log (1+\lambda)
\end{equation}
at which (for an appropriate constant) the adapted parabolic cubes are stationary. By choosing a constant~$0< c \ll 1$ depending only on~$\lambda_{\m_0}$ and~$\Lambda_{\m_0}$ an application of Lemma~\ref{l.adapted.means} gives
\begin{equation*}
\bfAhom(\cudot_{c^{-1}n}) \leq \bfAhom(\cusdot_n) + C(d,\gamma,K_{\Psi_{\S}},\Pi_{\mathrm{par}}) 3^{-(1-\gamma)(c^{-1}-1)n}\bfE
\end{equation*}
and
\begin{equation*}
\bfAhom(\cusdot_n) \leq \bfAhom(\cusdot_{cn}) + C(d,\gamma,K_{\Psi_{\S}},\Pi_{\mathrm{par}}) 3^{-(1-\gamma)(1-c)n}\bfE\,,
\end{equation*}
so that~$\lim_{n\to\infty} \bfAhom(\cusdot_n) = \bfAhom$ and we similarly obtain the immediate qualitative convergence of the quantities in the adapted parabolic cylinders. 

The proof of Theorem~\ref{t.small.contrast.rate} works with adapted versions of the ellipticity defined by
\begin{equation}
\label{e.def.hat.Theta}
\hat{\Theta}_n := \frac{1}{d}\mathrm{trace}\bigl(\shom_*^{-\nicefrac12}(\cusdot_n)\shom(\cusdot_n)\shom_*^{-\nicefrac12}(\cusdot_n)\bigr)\quad \mbox{and} \quad \tilde{\Theta}_n = \bigl|\shom_*^{-\nicefrac12}(\cusdot_n)\shom(\cusdot_n)\shom_*^{-\nicefrac12}(\cusdot_n)\bigr|\,.
\end{equation}
These are monotone by~\eqref{e.monotone.s}, and we let~$\hat{\Theta}$ and~$\tilde{\Theta}$ denote the analogous quantities defined for~$\bfE$. We use the trace because additivity defects which appear in the proof are more easily estimated by a linear quantity. Of course we have that
\begin{equation*}
\frac{1}{d}(\tilde{\Theta}_n - 1) \leq \hat{\Theta}_n - 1 \leq \tilde{\Theta}_n - 1\,.
\end{equation*}
Since~$\Theta_n$ is defined in the normal parabolic cylinders while~$\hat{\Theta}_n$ and~$\tilde{\Theta}_n$ are defined in the adapted parabolic cylinders we need to transfer bounds on~$\tilde{\Theta}_n$ to those on~$\Theta_n$ and vice-versa which is the statement of the following lemma; note that the constant in the brackets will be small provided the scale separation parameters~$n-k$ and~$m-n$ are large. We include the following straightforward lemma because its proof does not appear in~\cite[Section 4]{AK.HC}.

\begin{lemma}
\label{l.bfAhoms.equivalent}
If~$k < n$ such that~$\cudot_k \subseteq \cusdot_n$ and~$m > n$ such that~$\cusdot_n \subseteq \cudot_m$ then for the length scales
\begin{align*}
& L \coloneqq \frac{Cd^{\nicefrac{3}{2}}K_{\Psi_{\S}}^9}{1-\gamma}\max\{\Pi_{\m_0}^{\nicefrac{\gamma}{2}},\lambda_{\m_0}^{-\nicefrac{\gamma}{2}}\} (1+3^{3-k}K_{\Psi_{\S}}^2)(1+32(\Theta - 1))
\\
& L' \coloneqq \frac{Cd^{\nicefrac{3}{2}}K_{\Psi_{\S}}^9}{1-\gamma}\Pi_{\m_0}^{\frac{1-\gamma}{2}} (1+32(\Theta-1))
\end{align*}
we have
\begin{align}
\bfAhom(\cusdot_n) \leq \bfAhom (\cudot_k)\bigl( 1 + L3^{-(1-\gamma)(n-k)}\bigr) \quad \mathrm{and} \quad \bfAhom(\cudot_m) \leq \bfAhom(\cusdot_n) \bigl( 1 + L'3^{-(1-\gamma)(m-n)}\bigr) \label{e.bfAhoms.equiv}\,.
\end{align}
Consequently, for any~$\delta \in (0,1]$, if
\begin{equation}
\label{e.scale.sep}
l := \bigg\lceil\frac{C(d)}{1-\gamma}\log\biggl(\frac{C(d)K_{\Psi_{\S}}\max\{\Pi_{\m_0},\lambda^{-1}_{\m_0}\}}{\delta(1-\gamma)} \biggr)\bigg\rceil
\end{equation}
and we assume further that~$n \geq k + l$,~$m\geq n + l$ and~$\Theta \leq 4$ then
\begin{align}
\tilde{\Theta}_n - 1 \leq (\Theta_k - 1) + \delta 3^{-(1-\gamma)(n-k-l)} \quad \mathrm{and} \quad \Theta_m - 1 \leq 4(\tilde{\Theta}_n - 1) + \delta 3^{-(1-\gamma)(m-n-l)}\,. \label{e.thetas.equiv}
\end{align}
\end{lemma}

\begin{proof}
\emph{Step 1: Bounds on coarse-grained matrices.}
Take~$k < n$ such that~$\cudot_k \subseteq \cusdot_n$. Then by~\eqref{e.adapted.mean.by.Ahom} and~\eqref{e.bfAhom.by.E0}
\begin{align*}
\bfAhom(\cusdot_n) & \leq \bfAhom(\cudot_k) + \frac{Cd^{\nicefrac{3}{2}}K_{\Psi_{\S}}^9}{1-\gamma}\max\{\Pi_{\m_0}^{\nicefrac{\gamma}{2}},\lambda_{\m_0}^{-\nicefrac{\gamma}{2}}\} 3^{-(1-\gamma)(n-k)}\bfE \\
& \leq \bfAhom (\cudot_k)\biggl( 1 + \frac{Cd^{\nicefrac{3}{2}}K_{\Psi_{\S}}^9}{1-\gamma}\max\{\Pi_{\m_0}^{\nicefrac{\gamma}{2}},\lambda_{\m_0}^{-\nicefrac{\gamma}{2}}\} 3^{-(1-\gamma)(n-k)}(1+3^{3-k}K_{\Psi_{\S}}^2)(1+32(\Theta - 1))\biggr)\,,
\end{align*}
which is the first inequality in~\eqref{e.bfAhoms.equiv}. To prove the second inequality we fix~$m>n$ such that~$\cusdot_n \subseteq \cudot_m$ and use~\eqref{e.Ahom.by.adapted.mean} to obtain
\begin{align}
\label{e.something.to.bound}
\bfAhom(\cudot_m) & \leq  \bfAhom(\cusdot_n) + \frac{Cd^{\nicefrac{3}{2}}K_{\Psi_{\S}}^9}{1-\gamma}\Pi_{\m_0}^{\frac{1-\gamma}{2}} 3^{-(1-\gamma)(m-n)}\bfE \,.
\end{align}
We can bound~$\bfE$ by a multiple of~$\bfAhom(\cusdot_n)$ by taking~$l' > n$ such that~$\cusdot_n \subseteq \cudot_{l'}$, combining~\eqref{e.bfAhom.by.E0} and~\eqref{e.Ahom.by.adapted.mean} to get
\begin{align*}
\bfE & \leq 2(1+32(\Theta-1))\bfAhom(\cudot_{l'}) \\
& \leq 2(1+32(\Theta-1))\biggl(\bfAhom(\cusdot_n) + \frac{C(d)K_{\Psi_{\S}}^9}{1-\gamma}\Pi_{\m_0}^{\frac{1-\gamma}{2}} 3^{-(1-\gamma)(l'-n)}\bfE \biggr)\,.
\end{align*}
Taking~$l'$ large enough we may make the coefficient on~$\bfE$ on the right-hand side less than~$\frac{1}{2}$, re-absorb, and obtain
\begin{equation}
\label{e.bfE.bfAhom}
\bfE \leq 4(1+32(\Theta-1))\bfAhom(\cusdot_n)\,.
\end{equation}
We then combine the above display with~\eqref{e.something.to.bound} to obtain~\eqref{e.bfAhoms.equiv}.

\smallskip
\emph{Step 2: Bounds on ellipticity.}
\newline
Retaining~$k,n,m$ as above, we have from the definition of~$\tilde{\Theta}_n$ and~\eqref{e.bfAhoms.equiv} that
\begin{align*}
\tilde{\Theta}_n & = |(\shom_*^{-\nicefrac12}\shom \shom_*^{-\nicefrac12})(\cusdot_n)| \\
& \leq \biggl( 1 + \frac{Cd^{\nicefrac{3}{2}}K_{\Psi_{\S}}^9}{1-\gamma}\max\{\Pi_{\m_0}^{\nicefrac{\gamma}{2}},\lambda_{\m_0}^{-\nicefrac{\gamma}{2}}\} 3^{-(1-\gamma)(n-k)}(1+3^{3-k}K_{\Psi_{\S}}^2)(1+32(\Theta - 1))\biggr)^2  \\
& \qquad \times |(\shom_*^{-\nicefrac12}\shom \shom_*^{-\nicefrac12})(\cudot_k)| \\
& \leq \Theta_k \biggl( 1 + \frac{Cd^{\nicefrac{3}{2}}K_{\Psi_{\S}}^9}{1-\gamma}\max\{\Pi_{\m_0}^{\nicefrac{\gamma}{2}},\lambda_{\m_0}^{-\nicefrac{\gamma}{2}}\} 3^{-(1-\gamma)(n-k)}(1+3^{3-k}K_{\Psi_{\S}}^2)(1+32(\Theta - 1))\biggr)^2\,.
\end{align*}
Applying the scale separation in~\eqref{e.scale.sep} with appropriate constants we obtain
\begin{equation*}
\tilde{\Theta}_n - 1 \leq \Theta_k\bigl(1 + \frac{\delta}{10}3^{-(1-\gamma)(n-k-l)} + \frac{\delta^2}{100}3^{-2(1-\gamma)(n-k-l)}\bigr) - 1 \leq (\Theta_k - 1) + \frac{\delta}{8}\Theta_k 3^{-(1-\gamma)(n-k-l)}\,,
\end{equation*}
and we conclude since by assumption~$\Theta_k \leq 4$. In the other direction we note that~\eqref{e.symm.k.controlled} gives
\begin{equation*}
\bigl|\bigl(\shom_*^{-\nicefrac12}(\khom + \khom^t)\shom_*^{-1}(\khom + \khom^t)\shom_*^{-\nicefrac12}\bigr)(\cudot_m)\bigr| \leq (|(\shom_*^{-\nicefrac12}\shom\shom_*^{-\nicefrac12})(\cudot_m)| - 1)^2\,,
\end{equation*}
so that applying the definition of~$\Theta_m$
\begin{align*}
\Theta_m & \leq \bigl|\bigl(\shom_*^{-\nicefrac12}( \shom + \frac{1}{4}(\khom+\khom^t)\shom_*^{-1}(\khom+\khom^t) )\shom_*^{-\nicefrac12}\bigr)(\cudot_m)\bigr| \\
& \leq |(\shom_*^{-\nicefrac12}\shom\shom_*^{-\nicefrac12})(\cudot_m)| + \frac{1}{4}(|(\shom_*^{-\nicefrac12}\shom\shom_*^{-\nicefrac12})(\cudot_m)| - 1)^2\,.
\end{align*}
Re-arranging and using the assumption~$\Theta \leq 4$ we obtain
\begin{equation*}
\Theta_m - 1 \leq  4 \bigl(|(\shom_*^{-\nicefrac12}\shom\shom_*^{-\nicefrac12})(\cudot_m)| - 1 \bigr)
\end{equation*}
and we then conclude as before using the scale separation in~\eqref{e.scale.sep} since
\begin{equation*}
|\shom_*^{-\nicefrac12}(\cudot_m)\shom(\cudot_m)\shom_*^{-\nicefrac12}(\cudot_m)| \leq \tilde{\Theta}_n \biggl( 1 + \frac{Cd^{\nicefrac{3}{2}}K_{\Psi_{\S}}^9}{1-\gamma}\Pi_{\m_0}^{\frac{1-\gamma}{2}} 3^{-(1-\gamma)(m-n)}(1+32(\Theta-1))\biggr)^2
\end{equation*}
which follows from~\eqref{e.bfAhoms.equiv}.
\end{proof}

Lemma~\ref{l.bfAhoms.equivalent} means that we can freely move back and forth between normal parabolic cylinders and adapted parabolic cylinders provided that have a scale separation~$l$ given by~\eqref{e.scale.sep}. For instance, fixing~$\delta=1$ in the definition of~$l$, if~$n\geq 2l$ then
\begin{equation}
\label{e.bfA.cus.equiv}
\bfAhom(\cusdot_n) \leq C(d)\bfAhom(\cudot_l) \leq C(d)\bfE \leq C(d)\mathbf{M}_0 \leq C(d)\bfAhom(\cudot_{2n}) \leq C(d)\bfAhom(\cusdot_n)\,.
\end{equation}
The proof of Theorem~\ref{t.small.contrast.rate} is now as follows. It has been checked\footnote{This calculation has been fully written out, but we have chosen not to include it because there are no substantial differences in the parabolic case} that the arguments of~\cite[Section 4]{AK.HC} generalize immediately to the parabolic setting, with the same proofs. This is essentially the statement that the properties of the coarse-grained matrices proved in Section~\ref{s.coarse.graining} are the same as the properties of the elliptic coarse-grained matrices, and that these properties alone suffice to prove Theorem~\ref{t.small.contrast.rate}. This implies that there exists a constant~$\sigma(d)$ such that
\begin{equation*}
\hat{\Theta}_n - 1 \leq \sigma(d) \implies \hat{\Theta}_{n+m} - 1 \leq 3^{-\kappa (m-m_0)}\,,
\end{equation*}
with~$\kappa$ and~$m_0$ given by~\eqref{e.small.c.params}. Combining this with Lemma~\ref{l.bfAhoms.equivalent} we see that there exists a constant~$\sigma_0(d)$ such that for~$l$ given as in~\eqref{e.scale.sep}, if~$n\geq Cl$ then
\begin{equation*}
\Theta_0 - 1 \leq \sigma_0(d) \implies \hat{\Theta}_n -1 \leq \sigma(d) \implies \hat{\Theta}_{n+m} - 1 \leq 3^{-\kappa (m-m_0)} \implies \Theta_{2n+m} - 1 \leq 3^{-\kappa (m-m_0)}\,.
\end{equation*}
By relabelling parameters and enlarging the constants we obtain Theorem~\ref{t.small.contrast.rate}. For details on the proofs of Theorems~\ref{t.small.contrast.rate} and~\ref{t.quenched} we refer the reader to~\cite[Section 4]{AK.HC}.

\section{Quantitative homogenization and coarse-grained parabolic estimates}
\label{s.homogenize}

In Section~\ref{s.coarse.graining} we established the qualitative well-posedness of the Cauchy-Dirichlet problem for coefficient fields~$\a\in\Omega$, and we will throughout this section understand the equation~$\partial_t u = \nabla \cdot \a\nabla u + \nabla \cdot \mathbf{f}$ in that sense. In Section~\ref{ss.coarse.estimates} we go beyond this and develop parabolic estimates depending quantitatively on the coarse-grained ellipticity constants. We use this framework in Section~\ref{ss.homogenize} to prove a black box homogenization statement and prove Theorem~\ref{t.theoremA}.

\subsection{Coarse-grained parabolic estimates}
\label{ss.coarse.estimates}

Suppose that~$I = (0,T)$ is a finite time interval,~$U\subset \mathbb{R}^d$ is a bounded Lipschitz domain and~$m\in\mathbb{Z}$ is the smallest integer such that~$I\times U \subseteq \cudot_m$. We will make the assumption that there is some~$0<c<1$ such that~$c|\cudot_m| \leq |I\times U|$ for convenience, since otherwise factors of~$\frac{|\cudot_m|}{|I\times U|}$ would appear in all of our estimates. We therefore allow constants in this section to depend on the shape of~$I\times U$, but not on its size, which is on scale~$3^m$. We will use the notation~$\partial_{\sqcup}(I\times U)$ defined in~\eqref{e.sqcup.def}, define the set of lattice points ``close" to the domain by
\begin{equation}
\label{e.def.Lstar}
\mathcal{Z}_j(I\times U) := \{z\in\mathcal{Z}_j : (z+\cudot_j) \cap (I\times U) \neq \emptyset\}\,,
\end{equation}
and for~$s \in (0,1)$,~$p \in [1,\infty)$ and~$q\in [1,\infty]$ define the semi-norm
\begin{equation}
\label{e.Besov.semi.def}
[g]_{\underline{B}_{p,q}^s(I\times U)}  := \biggl( \sum_{j=-\infty}^m 3^{-sqj} \biggl(\avsum_{z\in\mathcal{Z}_{j-1}^*} \|g-(g)_{(z+\cudot_j) \cap (I\times U)}\|_{\L^p( (z+\cudot_j) \cap (I\times U))}^p \biggr)^{\frac{q}{p}} \biggr)^{\frac{1}{q}}\,,
\end{equation}
and norm
\begin{equation}
\label{e.Besov.gen.def}
\|g\|_{\underline{B}_{p,q}^s(I\times U)} = 3^{-sm}|(g)_{I\times U}| + [g]_{\underline{B}_{p,q}^s(I\times U)}\,.
\end{equation}

In comparison to~\eqref{e.Besov.def}, our choice of the domains in~\eqref{e.Besov.semi.def} ensures that each subcube~$(z+\cudot_j) \cap (I\times U)$, with~$z\in\mathcal{Z}_{j-1}$, has a volume equal in size, up to constants depending only on the Lipschitz norm of~$U$, to~$z+\cudot_j$, while the set of all such domains is still a partition, with bounded overlaps, of~$I\times U$. By Proposition~\ref{p.Besov.equiv},~$B_{2,2}^s(I\times U) = H^{s,\nicefrac{s}{2}}(I\times U)$, where the latter space is the standard space defined in~\cite[Chapter 4, Section 2]{LMV2}. 
For~$s\in (0,1)$,~$p\in [1,\infty)$,~$q\in [1,\infty]$, and~$p',q'$ the respective H\"older conjugates, the dual norms are defined by
\begin{align}
\|f\|_{\underline{\hat{B}}_{p,q}^{-s}(I\times U)} \coloneqq \sup \bigg\{ \fint_{I\times U} fg : g\in C^\infty(I\times U)\,, \|g\|_{\underline{B}_{p',q'}^s(I\times U)} \leq 1 \bigg\} \label{e.Besov.dual.hat.def.gen}\,, \\
\|f\|_{\underline{B}_{p,q}^{-s}(I\times U)} \coloneqq \sup \bigg\{ \fint_{I \times U} fg : g\in C_c^\infty(I\times U) \,, \|g\|_{\underline{B}_{p',q'}^s(I\times U)} \leq 1 \bigg\} \label{e.Besov.dual.def.gen}\,.
\end{align}
By Lemma~\ref{l.general.dual.norm}, if~$0< s < \nicefrac12$ then there exists a constant~$C = C(I,U,s,d)$ such that
\begin{equation}
\label{e.interior.cubes}
\|f\|_{\underline{\hat{B}}_{2,2}^{-s}(I\times U)} \leq C\biggl(\sum_{k=-\infty}^n 3^{2sk} \avsum_{z\in\mathcal{Z}_k, z+\cudot_k \subseteq I\times U} | (f)_{z+\cudot_k}|^2 \biggr)^{\nicefrac{1}{2}}\,.
\end{equation}
If~$\partial_t u = \nabla \cdot \a\nabla u$ in~$I\times U$ then this implies that, analogous to Lemma~\ref{l.mathringB.bounds}, we have
\begin{align}
\|\nabla u\|_{\underline{\hat{B}}_{2,2}^{-s}(I\times U)} & \leq C 3^{sm}\lambda_{s,2}^{-\nicefrac12}(\cudot_m) \|\s^{\nicefrac12}\nabla u\|_{\L^2(I\times U)} \label{e.general.grad}\\
\|\a\nabla u\|_{\underline{\hat{B}}_{2,2}^{-s}(I\times U)} & \leq C 3^{sm}\Lambda_{s,2}^{\nicefrac12}(\cudot_m) \|\s^{\nicefrac12}\nabla u\|_{\L^2(I\times U)} \label{e.general.flux}\,.
\end{align}

Recall that by~\eqref{e.bounded.cg} our ellipticity assumption~\ref{a.ellipticity.dagger} implies boundedness of the coarse-grained ellipticity constants for~$s > \nicefrac{\gamma}{2}$. The next proposition gives an analogue of~\eqref{e.general.grad} and~\eqref{e.general.flux} for equations with non-zero right-hand side.

\begin{proposition}[Coarse-grained Poincar\'e inequality with RHS]
\label{p.poincare.RHS}
Suppose that~$\partial_t u - \nabla \cdot \a \nabla u = \nabla \cdot \mathbf{f}$ in~$I\times U$,~$0< s < \nicefrac12$ and~$0< \epsilon < s$. There exists a constant~$C = C(s,\epsilon,U,I,d)$ such that
\begin{equation}
\label{e.CG.grad}
\|\nabla u \|_{\underline{\hat{B}}_{2,2}^{-s}(I\times U)} \leq C 3^{sm} \lambda_{s-\epsilon,2}^{-\nicefrac12}(\cudot_m)\|\s^{\nicefrac12}\nabla u\|_{\L^2(I\times U)} + C3^{2sm}\lambda_{s-\epsilon,2}^{-1}(\cudot_m)\|\mathbf{f}\|_{\underline{B}_{2,2}^s(I\times U)} 
\end{equation}
and
\begin{equation}
\label{e.CG.flux}
\| \a\nabla u\|_{\underline{\hat{B}}_{2,2}^s(I\times U)} \leq C 3^{sm}\Lambda_{s,2}^{\nicefrac12}(\cudot_m) \|\s^{\nicefrac12}\nabla u \|_{\L^2(I\times U)} + C \frac{\Lambda_{s,2}^{\nicefrac12}(\cudot_m)}{\lambda_{s-\epsilon,2}^{\nicefrac12}(\cudot_m)} 3^{2sm}\|\mathbf{f}\|_{\underline{B}_{2,2}^s(I\times U)}\,.
\end{equation}
\end{proposition}

\begin{proof}
\emph{Step 1: Gradients.}
In view of~\eqref{e.interior.cubes} we need to control averages~$(\nabla u)_{z+\cudot_j}$ in interior cubes. For each subcube~$z+\cudot_k$ we fix the constant in~$\mathbf{f}$ so that it is mean-zero and let~$\rho_{z+\cudot_k}$ solve
\begin{equation*}
\left\{
\begin{aligned}
& \partial_t \rho_{z+\cudot_k} - \nabla \cdot \a\nabla \rho_{z+\cudot_k} = \nabla \cdot \mathbf{f}  & \quad \mbox{in} & \quad z+\cudot_k \\
& \rho_{z+\cudot_k} = 0 & \quad \mbox{on} & \quad \partial_{\sqcup}(z+\cudot_k)\\
\end{aligned}
\right.
\end{equation*}
This construction ensures that~$(\nabla \rho_{z+\cudot_k} )_{z+\cudot_k} = 0$, and that~$\partial_t (u - \rho_{z+\cudot_k}) =\nabla \cdot \a \nabla ( u - \rho_{z+\cudot_k} )$. This means that~$(\nabla u)_{z+\cudot_k} = (\nabla u - \nabla \rho_{z+\cudot_k} )_{z+\cudot_k}$, and we can apply the coarse-graining inequality~\eqref{e.energymaps}. From testing the equation for~$\rho(z+\cu_k)$ with itself we have
\begin{align*}
\|\s^{\nicefrac12}\nabla \rho_{z+\cudot_k}\|_{\L^2(z+\cudot_k)} & \leq \|\mathbf{f}\|_{\underline{B}_{2,2}^s(z+\cudot_k)}^{\nicefrac12}\|\nabla \rho_{z+\cudot_k}\|_{\underline{\hat{B}}_{2,2}^{-s}(z+\cudot_k)}^{\nicefrac12} \\
& \leq C\|\mathbf{f}\|_{\underline{B}^{s}_{2,2}(z+\cudot_k)}^{\nicefrac12}\|\nabla u\|_{\underline{\hat{B}}^{-s}_{2,2}(z+\cudot_k)}^{\nicefrac12} +  C\|\mathbf{f}\|_{\underline{B}_{2,2}^{s}(z+\cudot_k)}^{\nicefrac12}\|\nabla u-\nabla\rho_{z+\cudot_k}\|_{\underline{\hat{B}}_{2,2}^{-s}(z+\cudot_k)}^{\nicefrac12}\,.
\end{align*}
Now by~\eqref{e.general.grad},
\begin{align*}
\|\nabla u-\nabla\rho_{z+\cudot_k}\|_{\underline{B}_{2,2}^{-s}(z+\cudot_k)} \leq C 3^{sk}\lambda_s^{-\nicefrac12}(z+\cudot_k)\|\s^{\nicefrac12}(\nabla u-\nabla \rho_{z+\cudot_k})\|_{\L^2(z+\cudot_k)}\,.
\end{align*}
Then combining the above,
\begin{align*}
\|\s^{\nicefrac12}(\nabla u -\nabla\rho_{z+\cudot_k})\|_{\L^2(z+\cudot_k)}
& \leq
\|\s^{\nicefrac12}\nabla u\|_{\L^2(z+\cudot_k)} + \|\s^{\nicefrac12}\nabla \rho_{z+\cudot_k}\|_{\L^2(z+\cudot_k)}
\\ & \leq
\|\s^{\nicefrac12}\nabla u\|_{\L^2(z+\cudot_k)} + C\|\mathbf{f}\|_{\underline{B}_{2,2}^s(z+\cudot_k)}^{\nicefrac12}\|\nabla u\|_{\underline{\hat{B}}^{-s}_{2,2}(z+\cudot_k)}^{\nicefrac12}
\\ & \qquad
+ C\|\mathbf{f}\|_{\underline{B}_{2,2}^s(z+\cudot_k)}^{\nicefrac12} 3^{\frac{sk}{2}}\lambda_s^{-\nicefrac{1}{4}}(z+\cudot_k) \|\s^{\nicefrac12}\nabla (u-\rho_{z+\cudot_k})\|_{\L^2(z+\cudot_k)}^{\nicefrac12}\,.
\end{align*}
Using Young's inequality to re-absorb the energy factor in the last term we obtain
\begin{align}
\label{e.energy.error}
\|\s^{\nicefrac12}\nabla (u -\rho_{z+\cudot_k})\|_{\L^2(z+\cudot_k)} & \leq
C \|\s^{\nicefrac12}\nabla u\|_{\L^2(z+\cudot_k)} + C\|\mathbf{f}\|_{\underline{B}_{2,2}^s(z+\cudot_k)}^{\nicefrac12}\|\nabla u\|_{\underline{\hat{B}}^{-s}_{2,2}(z+\cudot_k)}^{\nicefrac12}
\notag \\ & \qquad
+ C\|\mathbf{f}\|_{\underline{B}_{2,2}^s(z+\cudot_k)} 3^{sk} \lambda_s^{-\nicefrac{1}{2}}(z+\cudot_k) \,.
\end{align}
Then using the definition of the dual norm,~\eqref{e.general.grad},~\eqref{e.energy.error}, and Young's inequality, for any~$\delta_k > 0$
\begin{align*}
\lefteqn{
3^{sk}|(\nabla u - \nabla \rho_{z+\cudot_k})_{z+\cudot_k}| 
} \qquad \qquad  & \\ 
& \leq \|\nabla u - \nabla \rho_{z+\cudot_k}\|_{\underline{\hat{B}}^{-s}_{2,2}(z+\cudot_k)} \\
& \leq  C 3^{sk}\lambda_s^{-\nicefrac12}(z+\cudot_k) \|\s^{\nicefrac12}\nabla (u-\rho_{z+\cudot_k})\|_{\L^2(z+\cudot_k)} \\
& \leq C 3^{sk}\lambda_s^{-\nicefrac12}(z+\cudot_k)\|\s^{\nicefrac12}\nabla u\|_{\L^2(z+\cudot_k)} + C 3^{sk}\lambda_s^{-\nicefrac12}(z+\cudot_k)\|\mathbf{f}\|_{\underline{B}_{2,2}^s(z+\cudot_k)}^{\nicefrac12} \|\nabla u\|_{\underline{\hat{B}}^{-s}_{2,2}(z+\cudot_k)}^{\nicefrac12} \\
& \qquad + C 3^{2sk}\lambda_s^{-1}(z+\cudot_k)\|\mathbf{f}\|_{\underline{B}_{2,2}^s(z+\cudot_k)} \\
& \leq C 3^{sk}\lambda_s^{-\nicefrac12}(z+\cudot_k)\|\s^{\nicefrac12}\nabla u\|_{\L^2(z+\cudot_k)} + \delta^{-1}_k C 3^{2sk}\lambda_s^{-1}(z+\cudot_k)\|\mathbf{f}\|_{\underline{B}_{2,2}^s(z+\cudot_k)} \\
& \qquad + \delta_k \|\nabla u\|_{\underline{\hat{B}}^{-s}_{2,2}(z+\cudot_k)}\,.
\end{align*}
We make the choice~$\delta_k = \delta 3^{\epsilon(k-m)}$ so that, adding in the zero-mean term~$(\nabla \rho_{z+\cudot_k})_{z+\cudot_k}$,
\begin{align*}
\|\nabla u\|_{\underline{\hat{B}}_{2,2}^{-s}(I\times U)}^2 & \leq C\sum_{k=-\infty}^m 3^{2sk} \avsum_{z\in\mathcal{Z}_k, z+\cudot_k \subseteq I\times U} | (\nabla u - \nabla \rho_{z+\cudot_k})_{z+\cudot_k}|^2  \\
& \leq C \|\s^{\nicefrac12}\nabla u\|_{\L^2(I\times U)}^2 \sum_{k=-\infty}^m  \sup_{z \in \mathcal{Z}_k \cap \cudot_m} 3^{2sk}\lambda_s^{-1}(z+\cudot_k)  \\
& \qquad + C \delta^{-2} \|\mathbf{f}\|_{\underline{B}_{2,2}^s (I\times U)}^2 \sum_{k=-\infty}^m   \sup_{z \in \mathcal{Z}_k \cap \cudot_m} 3^{-2\epsilon(k-m)} 3^{4sk}\lambda_s^{-2}(z+\cudot_k) \\
& \qquad + C \sum_{k=-\infty}^m \delta^2 \avsum_{z\in\mathcal{Z}_k, z+\cudot_k \subseteq I\times U} 3^{2\epsilon(k-m)}\|\nabla u\|_{\underline{\hat{B}}^{-s}_{2,2}(z+\cudot_k)}^2\,.
\end{align*}
We control the ellipticity factor by bounding
\begin{align*}
\lefteqn{
\sum_{k=-\infty}^m \sup_{z \in \mathcal{Z}_k \cap \cudot_m} 3^{-2\epsilon(k-m)} 3^{4sk}\lambda_s^{-2}(z+\cudot_k)
} \qquad & \\
& \leq \biggl( \sum_{k=-\infty}^m \sup_{z \in \mathcal{Z}_k \cap \cudot_m} 3^{-\epsilon(k-m)} 3^{2sk}\lambda_{s-\epsilon}^{-1}(z+\cudot_k) \biggr)^2 \\
& = \biggl( \sum_{k=-\infty}^m \sup_{z \in \mathcal{Z}_k \cap \cudot_m} 3^{-\epsilon(k-m)} 3^{2sk} \sum_{j=-\infty}^k 3^{2(s-\epsilon)(j-k)} \sup_{z'\in\mathcal{Z}_j \cap z+\cudot_k} |\s_*^{-1}(z'+\cudot_j)| \biggr)^2 \\
& \leq 3^{sm} \biggl( \sum_{j=-\infty}^m 3^{2(s-\epsilon)(j-m)} \sup_{z'\in\mathcal{Z}_j \cap \cudot_m} |\s_*^{-1}(z+\cudot_j)| \sum_{k=j}^m 3^{\epsilon (k-m)} \biggr)^2 \\
& \leq C\epsilon^{-1} 3^{sm} \lambda_{s-\epsilon,2}^{-2}(\cudot_m)\,,
\end{align*}
with the other ellipticity factor bounded similarly. Finally,
\begin{align*}
\lefteqn{
\sum_{k=-\infty}^m \avsum_{z\in\mathcal{Z}_k, z+\cudot_k \subseteq I\times U} 3^{2\epsilon (k-m)}\|\nabla u\|_{\underline{\hat{B}}^{-s}_{2,2}(z+\cudot_k)}^2
} \qquad & \\
& \leq \sum_{k=-\infty}^m \delta^2 \avsum_{z\in\mathcal{Z}_k, z+\cudot_k \subseteq I\times U} 3^{2\epsilon (k-m)} \sum_{j=-\infty}^k 3^{2sj} \avsum_{z'\in\mathcal{Z}_j \cap z+\cudot_k} | (\nabla u)_{z' + \cudot_j}|^2 \\
& \leq C \epsilon^{-1} \| \nabla u\|_{\underline{\hat{B}}_{2,2}^{-s}(I\times U)}^2 \,.
\end{align*}
We then re-absorb the last term by taking~$\delta$ small enough.

\smallskip

\emph{Step 2: Fluxes.} For the fluxes, if~$z = (z^0,z') \in \mathcal{Z}_k$ we fix the constant in~$\mathbf{f}$ so that it is mean zero in the cube and let~$\chi_{z+\cudot_k}$ be the unique solution to
\begin{equation*}
\left\{
\begin{aligned}
& \partial_t \chi_{z+\cudot_k} - \nabla \cdot \a\nabla \chi_{z+\cudot_k} = \nabla \cdot \mathbf{f}  & \quad \mbox{in}  \quad z+\cudot_k \\
& \hat{n}\cdot (\a\nabla\chi_{z+\cudot_k}  +\mathbf{f}) = 0 & \quad \mbox{on}  \quad (z_0 + I_k) \times \partial (z' + \cu_k)\\
& \chi(z+\cudot_k) & \quad   \mbox{periodic in time}
\end{aligned}
\right.
\end{equation*}
Then testing the equation with the function~$x \mapsto x_i$
\begin{equation*}
\int_{z+\cudot_k} \a\nabla \chi_{z+\cudot_k} = \int_{z'+\cu_k} (\chi_{z+\cudot_k}(\cdot,x) - \chi(\cudot_k)(\cdot,x) )\biggr|_{z^0 +I_k} x_i \, dx= 0\,.
\end{equation*}
Now testing the equation for~$\chi_{z+\cudot_k}$ with itself and then applying~\eqref{e.CG.grad},
\begin{align*}
\|\s^{\nicefrac12}\nabla \chi_{z+\cudot_k}\|_{\L^2(z+\cudot_k)} & \leq \|\mathbf{f}\|_{\underline{B}_{2,2}^s(z+\cudot_k)} \|\nabla \chi_{z+\cudot_k}\|_{\underline{B}^{-s}_{2,2}(z+\cudot_k)} \\
& \leq C 3^{sk}\lambda_{s-\epsilon,2}^{-\nicefrac12}(z+\cudot_k) \|\s^{\nicefrac12}\nabla \chi_{z+\cudot_k}\|_{\L^2(z+\cudot_k)}\|\mathbf{f}\|_{\underline{B}_{2,2}^s(z+\cudot_k)}  \\
& \qquad + C3^{2sk} \lambda_{s-\epsilon,2}^{-1}(z+\cudot_k) \|\mathbf{f}\|_{\underline{B}_{2,2}^s(z+\cudot_k)}^2\,,
\end{align*}
so that after re-absorbing the energy factor,
\begin{align*}
\|\s^{\nicefrac12}\nabla \chi_{z+\cudot_k}\|_{\L^2(z+\cudot_k)} \leq C3^{sk}\lambda_{s-\epsilon,2}^{-\nicefrac12}(z+\cudot_k) \|\mathbf{f}\|_{\underline{B}_{2,2}^s(z+\cudot_k)}\,.
\end{align*}
Combining~\eqref{e.energymaps.flux} with the above display,
\begin{align*}
3^{sk}| (\a\nabla u - \a\nabla \chi_{z+\cudot_k})_{z+\cudot_k} | & \leq 3^{sk}|\b^{\nicefrac12}(z+\cudot_k)| \|\s^{\nicefrac12}\nabla ( u - \chi_{z+\cudot_k} ) \|_{\L^2(z+\cudot_k)} \\
& \leq 3^{sk}|\b^{\nicefrac12}(z+\cudot_k)| \|\s^{\nicefrac12}\nabla  u \|_{\L^2(z+\cudot_k)} \\
& \qquad + C3^{2sk}|\b^{\nicefrac12}(z+\cudot_k)| \lambda_{s-\epsilon,2}^{-\nicefrac12}(z+\cudot_k) \|\mathbf{f}\|_{\underline{B}_{2,2}^s(z+\cudot_k)}\,.
\end{align*}
Using Lemma~\ref{l.general.dual.norm},~$(\a\nabla u - \a\nabla \chi_{z+\cudot_k})_{z+\cudot_k} = (\a\nabla u)_{z+\cudot_k}$ and the above display,
\begin{align*}
\| \a\nabla u\|_{\underline{\hat{B}}_{2,2}^{-s}(I\times U)}^2 & \leq C \sum_{k=-\infty}^m \avsum_{z+\cudot_k \subseteq I\times U} 3^{2sk}| (\a\nabla u - \a\nabla \chi_{z+\cudot_k})_{z+\cudot_k} |^2  \\
& \leq C 3^{2sm}\|\s^{\nicefrac12}\nabla u\|_{\L^2(I\times U)}^2 \Lambda_{s,2}(\cudot_m) \\
& \qquad  +  C \|\mathbf{f}\|_{\underline{B}_{2,2}^s(I\times U)}^2 \sum_{k=-\infty}^m 3^{2sk}\sup_{z\in \mathcal{Z}_k\cap \cudot_m}|\b(z+\cudot_k)| 3^{2sk}\sup_{z\in\mathcal{Z}_k\cap \cudot_m} \lambda_{s-\epsilon,2}^{-1}(z+\cudot_k) \\
& \leq C\|\s^{\nicefrac12}\nabla u\|_{\L^2(I\times U)}^2 3^{2sm}\Lambda_{s,2}(\cudot_m) + C 3^{4sm}\|\mathbf{f}\|_{\underline{B}_{2,2}^s(I\times U)}^2\Lambda_{s,2}(\cudot_m) \lambda_{s-\epsilon,2}^{-1}(\cudot_m)\,.
\end{align*}
where we use that since~$\mathbf{f}$ is zero mean in each subcube, we really have semi-norms which are bounded by
\begin{equation*}
\avsum_{z\in\mathcal{Z}_k, z+\cudot_k \subseteq (I\times U)} [\mathbf{f}-(\mathbf{f})_{z+\cudot_k}]^2_{\underline{B}_{2,2}^s(z+\cudot_k} \leq \|\mathbf{f}\|_{\underline{B}_{2,2}^s(I\times U)}^2\,,
\end{equation*}
by~\eqref{e.fractional.integral}.
\end{proof}

The next proposition establishes the basic energy estimate for parabolic equations, but with dependence on the coarse-grained ellipticity constants. The spatial boundary data~$g$ is effectively treated as a right-hand side in the equation, which explains the somewhat complicated dependence on~$g$. In our context what matters is the dependence on the coarse-grained ellipticity constants.

\begin{proposition}[Coarse-grained energy estimate]
\label{p.cg.energy}
Let~$I = (0,T)$ be a finite interval,~$U$ a\newline bounded Lipschitz domain,~$s\in (0,\nicefrac12)$,~$\mathbf{f}\in B^s_{2,2}(I\times U)^d$,~$u_0 \in L^2(U)$, and~$g$ such that the norm on the right-hand side of~\eqref{e.cg.energy} is finite. There exists a constant~$C = C(I,U,s,d)$ such that if~$u \in W_{\s}^1(I \times U)$ solves
\begin{equation}
\label{e.parabolic.eq}
\left\{
\begin{aligned}
& \partial_t u - \nabla \cdot \a\nabla u = \nabla \cdot \mathbf{f}  & \quad \mbox{in} & \quad I\times U \\
& u = g & \quad \mbox{on} & \quad I \times \partial U \\
& u = u_0 & \quad \mbox{on} & \quad  \{0\} \times U \\
\end{aligned}
\right.
\end{equation}
then
\begin{align}
\label{e.cg.energy}
\lefteqn{
|I|^{-\nicefrac12} \sup_{t\in I} \|u(t)\|_{\L^2(U)} + \|\s^{\nicefrac12}\nabla u\|_{\L^2(I\times U)}
} \qquad & \notag \\
& \leq  C3^{sm}\lambda_{s-\epsilon,2}^{-\nicefrac12}(\cudot_m)\|\mathbf{f}\|_{\underline{B}^s_{2,2}(I\times U)} + C|I|^{-\nicefrac12} \|u_0\|_{\L^2(U)} +  C|I|^{-\nicefrac12} \sup_{t\in I} \|g(t)\|_{\L^2(U)} \notag \\
& \qquad + C 3^{sm}\Lambda_{s,2}^{\nicefrac12}(\cudot_m)\|\nabla g\|_{\underline{B}_{2,2}^s(I\times U)} + C3^{sm} \lambda_{s,2}^{-\nicefrac12}(\cudot_m)  \|\nabla \Delta_U^{-1} \partial_t g\|_{\underline{B}_{2,2}^s(I\times U)} \notag \\
& \qquad + C\|\nabla g\|_{\L^2(I\times U)}^{\nicefrac12} \|\partial_t g \|^{\nicefrac12}_{\L^2(I;\underline{H}^{-1}(U))}\,.
\end{align}
\end{proposition}
\begin{proof}
Let~$u = v+w$ where~$v$ solves~\eqref{e.parabolic.eq} with~$\mathbf{f} = 0$ and~$u_0 = 0$, and~$w$ solves~\eqref{e.parabolic.eq} with~$g = 0$.

\emph{Step 1:} By testing the equation for~$w$ with itself and applying~\eqref{e.CG.grad},
\begin{align*}
\frac{1}{2|I|}\sup_{t \in I} \fint_U w^2(t) + \fint_I \fint_U \nabla w \cdot \s \nabla w & \leq \fint_I \fint_U |\mathbf{f}\cdot \nabla w| + \frac{1}{2|I|} \fint_U u_0^2 \\
& \leq \|\mathbf{f}\|_{\underline{B}_{2,2}^s(I\times U)} \|\nabla w\|_{\underline{\hat{B}}^{-s}_{2,2}(I\times U)} + \frac{1}{2|I|} \fint_U u_0^2 \\
& \leq C3^{sm}\lambda_{s-\epsilon,2}^{-\nicefrac12}(\cudot_m)\|\mathbf{f}\|_{\underline{B}_{2,2}^s(I\times U)}\|\s^{\nicefrac12}\nabla w\|_{\L^2(I\times U)} \\
& \qquad + C3^{2sm}\lambda_{s-\epsilon,2}^{-1}(\cudot_m)\|\mathbf{f}\|_{\underline{B}_{2,2}^s(I\times U)}^2 + \frac{1}{2|I|} \fint_U u_0^2\,.
\end{align*}
Re-absorbing the energy term,
\begin{align*}
\|\s^{\nicefrac12} \nabla w\|_{\L^2(I\times U)} \leq C3^{sm}\lambda_{s-\epsilon,2}^{-\nicefrac12}(\cudot_m)\|\mathbf{f}\|_{\underline{B}^s_{2,2}(I\times U)} + C|I|^{-\nicefrac12} \|u_0\|_{\L^2(U)}\,,
\end{align*}
and using this in the first estimate (from testing the equation), we also get the same bound for~$|I|^{-\nicefrac12}\sup_{t\in I} \|w(t)\|_{\L^2(U)}$.

\emph{Step 2:} Now consider the equation for~$v$, and write~$\partial_t g = \nabla \cdot \nabla \Delta^{-1}_U \partial_t g$, where~$\Delta^{-1}_U$ inverts the Laplacian with zero Dirichlet data. Then~$v-g$ satisfies
\begin{equation*}
\partial_t (v-g) = \nabla \cdot \a\nabla v - \nabla \cdot \nabla \Delta_U^{-1} \partial_t g\,.
\end{equation*}
Testing this equation with~$v-g$ yields
\begin{align*}
\fint_I \fint_U \frac{1}{2}\partial_t(v-g)^2 + \nabla v \cdot \s \nabla v = \fint_I \fint_U \a\nabla v \cdot \nabla g + \nabla \Delta_U^{-1} \partial_t g \cdot \nabla (v-g)\,.
\end{align*}
Then estimating the integrals by duality and applying~\eqref{e.general.grad} and~\eqref{e.general.flux}
\begin{align*}
\lefteqn{
\frac{1}{2|I|}\sup_{t\in I} \fint_U (v-g)^2(t) + \fint_I \fint_U \nabla v \cdot \s \nabla v 
} \qquad & \\
& \leq \|\a\nabla v \|_{\underline{\hat{B}}^{-s}_{2,2}(I\times U)} \|\nabla g\|_{\underline{B}_{2,2}^s(I\times U)} + \|\nabla v \|_{\underline{\hat{B}}^{-s}_{2,2}(I\times U)} \|\nabla \Delta_U^{-1} \partial_t g\|_{\underline{B}_{2,2}^s(I\times U)} \\
& \qquad + \|\nabla g\|_{\L^2(I\times U)}\| \nabla \Delta_U^{-1} \partial_t g\|_{\L^2(I\times U)} \\
& \leq C3^{sm}\Lambda_{s,2}^{\nicefrac12}(\cudot_m)\|\s^{\nicefrac12}\nabla v\|_{\L^2(I\times U)} \|\nabla g\|_{\underline{B}_{2,2}^s(I\times U)}  \\
& \qquad + C3^{sm}\lambda_{s,2}^{-\nicefrac12}(\cudot_m) \|\s^{\nicefrac12}\nabla v\|_{\L^2(I\times U)}\|\nabla \Delta_U^{-1} \partial_t g\|_{\underline{B}_{2,2}^s(I\times U)} + \|\nabla g\|_{\L^2(I\times U)} \|\partial_t g \|_{\L^2(I;\underline{H}^{-1}(U))}
\end{align*}
where the last term was bounded by~\cite[Remark 3.5]{ABM}. Therefore after re-absorbing the energy factor
\begin{align*}
\|\s^{\nicefrac12}\nabla v\|_{\L^2(I\times U)} & \leq C 3^{sm}\Lambda_{s,2}^{\nicefrac12}(\cudot_m)\|\nabla g\|_{\underline{B}_{2,2}^s(I\times U)} + C3^{sm} \lambda_{s,2}^{-\nicefrac12}(\cudot_m)  \|\nabla \Delta_U^{-1} \partial_t g\|_{\underline{B}_{2,2}^s(I\times U)} \\
& \qquad + C\|\nabla g\|_{\L^2(I\times U)}^{\nicefrac12} \|\partial_t g \|^{\nicefrac12}_{\L^2(I;\underline{H}^{-1}(U))}\,.
\end{align*}
We then obtain~\eqref{e.cg.energy} by putting together all our estimates and using the triangle inequality.
\end{proof}

In Lemma~\ref{l.Besov.regularity.solns} we proved a coarse-grained Poincar\'e inequality where the coarse-grained ellipticity constants replace the standard point-wise ellipticity factors. We can do the same for the Caccioppoli inequality. The proof is almost identical to the elliptic case in~\cite[Proposition 2.5]{AK.HC}. The statement below is not formulated in diffusively scaled domains -- see Corollary~\ref{c.adapted.caccio}.

\begin{proposition}[Coarse-grained Caccioppoli inequality]
\label{p.coarse.caccio}
Let~$s,t\in (0,1)$ such that~$s+t <1$ and suppose that~$u\in \mathcal{A}(\cudot_m)$. Then
\begin{align}
\label{e.coarse.caccio}
& \|\s^{\nicefrac12} \nabla u\|_{\L^2(\cudot_{m-1})}^2 \nonumber \\
& \leq \biggl(\frac{C(d)}{1-s-t}\biggr)^{2 + \frac{4s}{1-s-t}}\biggl( 1 + (\Lambda_{s,1}(\cudot_m))^{\frac{1-t}{1-s-t}} \bigl( \lambda^{-1}_{t,1}(\cudot_m) + \Lambda_{t,1}(\cudot_m) \bigr)^{\frac{s}{1-s-t}}\biggr) 3^{-2m}\|u\|_{\L^2(\cudot_m)}^2\,.
\end{align}
\end{proposition}
\begin{proof}
By scaling we may take~$m=1$ and suppose that~$u\in\mathcal{A}(\cudot_1)$. Fix~$s,t \in (0,1)$ such that~$s+t < 1$, let~$0 \leq r < R < 1$ and for any length scale~$\xi$ let~$\nu_{\xi} := \log_3(\xi)$. We will compare the Dirichlet energy of~$u$ in~$\cudot_{\nu_r}$ to that in~$\cudot_{\nu_R}$ in such a way that iteration (via \cite[Lemma C.6]{AKMBook}) yields the desired inequality.

For appropriate constants depending only on dimension take~$c(R-r) \leq 3^{k} \leq C(R-r)$ and fix~$\varphi \in C^\infty(\cudot_0)$ such that
\begin{equation}
\label{e.cutoff.size}
\left\{
\begin{aligned}
&\varphi \equiv 1 \quad \mbox{ on } \cudot_{\nu_r}\,, \quad \supp(\varphi) \subseteq \cudot_{\nu_{\nicefrac{(r+R)}{2}}} \\
& \|\shom^{\nicefrac12} \nabla^j\varphi\|_{L^\infty(\cudot_0)} \leq C3^{-kj} \quad \mbox{ for } \quad j\in \{1,2\} \\
& \|\partial_t \varphi\|_{L^\infty(\cudot_0)} \leq C3^{-2k}
\end{aligned}
\right. \,.
\end{equation}
By Lemma~\ref{l.testing.with.u} we can test the equation for~$u$ with~$u\varphi$ to get
\begin{align}
\label{e.caccio.terms}
\fint_{\cudot_0} \varphi \nabla u \cdot \a \nabla u & = - \fint_{\cudot_0} u \nabla \varphi \cdot \a\nabla u + \frac{1}{2}\fint_{\cudot_0} \partial_t \varphi u^2 \nonumber \\
& = - \avsum_{z\in\mathcal{Z}_h \cap \cudot_0} \fint_{z+\cudot_h} (u - (u)_{z+\cudot_h})\nabla \varphi \cdot \a\nabla u + (u)_{z+\cudot_h} \fint_{z+\cudot_h} \nabla \varphi \cdot \a \nabla u \nonumber \\
& \hspace*{1cm}+ \frac{1}{2}\fint_{\cudot_0} \partial_t \varphi u^2\,.
\end{align}
where~$h \leq k - 4$ is a scale much smaller than~$k$ which, when fixed below, will allow us to gain from scale discount factors. We bound the first term in~\eqref{e.caccio.terms} by
\begin{align*}
& \fint_{z+\cudot_h} (u - (u)_{z+\cudot_h})\nabla \varphi \cdot \a\nabla u  \leq \| (u - (u)_{z+\cudot_h})\nabla \varphi\|_{\underline{B}_{2,\infty}^s(z+\cudot_h)}[\a\nabla u]_{\underline{\hat{B}}_{2,1}^{-s}(z+\cudot_h)} \\
& \leq C(d)\biggl(\|\nabla \varphi\|_{\underline{W}^{1,\infty}_{\mathrm{par}}(z+\cudot_h)} ( [\nabla u]_{\Bring^{-(1-s)}(z+\cudot_{h+1})} + [\a\nabla u]_{\Bring_{2,1}^{s-1}(z+\cudot_{h+1})} ) \biggr) [\a\nabla u]_{\Bring_{2,1}^{-s}(z+\cudot_h)}
\end{align*}
using~\eqref{e.other.div.curl}. Our assumptions in~\eqref{e.cutoff.size} and~$h\leq k$ imply that~$\|\nabla \varphi\|_{\underline{W}^{1,\infty}_{\mathrm{par}}(z+\cudot_h)} \leq 3^{-k}$, while using Lemma~\ref{l.mathringB.bounds} to bound the negative Besov semi-norms,
\begin{align*}
\fint_{z+\cudot_h} & (u - (u)_{z+\cudot_h})\nabla \varphi \cdot \a\nabla u \\
& \leq C(d) 3^{-k} 3^{(1-s)h}\bigl(\lambda_{1-s,1}^{-\nicefrac12}(z+\cudot_{h+1}) + \Lambda_{1-s,1}^{\nicefrac12}(z+\cudot_{h+1}) \bigr) 3^{sh}\Lambda_{s,1}^{\nicefrac12}(z+\cudot_h) \|\s^{\nicefrac12}\nabla u\|_{\L^2(z+\cudot_{h+1})}^2 \\
& \leq C(d) 3^{-k+(1-s-t) h} \bigl( \lambda^{-\nicefrac12}_{t,1}(\cudot_0) + \Lambda^{\nicefrac12}_{t,1}(\cudot_0) \bigr)\Lambda^{\nicefrac12}_{s,1}(\cudot_0) \|\s^{\nicefrac12}\nabla u\|_{\L^2(z+\cudot_{h+1})}^2\,,
\end{align*}
where we have used
\begin{equation*}
\left\{
\begin{aligned}
& \lambda_{1-s,1}^{-\nicefrac12}(z+\cudot_{h}) \leq \lambda_{t,1}^{-\nicefrac12}(z+\cudot_{h}) \leq 3^{-th}\lambda_{t,1}^{-\nicefrac12}(z+\cudot_{h}) \\
& \Lambda_{1-s,1}^{\nicefrac12}(z+\cudot_{h}) \leq \Lambda_{t,1}^{\nicefrac12}(z+\cudot_{h}) \leq 3^{-th}\Lambda_{t,1}^{\nicefrac12}(z+\cudot_{h})
\end{aligned}
\right.
\end{equation*}
to bound the coarse-grained ellipticity constants. The second term in~\eqref{e.caccio.terms} is bounded similarly as
\begin{align*}
(u)_{z+\cudot_h} \fint_{z+\cudot_h} \nabla \varphi \cdot \a \nabla u & \leq \|u\|_{\L^2(z+\cudot_h)}\|\nabla \varphi\|_{\underline{B}_{2,\infty}^s(z+\cudot_h)} [\a\nabla u]_{\underline{\hat{B}}_{2,1}^{-s}(z+\cudot_h)} \\
& \leq \|u\|_{\L^2(z+\cudot_h)} 3^{-k-sh}3^{sh}\Lambda_{s,1}^{\nicefrac12}(z+\cudot_h)\|\s^{\nicefrac12}\nabla u\|_{\L^2(z+\cudot_h)} \\
& \leq 3^{-k-sh}\Lambda_{s,1}^{\nicefrac12}(\cudot_0)\|u\|_{\L^2(z+\cudot_h)} \|\s^{\nicefrac12}\nabla u\|_{\L^2(z+\cudot_h)}\,.
\end{align*}
Bounding the third term in~\eqref{e.caccio.terms} trivially by~$3^{-2k}\|u\|_{\L^2(\cudot_0)}$ and summing the other terms over~$z\in \mathcal{Z}_h$ such that~$z+\cudot_h$ intersects the support of~$\varphi$, we obtain
\begin{align*}
\|\s^{\nicefrac12}\nabla u\|^2_{\L^2(\cudot_{\nu_r})} & \leq C(d)3^{-k+(1-s-t)h} \bigl( \lambda^{-\nicefrac12}_{t,1}(\cudot_0) + \Lambda^{\nicefrac12}_{t,1}(\cudot_0) \bigr)\Lambda^{\nicefrac12}_{s,1}(\cudot_0) \|\s^{\nicefrac12}\nabla u\|_{\L^2(\cudot_{\nu_R})}^2 \\
& \quad + C(d) 3^{-k-sh}\Lambda_{s,1}^{\nicefrac12}(\cudot_0)\|u\|_{\L^2(\cudot_0)} \|\s^{\nicefrac12}\nabla u\|_{\L^2(\cudot_{\nu_R})} \\
& \quad + 3^{-2k}\|u\|^2_{\L^2(\cudot_0)}\,.
\end{align*}
We now choose~$h$ small enough that the first term has a factor of~$\nicefrac12$, namely
\begin{equation*}
h :=  \bigg\lfloor \frac{1}{1-s-t}\log_3 \biggl( \frac{3^k}{C(d)\bigl( \lambda^{-\nicefrac12}_{t,1}(\cudot_0) + \Lambda^{\nicefrac12}_{t,1}(\cudot_0) \bigr)\Lambda^{\nicefrac12}_{s,1}(\cudot_0)} \biggr) \bigg\rfloor
\end{equation*}
so we get
\begin{align*}
\|\s^{\nicefrac12} & \nabla u\|^2_{\L^2(\cudot_{\nu_r})} \leq \frac{1}{2}\|\s^{\nicefrac12}\nabla u\|^2_{\L^2(\cudot_{\nu_R})}  \\
& \quad + \biggl(\frac{1}{R-r} \biggr)^{\frac{2(1-t)}{1-s-t}}\biggl[ 1 + C(d)^{\frac{2s}{1-s-t}}(\Lambda_{s,1}(\cudot_0))^{\frac{2(1-t)}{1-s-t}} \bigl( \lambda^{-\nicefrac12}_{t,1}(\cudot_0) + \Lambda^{\nicefrac12}_{t,1}(\cudot_0) \bigr)^{\frac{2s}{1-s-t}}\biggr] \|u\|^2_{\L^2(\cudot_0)}\,.
\end{align*}
We complete the proof by applying~\cite[Lemma C6]{AK.Book}.
\end{proof}

Choosing the proper parabolic scaling for the Caccioppoli inequality is non-obvious, because there is a feedback loop: the ellipticity depends on the domains in which we coarse-grain, and the domains we choose to coarse-grain are adapted to the ellipticity. The following corollary illustrates one possible choice. The coarse-grained ellipticity constants in adapted domains are defined in~\eqref{e.adapted.Lambda} and~\eqref{e.adapted.lambda}, and Lemma~\ref{l.change.variables} states how the coarse-grained matrices behave under a change of variables.

\begin{corollary}[Caccioppoli in parabolic domains]
\label{c.adapted.caccio}
Suppose that~$s\in (0,\nicefrac12)$, the adapted cubes are defined by~$\m_0$ such that
\begin{equation}
\label{e.good.bound}
\lambda_{s,1}(\cusdot_m) \leq 1\,.
\end{equation}
and~$\partial_t u = \nabla \cdot \a \nabla u$ in~$\cusdot_m$. Then
\begin{align}
\label{e.coarse.caccio.adapted}
\lambda_{\m_0}^{-\nicefrac12}\|\s^{\nicefrac12} \nabla u\|_{\L^2(\cusdot_{m-1})}^2 \leq C(d,s) \biggl( \frac{\Lambda_{s,1}(\cusdot_m)}{\lambda_{s,1}(\cusdot_m)} \biggr)^{\frac{1}{1-2s}} 3^{-2m}\|u\|_{\L^2(\cusdot_m)}^2\,.
\end{align}
\end{corollary}
\begin{proof}
By transforming as in~\eqref{e.transform.eq}, applying Proposition~\ref{p.coarse.caccio} with~$s=t$ and transforming back, we obtain
\begin{align*}
\lefteqn{
\lambda_{\m_0}^{-\nicefrac12}\|\s^{\nicefrac12}\nabla u\|_{\L^2(\cusdot_m)}
} \qquad & \\
& \leq C(d,s)\biggl( 1 + (\Lambda_{s,1}(\cusdot_m))^{\frac{1-s}{1-2s}} \bigl( \lambda^{-1}_{s,1}(\cusdot_m) + \Lambda_{s,1}(\cusdot_m) \bigr)^{\frac{s}{1-2s}}\biggr) 3^{-2m}\|u\|_{\L^2(\cudot_m)}^2\,.
\end{align*}
The ellipticity factor is re-arranged, using~\eqref{e.good.bound}, as
\begin{align*}
\lefteqn{
(\Lambda_{s,1}(\cusdot_m))^{\frac{1-s}{1-2s}} \bigl(\lambda^{-1}_{s,1}(\cusdot_m) + \Lambda_{s,1}(\cusdot_m) \bigr)^{\frac{s}{1-2s}}
} \qquad & \\
& \leq \lambda_{s,1}(\cusdot_m) \biggl( \frac{\Lambda_{s,1}(\cusdot_m)}{\lambda_{s,1}(\cusdot_m)}\biggr)^{\frac{1-s}{1-2s}} \biggl( 1 + \frac{\Lambda_{s,1}(\cusdot_m)}{\lambda_{s,1}(\cusdot_m)} \lambda_{s,1}^2(\cusdot_m)\biggr)^{\frac{s}{1-2s}} \\
& \leq \biggl( \frac{\Lambda_{s,1}(\cusdot_m)}{\lambda_{s,1}(\cusdot_m)} \biggr)^{\frac{1}{1-2s}}\,,
\end{align*}
which concludes the proof.
\end{proof}

Finally, we state the coarse-grained Caccioppoli inequality for equations with non-zero right-hand side.

\begin{proposition}[Coarse-grained Caccioppoli with RHS]
Let~$s,t \in (0,1)$ such that~$s+t < 1$, and suppose that~$\partial_t u = \nabla \cdot \a \nabla u + \nabla \cdot \mathbf{f}$ in~$\cudot_m$. There exists a constant~$C = C(d)$ such that
\begin{align*}
\lefteqn{
\|\s^{\nicefrac12}\nabla u\|_{\L^2(\cudot_{m-1})}
} \quad & \\
& \leq D 3^{-2m}\|u\|_{\L^2(\cudot_m)} + CD(1+ \lambda_{1,1}^{-\nicefrac12}(\cudot_m) + \Lambda_{1,1}^{\nicefrac12}(\cudot_m))(1 + \lambda_{s,2}^{-\nicefrac12}(\cudot_m)) 3^{sm}\|\mathbf{f}\|_{\underline{B}_{2,2}^s(\cudot_m)}\,,
\end{align*}
where 
\begin{equation*}
D \coloneqq \biggl(\frac{C(d)}{1-s-t}\biggr)^{2 + \frac{4s}{1-s-t}}\biggl( 1 + (\Lambda_{s,1}(\cudot_m))^{\frac{1-t}{1-s-t}} \bigl( \lambda^{-1}_{t,1}(\cudot_m) + \Lambda_{t,1}(\cudot_m) \bigr)^{\frac{s}{1-s-t}}\biggr)\,.
\end{equation*}
\end{proposition}
\begin{proof}
Let~$v$ solve
\begin{equation*}
\left\{
\begin{aligned}
& \partial_t v = \nabla \cdot \a \nabla v + \nabla \cdot \mathbf{f} & \quad \mbox{ in } \quad \cudot_m \\
& v = 0 & \quad \mbox{ on } \partial_{\sqcup} \cudot_m
\end{aligned}
\right.
\end{equation*}
By Proposition~\ref{p.coarse.caccio},
\begin{align}
& \|\s^{\nicefrac12} \nabla (u-v)\|_{\L^2(\cudot_{m-1})}^2 \nonumber \\
& \leq \biggl(\frac{C(d)}{1-s-t}\biggr)^{2 + \frac{4s}{1-s-t}}\biggl( 1 + (\Lambda_{s,1}(\cudot_m))^{\frac{1-t}{1-s-t}} \bigl( \lambda^{-1}_{t,1}(\cudot_m) + \Lambda_{t,1}(\cudot_m) \bigr)^{\frac{s}{1-s-t}}\biggr) 3^{-2m}\|u-v\|_{\L^2(\cudot_m)}^2\,.
\end{align}
Now by the triangle inequality and Proposition~\ref{p.cg.energy},
\begin{align*}
\|\s^{\nicefrac12} \nabla u\|_{\L^2(\cudot_{m-1})} & \leq \|\s^{\nicefrac12} \nabla (u-v)\|_{\L^2(\cudot_{m-1})} + \|\s^{\nicefrac12} \nabla v\|_{\L^2(\cudot_{m-1})} \\
& \leq \|\s^{\nicefrac12} \nabla (u-v)\|_{\L^2(\cudot_{m-1})} + C\lambda_{s,2}^{-\nicefrac12}(\cudot_m) 3^{sm}\|\mathbf{f}\|_{\underline{B}_{2,2}^s(\cudot_m)}
\end{align*}
and by Proposition~\ref{l.multiscale.poincare} and Proposition~\ref{p.cg.energy},
\begin{align*}
\lefteqn{
\|u-v\|_{\L^2(\cudot_m)}
} \quad & \\
& \leq \|u\|_{\L^2(\cudot_m)} + \|v\|_{\L^2(\cudot_m)} \\
& \leq \|u\|_{\L^2(\cudot_m)} + C\|\nabla v\|_{\Bring_{2,1}^{-1}(\cudot_m)} + C\|\a\nabla v\|_{\Bring_{2,1}^{-1}(\cudot_m)} + C\|\mathbf{f}\|_{\Bring_{2,1}^{-1}(\cudot_m)} \\
& \leq \|u\|_{\L^2(\cudot_m)}  + C\|\mathbf{f}\|_{\Bring_{2,1}^{-1}(\cudot_m)} + 3^m C( \lambda_{1,1}^{-\nicefrac12}(\cudot_m) + \Lambda_{1,1}^{\nicefrac12}(\cudot_m))\|\s^{\nicefrac12}\nabla v\|_{\L^2(\cudot_m)} \\
& \leq \|u\|_{\L^2(\cudot_m)}  + C3^m\|\mathbf{f}\|_{\L^2(\cudot_m)} + C3^m( \lambda_{1,1}^{-\nicefrac12}(\cudot_m) + \Lambda_{1,1}^{\nicefrac12}(\cudot_m))\lambda_{s,2}^{-\nicefrac12}(\cudot_m) 3^{sm}\|\mathbf{f}\|_{\underline{B}_{2,2}^s(\cudot_m)}\,.
\end{align*}
Plugging in these estimates concludes the proof.
\end{proof}

\subsection{Homogenization of the Dirichlet Problem}
\label{ss.homogenize}

Define the homogenized matrix
\begin{equation}
\label{e.ahom.def}
\ahom \coloneqq \shom + \khom\,,
\end{equation}
with~$\shom$ and~$\khom$ the homogenized matrices defined in~\eqref{e.bfAhom.def}. Throughout this section we center the coefficient field such that
\begin{equation}
\label{e.khom.zero}
\khom = 0\,,
\end{equation}
as in Section~\ref{ss.bfA.def}. This assumption means that where we write the coefficient field~$\a(\cdot)$ in this section, we are really referring to the centred coefficient field~$\a(\cdot) - \khom$. In this section we work with the adapted cubes defined by~$\m_0=\ahom$ (which is symmetric, by the centring). The only reason for this choice is that it makes the dependence on the ellipticity of homogenized matrix~$\ahom$ explicit in our estimates. If~$I = (0,T)$ is a finite time interval and~$U\subseteq \mathbb{R}^n$ is a bounded Lipschitz domain we let~$n\in\mathbb{Z}$ be the smallest integer such that~$I\times U \subseteq \cusdot_n$ and assume that there exists a constant~$c >0$ such that~$c|\cusdot_n| \leq |I\times U|$. That is, we assume that~$I\times U$ is on the same scale as a parabolically scaled cube, and we allow the constants in this section to depend on the volume ratio~$c$.

We define
\begin{equation}
\label{e.mathcalE.def}
\mathcal{E}_s(\cusdot_m) \coloneqq \biggl(\sum_{k=-\infty}^m 3^{2s(k-m)} \sup_{z \in \mathbb{L}_k \cap \cusdot_m} |( \bfAhom^{-\nicefrac12}( \bfA(z+\cusdot_k) - \bfAhom )_+ \bfAhom^{-\nicefrac12}| \biggr)^{\nicefrac12}\,.
\end{equation}
Note that the centring assumption means that
\begin{equation*}
\bfAhom = \begin{pmatrix}
\ahom & 0 \\ 0 & \ahom^{-1}
\end{pmatrix}\,.
\end{equation*}

The top left block of~$\bfA(z+\cusdot_k)$ is~$\b(z+\cusdot_k)$. Since one possible norm on a block matrix is the sum of the norms of the block entries and all norms are equivalent, up to constants, we have
\begin{align*}
|\ahom^{-\nicefrac12}\b(z+\cusdot_k)\ahom^{-\nicefrac12}| & \leq |\ahom^{-\nicefrac12}(\b(z+\cusdot_k) - \ahom)_+\ahom^{-\nicefrac12}| + 1 \\
& \leq C(|( \bfAhom^{-\nicefrac12}( \bfA(z+\cusdot_k) - \bfAhom )_+ \bfAhom^{-\nicefrac12}| + 1)\,.
\end{align*}
Using the analogous bound for~$\s_*^{-1}(z+\cusdot_k)$ and the definitions in~\eqref{e.adapted.Lambda} and~\eqref{e.adapted.lambda}, we have
\begin{equation}
\label{e.lambdas.finite}
\Lambda_{s,2}(z+\cusdot_n) + \lambda_{s,2}^{-1}(\cusdot_n) \leq C(d)(1 + \mathcal{E}_s(z+\cusdot_n))\,.
\end{equation}

\begin{theorem}
\label{t.homogenize}
Suppose that~$I = (0,T)$ is a finite time interval,~$U \subseteq \mathbb{R}^d$ is a bounded domain that is either~$C^{1,1}$ or convex and Lipschitz,~$\ahom$ is the homogenized matrix in~\eqref{e.ahom.def},~$\khom = 0$,~$0 <s < \nicefrac12$,~$0 <\epsilon < s$,~$v\in L^2(I\times U)$ and~$\mathbf{f},\nabla v \in B^{s}_{2,2}(I\times U)^d$ such that~$\partial_t v - \nabla \cdot \ahom \nabla v = \nabla \cdot \mathbf{f}$ in~$I\times U$. There exists a constant~$C = C(I,U,s,d)$ such that if~$u \in W_{\s}^1(I\times U)$ solves
\begin{equation}
\left\{
\begin{aligned}
& \partial_t u -\nabla \cdot \a \nabla u  = \nabla \cdot \mathbf{f} & \quad \mbox{ in } I\times U\,, \\
& u = v & \quad \mbox{ on } \partial_{\sqcup}(I\times U)\,,
\end{aligned}
\right.
\end{equation}
then
\begin{align}
\lefteqn{
3^{-sm}\|\ahom^{\nicefrac12}\nabla (u - v) \|_{\underline{\hat{B}}_{2,2}^{-s}(I\times U)} + 3^{-sm}\|\ahom^{-\nicefrac12}(\a \nabla u - \ahom \nabla v)\|_{\underline{\hat{B}}_{2,2}^{-s}(I\times U)}
} \quad & \notag\\
& \leq C \sup_{z\in\mathbb{L}_n \cap \cusdot_m}\mathcal{E}_s(z+\cusdot_n) \|\s^{\nicefrac12}\nabla u\|_{\L^2(I\times U)} + C(1 + \sup_{z\in\mathbb{L}_n \cap \cusdot_m}\mathcal{E}_{s-\epsilon}(z+\cusdot_n)) 3^{sn}\|\lambda_{\ahom}^{-\nicefrac12}\mathbf{f}\|_{\underline{B}_{2,2}^s(I\times U)}\,.
\end{align}
\end{theorem}
\begin{proof}
Combine Proposition~\ref{p.duality} and Proposition~\ref{p.flux.diff.RHS}, below.
\end{proof}

\begin{remark}
\label{r.boundary}
The energy of the solution is controlled by the data, according to Proposition~\ref{p.cg.energy}. We omit this substitution because it complicates the statement of the theorem. In particular, the dependence on the spatial boundary data~$g$ in Proposition~\ref{p.cg.energy} results from transforming the boundary data into the right-hand of a divergence-form equation, leading to a somewhat complicated expression. The important point is that given fixed boundary data this is of constant size.
\end{remark}

\begin{proposition}
\label{p.duality}
Suppose that~$I = (0,T)$ is a finite time interval,~$U \subseteq \mathbb{R}^d$ is a bounded domain that is either~$C^{1,1}$ or convex and Lipschitz,~$\ahom$ is the homogenized matrix in~\eqref{e.ahom.def},~$\khom = 0$,~$s \in (0,\nicefrac12)$ and~$v\in L^2(I\times U)$ such that~$\nabla v \in B^{s}_{2,2}(I\times U)$. There exists a constant~$C = C(I,U,s,d)$ such that if~$u \in W_{\s}^1(I\times U)$ solves
\begin{equation}
\left\{
\begin{aligned}
& \partial_t u -\nabla \cdot \a \nabla u  = \partial_t v - \nabla \cdot \ahom \nabla v & \quad \mbox{ in } I\times U\,, \\
& u = v & \quad \mbox{ on } \partial_{\sqcup}(I\times U)\,,
\end{aligned}
\right.
\end{equation}
then
\begin{equation}
\|\ahom^{\nicefrac12}\nabla (u - v) \|_{\underline{\hat{B}}_{2,2}^{-s}(I\times U)} + \|\ahom^{-\nicefrac12}(\a \nabla u - \ahom \nabla v)\|_{\underline{\hat{B}}_{2,2}^{-s}(I\times U)} \leq C \|\ahom^{-\nicefrac12}(\a - \ahom)\nabla u\|_{\underline{\hat{B}}_{2,2}^{-s}(I\times U)}\,.
\end{equation}
\end{proposition}

\begin{proof}
The function~$u-v$ satisfies the equation
\begin{equation}
\label{e.b.equation}
	\left\{
	\begin{aligned}
	& (\partial_t - \nabla \cdot \ahom \nabla ) (u-v) = \nabla \cdot (\a - \ahom)\nabla u & \quad \mbox{in } I\times U\,, \\
	& u-v = 0 & \quad \mbox{on } \partial_{\sqcup} (I \times U)\,.
	\end{aligned}
	\right.
\end{equation}
Let~$\mathbf{h}\in C^\infty(I\times U;\mathbb{R}^d)$ and let~$w$ be the unique function in~$L^2(I; H^1(U))$ with~$\partial_t w \in L^2(I; H^{-1}(U))$ solving
\begin{equation*}
	\left\{
	\begin{aligned}
	& (\partial_t + \nabla \cdot \ahom \nabla ) w = \nabla \cdot \mathbf{h} & \quad \mbox{in } I\times U \\
	& w = 0 & \quad \mbox{on } (I \times \partial\, U) \cup (\{T\}\times U)
	\end{aligned}
	\right.
\end{equation*}
Note that~$w$ vanishes at the final time, so we have regularity estimates on~$w(-t)$. Testing the equation for~$w$ with~$u-v$ and then the equation for~$u-v$ with~$w$ we obtain
\begin{equation}
\label{e.homogenize.duality}
\fint_{I\times U} \ahom^{\nicefrac12}\nabla (u-v)\cdot \ahom^{-\nicefrac12}\mathbf{h} = \fint_{I\times U} \ahom^{\nicefrac12}\nabla w \cdot \ahom^{-\nicefrac12}(\a - \ahom)\nabla u\,.
\end{equation}
We estimate the right-hand side using Lemma~\ref{l.heat.estimate}
\begin{align*}
\fint_{I\times U} \ahom^{\nicefrac12}\nabla w \cdot \ahom^{-\nicefrac12}(\a - \ahom)\nabla u & \leq \|\ahom^{\nicefrac12}\nabla w\|_{\underline{B}_{2,2}^s(I\times U)} \|\ahom^{-\nicefrac12}(\a - \ahom)\nabla u \|_{\underline{\hat{B}}_{2,2}^{-s}(I\times U)} \\
& \leq C \|\ahom^{-\nicefrac12}\h\|_{\L^2(I;\underline{H}^s(U))} \|\ahom^{-\nicefrac12}(\a - \ahom)\nabla u \|_{\underline{\hat{B}}_{2,2}^{-s}(I\times U)}\,.
\end{align*}
Then by duality
\begin{align}
\label{e.grad.by.flux.diff}
\|\ahom^{\nicefrac12} \nabla (u-v)\|_{\underline{\hat{B}}_{2,2}^{-s}(I\times U)} \leq C \|\shom^{\nicefrac12}\nabla (u - v)\|_{\L^2(I;\underline{\hat{H}}^{-s}(U))} & \leq C \|\ahom^{-\nicefrac12}(\a-\ahom)\nabla u\|_{\underline{\hat{B}}^{-s}_{2,2}(I\times U)}\,.
\end{align}
\newline
By the triangle inequality and~\eqref{e.grad.by.flux.diff},
\begin{align}
\label{e.flux.by.flux.diff}
\|\ahom^{-\nicefrac12}(\a\nabla u-\ahom\nabla v)\|_{\underline{\hat{B}}_{2,2}^{-s}(I\times U)}& \leq \|\ahom^{\nicefrac12} \nabla (u-v)\|_{\underline{\hat{B}}_{2,2}^{-s}(I\times U)} + \|\ahom^{-\nicefrac12}(\a-\ahom)\nabla u\|_{\underline{B}_{2,2}^{-s}(I\times U)} \nonumber \\
& \leq C \|\ahom^{-\nicefrac12}(\a-\ahom)\nabla u\|_{\underline{B}^{-s}_{2,2}(I\times U)}\,,
\end{align}
which completes the proof.
\end{proof}

The next proposition will be applied only in sub-cubes, and applies only to solutions with zero right-hand side.
\begin{proposition}
\label{p.flux.diff}
Suppose that~$I$ is a finite time interval,~$U \subseteq \mathbb{R}^d$ is a bounded Lipschitz domain,~$\ahom$ is the homogenized matrix in~\eqref{e.ahom.def},~$\khom = 0$, and~$0 < s < \nicefrac12$. There exists a constant~$C = C(I,U,d)$ such that if~$u \in W_{\s}^1(I\times U)$ solves~$\partial_t u = \nabla \cdot \a\nabla u$ in~$I\times U$ then
\begin{equation}
\|\ahom^{-\nicefrac12}(\a - \ahom)\nabla u\|_{\underline{\hat{B}}_{2,2}^{-s}(I\times U)} \leq C 3^{sm}\mathcal{E}_s(\cusdot_m) \|\s^{\nicefrac12}\nabla u\|_{\L^2(I\times U)}\,.
\end{equation}
\end{proposition}
\begin{proof}
By Lemma~\ref{l.general.dual.norm} we need to control averages of~$\ahom^{-\nicefrac12}(\a - \ahom)\nabla u$ in sub-cubes, for which we use~\eqref{e.fluxmaps} with~$|p| = 1$ and~$q = \shom p$. This gives
\begin{align*}
\lefteqn{
\sum_{k=-\infty}^m 3^{2sk} \avsum_{z\in\mathbb{L}_k, z+\cusdot_k \subseteq I\times U} |(\ahom^{-\nicefrac12}(\a - \ahom)\nabla u)_{z+\cudot_k} |^2
} \qquad & \\
& \leq \sum_{k=-\infty}^m 3^{2sk} \avsum_{z\in\mathbb{L}_k, z+\cusdot_k \subseteq I\times U} \sup_{|p|=1} J(z+\cusdot_k,\ahom^{-\nicefrac12}p,\ahom^{\nicefrac12} p) \|\s^{\nicefrac12}\nabla u\|_{\L^2(z+\cudot_k)} \\
& \leq C \|\s^{\nicefrac12}\nabla u\|_{\L^2(I\times U)} \sum_{k=-\infty}^m 3^{2sk} \sup_{z \in \mathbb{L}_k \cap \cusdot_m} \sup_{|p|=1} J(z+\cusdot_k,\ahom^{-\nicefrac12}p,\ahom^{\nicefrac12} p)\,.
\end{align*}
Now~\eqref{e.Jaas.matform} implies that
\begin{equation*}
J(z+\cusdot_k,\ahom^{-\nicefrac12}p,\ahom^{\nicefrac12} p) =
\frac{1}{2}\begin{pmatrix}
-p \\ p
\end{pmatrix}
\cdot
( \bfAhom^{-\nicefrac12}( \bfA(z+\cusdot_k) - \bfAhom ) \bfAhom^{-\nicefrac12} )
\begin{pmatrix}
-p \\ p
\end{pmatrix}\,,
\end{equation*}
so we conclude by plugging the above estimates into Lemma~\ref{l.general.dual.norm} and using the definition of~$\mathcal{E}_s(\cusdot_m)$.
\end{proof}

\begin{proposition}
\label{p.flux.diff.RHS}
Suppose that~$I$ is a finite time interval,~$U \subseteq \mathbb{R}^d$ is a bounded Lipschitz domain,~$\ahom$ is the homogenized matrix in~\eqref{e.ahom.def},~$\khom = 0$,~$0 < s  < \nicefrac12$, and~$0 < \epsilon < s$. There exists a constant~$C = C(I,U,\epsilon,d)$ such that if~$u \in W_{\s}^1(I\times U)$ solves~$\partial_t u = \nabla \cdot \a\nabla u + \nabla \cdot \mathbf{f}$ in~$I\times U$ then
\begin{align}
3^{-sm}\|\shom^{-\nicefrac12}(\a - \ahom)\nabla u\|_{\underline{\hat{B}}_{2,2}^{-s}(I\times U)} & \leq C \sup_{z\in\mathbb{L}_n \cap \cusdot_m}\mathcal{E}_s(z+\cusdot_n) \|\s^{\nicefrac12}\nabla u\|_{\L^2(I\times U)} \notag \\
& \quad + C(1 + \sup_{z\in\mathbb{L}_n \cap \cusdot_m}\mathcal{E}_{s-\epsilon}(z+\cusdot_n)) 3^{sn} \|\lambda_{\ahom}^{-\nicefrac12}\mathbf{f}\|_{\underline{B}_{2,2}^{s}(I\times U)}
\end{align}
\end{proposition}
\begin{proof}
First, consider a sub-cube~$z+\cusdot_n$ where~$z\in\mathbb{L}_{n-1}$, and let~$w_{z+\cusdot_n}$ be the unique solution to
\begin{equation}
\left\{
\begin{aligned}
&\partial_t w_{z+\cusdot_n} - \nabla \cdot \a\nabla w_{z+\cusdot_n} = \nabla \cdot \mathbf{f} & \mbox{ in } z+\cusdot_n \cap I\times U \\
& w = 0 & \mbox{ on } \partial_{\sqcup}(z+\cusdot_n) \cap I\times U\\
\end{aligned}
\right.
\end{equation}
Then by the triangle inequality, Proposition~\ref{p.flux.diff} applied to the difference~$u-w_{z+\cusdot_n}$, and Proposition~\ref{p.poincare.RHS} followed by Proposition~\ref{p.cg.energy} applied to~$w_{z+\cusdot_n}$,

\begin{align*} 
\lefteqn{
\| \ahom^{-\nicefrac12}(\a-\ahom)\nabla u\|_{\underline{\hat{B}}_{2,2}^{-s}(z+\cusdot_n \cap I \times U)}
} \qquad & \\
& \leq \| \ahom^{-\nicefrac12}(\a-\ahom)\nabla (u-w)\|_{\underline{\hat{B}}_{2,2}^{-s}(z+\cusdot_n \cap I\times U)} + \| \ahom^{-\nicefrac12}\a\nabla w\|_{\underline{\hat{B}}_{2,2}^{-s}(z+\cusdot_n \cap I \times U)}  \\
& \qquad + \| \ahom^{\nicefrac12}\nabla w\|_{\underline{\hat{B}}_{2,2}^{-s}(z+\cusdot_n \cap I \times U)} \\
& \leq  C3^{sn}\mathcal{E}_s(z+\cusdot_n)\|\s^{\nicefrac12}\nabla (u-w)\|_{\L^2(z+\cusdot_n \cap I \times U)} \\
& \qquad + C(1 + \mathcal{E}_{s-\epsilon}(z+\cusdot_n)) 3^{2sn}\|\lambda_{\ahom}^{-\nicefrac12}\mathbf{f}\|_{\underline{B}_{2,2}^s(z+\cusdot_n \cap I \times U)} \\
& \leq C3^{sn}\mathcal{E}_s(z+\cusdot_n)\|\s^{\nicefrac12}\nabla u\|_{\L^2(z+\cusdot_n \cap I \times U)} + C(1 + \mathcal{E}_{s-\epsilon}(z+\cusdot_n)) 3^{2sn}\|\lambda_{\ahom}^{-\nicefrac12}\mathbf{f}\|_{\underline{B}_{2,2}^s(z+\cusdot_n \cap I\times U)}
\end{align*}
Now partition the domain~$I\times U$ into cubes~$z+\cusdot_{n-1}$ such that~$z\in \mathbb{L}_{n-1}$. Expanding the cubes yields an overlapping partition, so,
\begin{align*}
\| \ahom^{-\nicefrac12}(\a - \ahom)\nabla u \|_{\underline{B}_{2,2}^{-s}(I \times U)} & \leq C3^{s(m-n)}\biggl(\avsum_{z\in\mathbb{L}_{n-1}, z+\cusdot_{n-1} \cap I\times U \neq \emptyset} \|\ahom^{-\nicefrac12}(\a - \ahom)\nabla u \|_{\underline{B}_{2,2}^{-s}(z+\cusdot_n \cap I \times U)}^2 \biggr)^{\nicefrac12} \\
& \leq C 3^{sm}\sup_{z\in \mathbb{L}_n\cap \cusdot_m} \mathcal{E}_s(z+\cusdot_n) \|\s^{\nicefrac12}\nabla u\|_{\L^2(I\times U)} \\
& \qquad + C(1 + \sup_{z\in \mathbb{L}_n\cap \cusdot_m}\mathcal{E}_{s-\epsilon}(z+\cusdot_n)) 3^{sm} 3^{sn}\|\lambda_{\ahom}^{-\nicefrac12}\mathbf{f}\|_{\underline{B}_{2,2}^{s}(I\times U)}\,.
\end{align*}
\end{proof}

\subsection{Proof of Theorem~\ref{t.theoremA}}
The adapted cubes in this section are different from those in Section~\ref{s.smallcontrast}. However, the theorems of that section apply by Lemma~\ref{l.bfAhoms.equivalent} applied to the adapted cubes, noting that the ellipticity constants for the homogenized matrix are smaller than those in the statement of that lemma.

\begin{lemma}
\label{l.mathcalE.rate}
Assume~\ref{a.stationarity},\ref{a.ellipticity.dagger} and~\ref{a.CFS}, and~$\khom = 0$. For any~$\delta > 0$ and~$\gamma' \in (\gamma,1)$ there exists a random variable~$\mathcal{Y}_{\delta,\gamma'}$ and~$\theta >0$ (given by Theorem~\ref{t.quenched}) such that if~$3^m \geq \mathcal{Y}_{\delta,\gamma'}\vee \mathcal{S}$ and~$\nicefrac{\gamma'}{2} < s < 1$ then
\begin{equation}
\label{e.mathcalE.rate}
\sup_{z\in\mathbb{L}_n\cap \cusdot_m} \mathcal{E}_s(z+\cusdot_n) \leq 3^{\gamma'(m-n)}\frac{C(d)\delta^{\nicefrac12}}{2s-\gamma'}\biggl(\frac{\mathcal{Y}_{\delta,\gamma'} \vee \mathcal{S}}{3^m} \biggr)^{\nicefrac{\theta}{2}}\,.
\end{equation}
\end{lemma}
\begin{proof}
By Theorem~\ref{t.theta.rate} and~\ref{t.quenched}, for~$\delta>0, \gamma' \in (\gamma,1)$ there exists a random variable~$\mathcal{Y}_{\delta,\gamma'}$ such that for~$\theta$ as in the statement of that theorem, if~$3^m \geq \mathcal{Y}_{\delta,\gamma'} \vee \mathcal{S}$ then
\begin{equation*}
\bigl| \bigl( \bfAhom^{-\nicefrac12}\bfA(z+\cusdot_k)\bfAhom^{-\nicefrac12} - \mathrm{Id} \bigr)_+ \bigr| \leq \delta 3^{\gamma'(m-k)} \biggl(\frac{\mathcal{Y}_{\delta,\gamma'} \vee \mathcal{S}}{3^m} \biggr)^{\theta}\,.
\end{equation*}
Then for any~$k\leq m, z\in\mathbb{L}_{k-1} \cap \cusdot_m, |e|=1$
\begin{align*}
J(z+\cusdot_k,\shom^{-\nicefrac12}e,\shom^{\nicefrac12}e) \leq \delta 3^{\gamma'(m-k)} \biggl(\frac{\mathcal{Y}_{\delta,\gamma'} \vee \mathcal{S}}{3^m} \biggr)^{\theta}\,,
\end{align*}
Summing in~$k$ up to~$n$ we obtain
\begin{align*}
\sup_{z\in \mathbb{L}_n \cap \cusdot_m} \mathcal{E}_s(z + \cusdot_n) & \leq \biggl(\sum_{k=-\infty}^n 3^{2s(k-n)} \sup_{z\in\mathbb{L}_k\cap \cusdot_m, |e|=1}J(z+\cusdot_k,\shom^{-\nicefrac12}e,\shom^{\nicefrac12}e)\biggr)^{\nicefrac12} \\
& \leq \biggl(\sum_{k=-\infty}^n 3^{2s(k-n)} \delta 3^{\gamma'(m-k)} \biggl(\frac{\mathcal{Y}_{\delta,\gamma'} \vee \mathcal{S}}{3^m} \biggr)^{\theta} \biggr)^{\nicefrac12} \\
&  \leq 3^{\gamma' (m-n)}\frac{C\delta^{\nicefrac12}}{2s-\gamma'}\biggl(\frac{\mathcal{Y}_{\delta,\gamma'} \vee \mathcal{S}}{3^m} \biggr)^{\nicefrac{\theta}{2}}\,,
\end{align*}
which concludes the proof.
\end{proof}

\begin{proof}[Proof of Theorem~\ref{t.theoremA}]
Assume~\ref{a.stationarity},~\ref{a.ellipticity} and~\ref{a.CFS}, and define the adapted domains by~$\m_0 = \shom$. Let~$t_0$ be the initial time in the adapted interval~$J_0$. Define~$\a^\epsilon(t,x) = \a(\nicefrac{t}{\epsilon^2},\nicefrac{x}{\epsilon})$ and for fixed~$\mathbf{f}\in B^s_{2,2}(\cusdot_0)^d$ and~$u_0\in L^2(\cus_0)$ let~$u^\epsilon$ and~$v$ be the unique solutions to
\begin{equation}
\left\{
\begin{aligned}
& \partial_t u^\epsilon -\nabla \cdot \a^\epsilon \nabla u^\epsilon = \nabla \cdot \mathbf{f} & \mbox{ in } \cusdot_0\\
& u^\epsilon = 0 & \mbox{ on }  J_0 \times \partial\cus_0 \\
& u^\epsilon = u_0 & \mbox{ at } t = t_0\,.
\end{aligned}
\right.
\qquad 
\left\{
\begin{aligned}
& \partial_t v - \nabla \cdot \ahom \nabla v = \nabla \cdot \mathbf{f} & \mbox{ in } \cusdot_0 \\
& v = 0 & \mbox{ on }  J_0 \times \partial\cus_0 \\
& v = u_0 & \mbox{ at } t = t_0\,.
\end{aligned}
\right.
\end{equation}

First, it is immediate from the definitions that if~$\bfA^\epsilon(V)$ denotes the coarse-graining of~$\a^\epsilon$ in any space-time domain~$V$, then~$\bfA^\epsilon(V) = \bfA(\epsilon^{-1}V)$. It follows that if~$m = \lceil-\log_3\epsilon\rceil$ and~$\mathcal{E}^\epsilon_s(\cusdot_0)$ denotes the quantity in~\eqref{e.mathcalE.def} for~$\a^\epsilon$, then~$\mathcal{E}^\epsilon_s(\cusdot_0) = \mathcal{E}_s(\cusdot_m)$, where the latter is defined for the coefficient field~$\a$. We will make use of this abuse of notation.

Since $\partial_t (u^\epsilon-v) = \nabla \cdot ((\a^\epsilon-\khom)\nabla u^\epsilon - \shom \nabla v)$, we have by Proposition~\ref{l.multiscale.poincare} and Theorem~\ref{t.homogenize}, for any~$0 < s < \nicefrac{1}{2}$ and~$n \leq m$
\begin{align}
\label{e.first.step.thm}
\lefteqn{
\| u^\epsilon - v \|_{\L^2(\cusdot_0)}
} \qquad & \notag \\
& \leq C \lambda_{\shom}^{-\nicefrac12} \| \shom^{\nicefrac12} \nabla (u^\epsilon - v)\|_{\underline{\hat{B}}_{2,1}^{-1}(\cusdot_0)} + C \lambda_{\shom}^{-\nicefrac12}\|\shom^{-\nicefrac12} ((\a-\khom)\nabla u - \shom \nabla v)\|_{\underline{\hat{B}}_{2,1}^{-1}(\cusdot_0)} \notag \\
& \leq C \sup_{z\in\mathbb{L}_n \cap \cusdot_m}\mathcal{E}_s(z+\cusdot_n) \lambda_{\shom}^{-\nicefrac12}\|(\s^\epsilon)^{\nicefrac12}\nabla u^\epsilon\|_{\L^2(\cusdot_0)} \notag \\
& \qquad + C(1 + \sup_{z\in\mathbb{L}_n \cap \cusdot_m}\mathcal{E}_{\nicefrac{s}{2}}(z+\cusdot_n)) 3^{s(n-m)}\|\lambda_{\shom}^{-1}\mathbf{f}\|_{\underline{B}_{2,2}^s(\cusdot_0)}\,.
\end{align}
We now choose parameter~$\gamma' = s+\frac{\gamma}{2}$ and~$n = n(m)$ satisfying~$m-n = \frac{\theta}{4s +\gamma} m$, where~$\theta > 0$ is the constant given by Theorem~\ref{t.quenched}. This choice of~$n$ serves to balance the two error terms in~\eqref{e.first.step.thm}. Applying Lemma~\ref{l.mathcalE.rate} with these parameters and~$\delta = 1$,
\begin{equation}
\label{e.step2}
\sup_{z\in\mathbb{L}_n \cap \cusdot_m} \mathcal{E}_s(z+\cusdot_n) \leq 3^{-m\frac{\theta s}{4s+\gamma} } (\mathcal{Y}_{1,\gamma'} \vee \S)^{\nicefrac{\theta}{2}} \frac{C(d)}{1-\gamma'}
\end{equation}
and the parameters have been chosen such that the second term in~\eqref{e.first.step.thm} has the matching rate
\begin{equation}
\label{e.step3}
3^{s(n-m)} = 3^{-m \frac{\theta s}{4s+\gamma}} \,.
\end{equation}
By Proposition~\ref{p.cg.energy} and~\eqref{e.lambdas.finite}
\begin{equation}
\label{e.step4}
\|(\s^\epsilon)^{\nicefrac12}\nabla u^\epsilon\|_{\L^2(\cusdot_0)} \leq C(1 + \mathcal{E}_{\nicefrac{s}{2}}(\cusdot_m)) (\|\lambda_{\shom}^{-\nicefrac12}\mathbf{f}\|_{\underline{B}^s_{2,2}(\cusdot_0)} + \|u_0\|_{\L^2(\cus_0)})\,.
\end{equation}
Putting together~\eqref{e.first.step.thm},~\eqref{e.step2},~\eqref{e.step3} and~\eqref{e.step4}
\begin{align*}
\lefteqn{
\| u^\epsilon - v \|_{\L^2(\cusdot_0)}
} \qquad & \\
& \leq \epsilon^{\frac{\theta s}{4s+\gamma} } (\mathcal{Y}_{1,\gamma'} \vee \S)^{\nicefrac{\theta}{2}} \frac{C(d)}{1-\gamma'} \bigl( 1 + \epsilon^{\frac{\theta s}{4s+\gamma} } (\mathcal{Y}_{1,\gamma'} \vee \S)^{\nicefrac{\theta}{2}} \frac{C(d)}{1-\gamma'} \bigr) (\|\mathbf{f}\|_{\underline{B}^s_{2,2}(\cusdot_0)} + \|u_0\|_{\L^2(\cus_0)})  \,.
\end{align*}
Taking~$\rho = \frac{\theta s}{4s+\gamma}$ we obtain Theorem~\ref{t.theoremA}.
\end{proof}

\begin{proof}[Proof of Corollary~\ref{c.frd}]
Suppose that~$\a$ is a uniformly elliptic field satisfying~\eqref{e.uniform.ellipticity} with constants~$0<\lambda \leq \Lambda < \infty$, and with finite range of dependence~$1$ in space and~$T$ in time. We may assume without loss of generality that there exists~$k\in\mathbb{Z}$ such that~$\lambda = 3^{2k}$. Define
\begin{equation}
\label{e.rescale}
\tilde{\a}(t,x) = \lambda^{-1}\a(\lambda^{-1} t,x)\,.
\end{equation}
By Proposition~\ref{p.rescaling} we may apply Theorem~\ref{t.theoremB} to the dilation~$D_{n_0}\tilde{\a}$ provided that~$n_0$ is the smallest integer satisfying~\eqref{e.first.scale.sep}. We then undo the dilation by adding~$3^{n_0}$ to the length scale~$L$ given by Theorem~\ref{t.theoremB}, and move from~$\tilde{\a}$ back to~$\a$ by applying Lemma~\ref{l.change.variables} to obtain the statement of Corollary~\ref{c.frd}.
\end{proof}

\appendix

\section{Functional inequalities}
\label{aa.functions}

In this appendix we state the parabolic functional inequalities used in the paper. These will often be applied in the adapted domains given by~\eqref{e.adapted.cyl.def} with a positive-definite matrix~$\q_0$ and constant~$\lambda_r > 0$; in practice these will be given as in Section~\ref{ss.subadditivity}. The inequalities in adapted domains are obtained by the change of coordinates
\begin{equation}
\label{e.transformers}
\left.
\begin{aligned}
& \partial_t u - \nabla \cdot \a \nabla u = \nabla \cdot \mathbf{h} \quad \mbox{in} \quad \cusdot_n \\
& \tilde{u}(y,s):= u(\q_0(y),\lambda_r^{-1} s) \\
& \tilde{\a}(y,s) := \a(\q_0(y),\lambda_r^{-1}s) \\
& \tilde{\h}(y,s) := \h(\q_0(y),\lambda_r^{-1}s)
\end{aligned}
\right\}
\implies
\partial_s \tilde{u} - \nabla \cdot (\lambda_r^{\nicefrac12}\q_0)^{-1}\tilde{\a} (\lambda_r^{\nicefrac12}\q_0)^{-1}\nabla \tilde{u} = \nabla \cdot (\lambda_r \q_0)^{-1}\tilde{\h} \mbox{ in } \cudot_n\,.
\end{equation}

In our first proposition, the constant depends on norms~$\|\s^{-\nicefrac12}\k\s^{-\nicefrac12}\|_{L^\infty(I\times U)}$,~$\|\s^{-1}\|_{L^1(I\times U)}$ and~$\|\s\|_{\L^1(I\times U)}$ of the coefficient field. We use this proposition only once, in order to qualitatively justify the initial setup of the problem, and therefore make no effort to track the constants or scaling. The proof is a minor modification of results in~\cite[Section 3]{ABM}.
\begin{proposition}
\label{p.dumb.poincare}
Let~$I$ be a finite interval,~$U$ a bounded Lipschitz domain, and~$\partial_t u = \nabla \cdot \a\nabla u$ in~$I\times U$. There exists a constant~$C = C(d, \a, I, U)$ such that
\begin{equation}
\|u -(u)_{I\times U}\|_{\L^1(I\times U)} \leq C \|\s^{\nicefrac12}\nabla u\|_{\L^2(I\times U)}\,.
\end{equation}
\end{proposition}
\begin{proof}
\emph{Step 1:} We first work in time slices. Fix a spatial function~$\psi \in C_c^\infty(U)$ with~$\fint_{U} \psi = 1$ and let~$w$ be the unique mean-zero function solving the elliptic problem
\begin{equation*}
\left\{
\begin{aligned}
-\nabla \cdot \s \nabla w = 1-\psi & \mbox{ in } U \\
\mathbf{n} \cdot \s \nabla w = 0 & \mbox{ on } \partial U\,.
\end{aligned}
\right.
\end{equation*}
Then by testing the equation with itself,
\begin{align*}
\fint_{U} \nabla w \cdot \s \nabla w & = \fint_{U} (1-\psi)w \leq \|1-\psi\|_{L^\infty(U)}\|w\|_{\L^1(U)}  \leq C\|1-\psi\|_{L^\infty(U)}\|\nabla w\|_{\L^1(U)}\\
& \leq C\|1-\psi\|_{L^\infty(U)}\|\s^{-\nicefrac12}\|_{\L^2(U)}\|\s^{\nicefrac12}\nabla w\|_{\L^2(U)}\,,
\end{align*}
from which it follows that~$\|\s^{\nicefrac12}\nabla w\|_{\L^2(U)} \leq C\|1-\psi\|_{L^\infty(U)}\|\s^{-\nicefrac12}\|_{\L^2(U)}$. We may now estimate
\begin{align*}
\biggl| \fint_{U} u(1-\psi)\biggr| & = \biggl| \fint_{U} \s^{\nicefrac12} \nabla u \cdot \s^{\nicefrac12} \nabla w \biggr| \\
& \leq \|\s^{\nicefrac12} \nabla u\|_{\L^2(U)}\|\s^{\nicefrac12} \nabla w\|_{\L^2(U)} \\
& \leq C \|1-\psi\|_{L^\infty(U)} \|\s^{-\nicefrac12}\|_{\L^2(U)} \|\s^{\nicefrac12} \nabla u\|_{\L^2(U)}\,.
\end{align*}
Applying the usual Poincar\'e inequality and the triangle inequality we obtain, for each time,
\begin{align}
\label{e.time.slice.poincare}
\bigg\| u - \fint_{U}\psi u \bigg\|_{\L^1(U)} & \leq \| u - (u)_{U} \|_{\L^1(U)} + \biggl| \fint_{U} u(1-\psi) \biggr| \notag \\
& \leq C(1 + \|\psi\|_{L^\infty(U)} )\|\s^{-\nicefrac12}\|_{\L^2(U)} \|\s^{\nicefrac12} \nabla u\|_{\L^2(U)}\,.
\end{align}

\smallskip
\emph{Step 2:} We use the equation to compare the integral~$\fint_{U} \psi u$ in time slices to the space-time integral. Since~$\psi$ is compactly supported in space,
\begin{align*}
\lefteqn{
\sup_{t\in I} \biggl| \fint_{U} u(t,x)\psi(x) \, dx - \fint_{I} \fint_{U} u(t,x)\psi(x)\, dtdx \biggr|
} \qquad & \\
& \leq \int_{I} \bigg|\fint_{U} \partial_t u \psi \bigg| \leq \int_{I} \bigg| \fint_{U} \a\nabla u \cdot \nabla \psi\bigg| \\
& \leq |I| (1+\|\s^{-\nicefrac12}\k\s^{-\nicefrac12}\|_{L^\infty(I\times U)})\|\s^{\nicefrac12}\|_{\L^2(I\times U)}\|\s^{\nicefrac12}\nabla u\|_{\L^2(I\times U)} \|\nabla \psi\|_{L^\infty(U)}\,.
\end{align*}

\smallskip
\emph{Step 3:}
Combining the above two steps we get
\begin{align*}
\|u-(u)_{I\times U}\|_{\L^1(I\times U)} & \leq 2\inf_{c\in\mathbb{R}}\|u-c\|_{\L^1(I\times U)} \leq 2 \bigg\| u - \fint_{I}\fint_{U} u \psi\bigg\|_{\L^1(I\times U)} \\
& \leq \sup_{t\in I} \biggl| \fint_{U} u(t,x)\psi(x) \, dx - \fint_{I} \fint_{U} u(t,x)\psi(x)\, dtdx \biggr| \\
& \qquad + \bigg\| u - \fint_{U} u \psi\bigg\|_{\L^1(I\times U)} \\
& \leq C (1+\|\s^{-\nicefrac12}\k\s^{-\nicefrac12}\|_{L^\infty(I\times U)})\|\s^{\nicefrac12}\|_{\L^2(I\times U)}\|\s^{\nicefrac12}\nabla u\|_{\L^2(I\times U)} \|\nabla \psi\|_{L^\infty(U)} \\
& \qquad + C \|1-\psi\|_{L^\infty(U)} \|\s^{-\nicefrac12}\|_{\L^2(I\times U)} \|\s^{\nicefrac12} \nabla u\|_{\L^2(I\times U)}\,.
\end{align*}
Therefore fixing a function~$\psi \in C_c^\infty(U)$ such that~$\fint_{\cu_n}\psi = 1$,~$\|\psi\|_{L^\infty(\cu_n)} + \|\nabla \psi\|_{L^\infty(\cu_n)} \leq C$ we obtain the result.
\end{proof}

Our next proposition is the parabolic multiscale Poincar\'e inequality. We first introduce some notation. Recall that~$\partial_{\sqcup}(I\times U)$ is defined in~\eqref{e.sqcup.def} and define the volume-normalized norm
\begin{equation*}
\|g\|_{\underline{H}^1(U)} := |U|^{-\frac{1}{d}}\|g\|_{\L^2(U)} + \|\nabla g\|_{\L^2(U)}\,,
\end{equation*}
and the dual norms
\begin{align*}
\|f\|_{\underline{H}^{-1}(U)} & := \sup \bigg\{ \fint_U fg : g\in H_0^1(U) \,, \|g\|_{\underline{H}^1(U)} \leq 1 \bigg\} \\
\|f\|_{\underline{\hat{H}}^{-1}(U)} & := \sup \bigg\{ \fint_U fg : g\in H^1(U) \,, \|g\|_{\underline{H}^1(U)} \leq 1 \bigg\}\,.
\end{align*}
%Similarly
%\begin{equation*}
%\|g\|_{\underline{H}^2(U)} := |U|^{-\frac{2}{d}}\|g\|_{\L^2(U)} + |U|^{-\frac{1}{d}}\|\nabla g\|_{\L^2(U)} + \|\nabla^2 g\|_{\L^2(U)} \,,
%\end{equation*}
%and~$\|\cdot \|_{\underline{H}^{-2}(U)}$ and~$\|\cdot \|_{\underline{\hat{H}}^{-2}(U)}$ denote the dual norms with zero boundary data and arbitrary boundary data respectively.
The parabolic norm is
\begin{equation*}
\|u\|_{\underline{H}^1_{\mathrm{par}}(I\times U)} := \|u\|_{\L^2(I;\underline{H}^1(U))} + \|\partial_t u \|_{\L^2(I; \underline{H}^{-1}(U))}\,,
\end{equation*}
and we denote by~$\|\cdot \|_{\underline{H}^{-1}_{\mathrm{par}}(I\times U)}$ and~$\|\cdot \|_{\underline{\hat{H}}^{-1}_{\mathrm{par}}(I\times U)}$ the dual norms testing against, respectively, functions vanishing on~$\partial_{\sqcup}(I\times U)$ and arbitrary functions in~$H^1_{\mathrm{par}}(I\times U)$. By~\cite[Proposition 3.6]{ABM}, noting that the proof there extends immediately to include all scales down to~$k=-\infty$, we have
\begin{equation}
\label{e.bound.Hpar.dual}
\|f\|_{\underline{\hat{H}}^{-1}_{\mathrm{par}}(\cudot_n)} \leq C \|f\|_{\Bring^{-1}_{2,1}(\cudot_n)}\,,
\end{equation}
where the right-hand side was defined in~\eqref{e.Besov.dot.def}.

We state here the parabolic multiscale inequality proved in~\cite[Proposition 3.6, Proposition 3.7]{ABM}, with the differences that we use here all scales down to~$k=-\infty$ (as done in~\cite[Proposition 1.10]{AK.Book} in the elliptic case), and the estimate here is stated in terms of Besov spaces. Applying the transformation in~\eqref{e.transformers}, the statement in adapted domains is
\begin{equation*}
\| u- (u)_{\cusdot_n}\|_{\L^2(\cusdot_n)} \leq C(d)\bigl([\q_0\nabla u]_{\Bring_{2,1}^{-1}(\cusdot_{n+1})} + [\lambda_r^{-1}\q_0^{-1}\g]_{\Bring_{2,1}^{-1}(\cusdot_{n+1})}\bigr)\,.
\end{equation*}

\begin{lemma}[Parabolic Multiscale Poincar\'e Inequality]
\label{l.multiscale.poincare}
If~$\partial_t u = \nabla \cdot \g$ in~$\cudot_{n+1}$ then we have the interior estimate
\begin{align}
\label{e.parabolic.multiscale.poincare}
\| u- (u)_{\cudot_n}\|_{\L^2(\cudot_n)} \leq C(d)\bigl([\nabla u]_{\Bring_{2,1}^{-1}(\cudot_{n+1})} + [\g]_{\Bring_{2,1}^{-1}(\cudot_{n+1})}\bigr)\,.
\end{align}
If~$\partial_t u = \nabla \cdot \mathbf{g}$ in~$\cudot_n$ with $u=0$ on~$\partial_{\sqcup}\cudot_n$ then we have the global estimate
\begin{equation}
\label{e.boundary.poincare}
\|u\|_{\L^2(\cudot_n)} \leq C \bigl( \|\mathbf{g}\|_{\Bring^{-1}_{2,1}(\cudot_n)} + \|\nabla u\|_{\Bring^{-1}_{2,1}(\cudot_n)}\bigr)\,.
\end{equation}
\end{lemma}
\begin{proof}
The interior estimate~\eqref{e.parabolic.multiscale.poincare} can be found in~\cite[Proposition 3.6, Proposition 3.7]{ABM}, noting that the proof there can immediately be adapted to include sums going all the way down to~$-\infty$.
\newline
\emph{Proof of~\eqref{e.boundary.poincare}:}
Suppose that~$\partial_t u = \nabla \cdot \mathbf{g}$ in~$\cudot_n$ and~$u = 0$ on~$\partial_{\sqcup}\cudot_n$. Let
\begin{equation*}
\partial_{\sqcap}\cudot_n \coloneqq (\{\nicefrac{3^{2n}}{2}\} \times \cu_n) \cup (I_n \times \partial \cu_n)\,,
\end{equation*}
define an auxiliary function
\begin{align}
\begin{cases}
(\partial_t + \Delta)v = u & \mbox{in } \cudot_n\,, \\
v = 0 & \mbox{on }  \partial_{\sqcap}\cudot_n\,,
\end{cases}
\end{align}
which is a time-reversed solution to the heat equation. By standard arguments (for instance see~\cite[Lemma 3.9]{ABM} and use that our boundary data replaces periodicity) we have\newline $\|\nabla v\|_{\underline{H}^1_{\mathrm{par}}(\cudot_n)} \leq C \|u\|_{\L^2(\cudot_n)}$. Then using the equation for~$v$, the equation for~$u$ and duality,
\begin{align}
\label{e.integrate.parts}
\| u\|_{\L^2(\cudot_n)}^2 & = \fint_{\cudot_n} (\partial_t v + \Delta v) u  = \fint_{\cudot_n} -\partial_t u v - \nabla v \cdot \nabla u \nonumber \\
& = \fint_{\cudot_n} \mathbf{g}\cdot \nabla v - \nabla v \cdot \nabla u \nonumber \\
& \leq \bigl(\|\mathbf{g}\|_{\underline{\hat{H}}_{\mathrm{par}}^{-1}(\cudot_n)} + \|\nabla u\|_{\underline{\hat{H}}_{\mathrm{par}}^{-1}(\cudot_n)} \bigr) \|\nabla v\|_{\underline{H}_{\mathrm{par}}^{1}(\cudot_n)} \nonumber \\
& \leq C\bigl(\|\mathbf{g}\|_{\underline{\hat{H}}_{\mathrm{par}}^{-1}(\cudot_n)} + \|\nabla u\|_{\underline{\hat{H}}_{\mathrm{par}}^{-1}(\cudot_n)} \bigr) \|u\|_{\L^2(\cudot_n)}\,.
\end{align}
so by re-absorbing the last factor in~\eqref{e.integrate.parts} and applying~\eqref{e.bound.Hpar.dual},
\begin{equation}
\label{e.poincare.result}
\|u\|_{\L^2(\cudot_n)} \leq C( \|\mathbf{g}\|_{\Bring^{-1}_{2,1}(\cudot_n)} + \|\nabla u\|_{\Bring^{-1}_{2,1}(\cudot_n)} )\,.
\end{equation}

\end{proof}

Throughout the paper we bound weak norms by weighted sums of spatial averages as given in the next lemma.
\begin{lemma}
\label{l.bound.dual.seminorm}
Let~$q\in [1,\infty],p\in (1,\infty),s\in (0,1)$, with~$s = 1$ allowed if~$q=\infty$. Then for any~$n\in\Z$,
\begin{equation*}
[f]_{\underline{\hat{B}}_{p,q}^{-s}(\cudot_n)} \leq 3^{d+2+s} \biggl( \sum_{k=-\infty}^n \biggl( 3^{spk} \avsum_{z\in\mathcal{Z}_{k-1}, z+\cudot_k \subseteq \cudot_n} \bigl| (f)_{z+\cudot_k}\bigr|^p \biggr)^{\nicefrac{q}{p}}\biggr)^{\nicefrac{1}{q}}\,.
\end{equation*}
\end{lemma}
\begin{proof}
See~\cite[Lemma A.1]{AK.HC}.
\end{proof}

We will use in Section~\ref{s.hc} a div-curl lemma which requires us to estimate the norms of products in Besov spaces. In order to state the lemma we first define
\begin{equation}
\label{e.W1.par}
\|\psi\|_{\underline{W}^{1,\infty}_{\mathrm{par}}(\cudot_n)} := \|\psi\|_{L^\infty(\cudot_n)} + 3^{n}\|\nabla \psi\|_{L^\infty(\cudot_n)} + 3^{2n}\|\partial_t \psi\|_{L^\infty(\cudot_n)}\,.
\end{equation}
The scaling is chosen such that a cutoff function~$\psi\in C_c^\infty(\cudot_{n-1})$ can satisfy~$\|\psi\|_{\underline{W}^{1,\infty}_{\mathrm{par}}(\cudot_n)} \leq C$.

\begin{lemma}
\label{l.multiply.in.besov.norm}
Let~$s\in [0,1]$, $n\in\mathbb{N}$ and suppose that~$u\in B_{2,\infty}^s(\cudot_{n})$ and~$\varphi\in C^\infty(\cudot_n)$.Then
\begin{equation}
\label{e.other.div.curl}
\| (u-(u)_{\cudot_{n}})\nabla \varphi\|_{\underline{B}_{2,\infty}^s(\cudot_n)} \leq C(d)\|\nabla \varphi\|_{\underline{W}^{1,\infty}_{\mathrm{par}}(\cudot_n)}\bigl( [\nabla u]_{\Bring_{2,1}^{s-1}(\cudot_{n+1})} + [\a\nabla u]_{\Bring_{2,1}^{s-1}(\cudot_{n+1})} \bigr)\,.
\end{equation}
If we have~$\varphi \in C_c^\infty(\cudot_{n-1})$ then
\begin{equation}
\label{e.div.curl.lemma.eq}
\| (u-(u)_{\cudot_{n-1}})\nabla \varphi\|_{\underline{B}_{2,\infty}^s(\cudot_n)} \leq C(d)\|\nabla \varphi\|_{\underline{W}^{1,\infty}_{\mathrm{par}}(\cudot_n)}\bigl( [\nabla u]_{\Bring_{2,1}^{s-1}(\cudot_n)} + [\a\nabla u]_{\Bring_{2,1}^{s-1}(\cudot_n)} \bigr)\,.
\end{equation}
%\begin{align}
%\|v\varphi\|_{\underline{B}_{2,\infty}^s(\cusdot_n)}  & \leq C \|\varphi\|_{L^\infty(\cusdot_n)}[v]_{\underline{B}_{2,\infty}^s(\cusdot_n)} \nonumber \\
%& \quad + C\bigl(3^{(1-s)n}\|\q_0\nabla \varphi\|_{L^\infty(\cusdot_n)} + 3^{(2-s)n}\|\lambda_r^{-1}\partial_t \varphi\|_{L^\infty(\cusdot_n)} + 3^{-sn}\norm{\varphi}_{L^\infty(\cusdot_n)}\bigr)\|v\|_{\L^2(\cusdot_n)} \label{e.multiply.in.besov.norm}\,.
%\end{align}
\end{lemma}
\begin{proof}
We first prove that for any~$\psi \in C^\infty(\cudot_n)$
\begin{align}
\|u\psi\|_{\underline{B}_{2,\infty}^s(\cudot_n)}  & \leq C \|\psi\|_{L^\infty(\cudot_n)}[u]_{\underline{B}_{2,\infty}^s(\cudot_n)} \nonumber \\
& \quad + C\bigl(3^{(1-s)n}\|\nabla \psi\|_{L^\infty(\cudot_n)} + 3^{(2-s)n}\|\partial_t \psi\|_{L^\infty(\cudot_n)} + 3^{-sn}\norm{\psi}_{L^\infty(\cudot_n)}\bigr)\|u\|_{\L^2(\cudot_n)}\,. \label{e.multiply.in.Besov}
\end{align}
We decompose
\begin{equation*}
u\psi - (u\psi)_{z+\cudot_k} = (u-(u)_{z+\cudot_k})\psi + (u)_{z+\cudot_k}(\psi - (\psi)_{z+\cudot_k}) + (u)_{z+\cudot_k}(\psi)_{z+\cudot_k} - (u\psi)_{z+\cudot_k}\,,
\end{equation*}
where the last two terms can be written
\begin{align*}
(u)_{z+\cudot_k}(\psi)_{z+\cudot_k} - (u\psi)_{z+\cudot_k} = -\fint_{z+\cudot_k} u(\psi - (\psi)_{z+\cudot_k})\,.
\end{align*}
Then
\begin{align*}
[u\psi]_{\underline{B}_{2,\infty}^s(\cudot_n)} & = \sup_{k\in (-\infty,n] \cap \Z} 3^{-sk} \biggl(\avsum_{z\in\mathcal{Z}_{k-1}, z+\cudot_k \subseteq \cudot_n}\|u\psi -(u\psi)_{z+\cudot_k}\|_{\L^2(z+\cudot_k)}^2 \biggr)^{\nicefrac12} \\
& \leq C\sup_{k\in (-\infty,n] \cap \Z} 3^{-sk} \biggl( \avsum_{z\in\mathcal{Z}_{k-1}, z+\cudot_k \subseteq \cudot_n} \|\psi\|_{L^\infty(z+\cudot_k)}^2\|u-(u)_{z+\cudot_k}\|_{\L^2(z+\cudot_k)}^2 \\
& \hspace*{5cm} + \|u\|_{\L^2(z+\cudot_k)}^2\|\psi -(\psi)_{z+\cudot_k}\|_{\L^2(z+\cudot_k)}^2  \biggr)^{\nicefrac12} \\
& \leq C \|\psi\|_{L^\infty(\cudot_n)}[u]_{\underline{B}_{2,\infty}^s(\cudot_n)} + C\bigl(3^{(1-s)n}\|\nabla \psi\|_{\L^\infty(\cudot_n)} + 3^{(2-s)n}\|\partial_t \psi\|_{L^\infty(\cudot_n)}\bigr)\|u\|_{\L^2(\cudot_n)}\,.
\end{align*}
We complete the proof of~\eqref{e.multiply.in.Besov} by noting that~$3^{-sn}|(u\psi)_{\cudot_n}| \leq 3^{-sn}\norm{\psi}_{L^\infty(\cudot_n)}\norm{u}_{\L^2(\cudot_n)}$.

To obtain~\eqref{e.div.curl.lemma.eq} we fix~$\varphi \in C_c^\infty(\cudot_{n-1})$ and use~\eqref{e.multiply.in.Besov} with the replacements~$\psi \to \nabla \varphi$, and~$u \to u-(u)_{\cudot_{n-1}}$. If our cutoff function~$\varphi$ is supported in the interior cube~$\cudot_{n-1}$ our norms on the left-hand side are the interior norms; then we use Lemma~\ref{l.Besov.regularity.solns} to bound (assuming that the right-hand side is finite)
\begin{equation*}
[u - (u)_{\cudot_{n-1}}]_{\underline{B}^s_{2,\infty}(\cudot_{n-1})} \leq C(d) \bigl( [\nabla u]_{\Bring_{2,1}^{s-1}(\cudot_n)} + [\a\nabla u]_{\Bring_{2,1}^{s-1}(\cudot_n)}\bigr)\,,
\end{equation*}
and
\begin{align*}
3^{-sn}\|u-(u)_{\cudot_{n-1}}\|_{\L^2(\cudot_{n-1})} & \leq C(d)3^{-sn} \bigl( [\nabla u]_{\Bring_{2,1}^{-1}(\cudot_n)} + [\a\nabla u]_{\Bring_{2,1}^{-1}(\cudot_n)} \bigr) \\
& \leq C(d)\bigl( [\nabla u]_{\Bring_{2,1}^{s-1}(\cudot_n)} + [\a\nabla u]_{\Bring_{2,1}^{s-1}(\cudot_n)}\bigr)\,.
\end{align*}
The proof is completed by using the definition in~\eqref{e.W1.par}. We obtain~\eqref{e.other.div.curl} in the case that our cutoff function is not compactly supported by using Lemma~\ref{l.Besov.regularity.solns} directly with the larger domain~$\cudot_{n+1}$ on the right-hand side.

\end{proof}

The following simple lemma will be used in Proposition~\ref{p.Besov.equiv}.
\begin{lemma}
\label{l.avgs}
Suppose that~$V \subseteq U \subseteq \mathbb{R}^d$, and let~$\varphi \in C^\infty(U)$ be a smooth, non-negative function such that~$\int_U \varphi \geq c|V|$ for some constant~$c >0$. Then
\begin{align*}
\| u - (u)_V \|^p_{\L^p(V)} \leq C(p,d) \frac{1}{|V|} \fint_V \int_U |u(x) - u(y)|^p \varphi(y) dy\,.
\end{align*}
\end{lemma}
\begin{proof}
First, we have
\begin{equation}
\label{e.inf}
\| u - (u)_V \|_{\L^p(V)} \leq \inf_{a\in \mathbb{R}}\bigl( \| u - a\|_{\L^p(V)} + \|(u)_V - a \|_{\L^p(V)} \bigr)\leq 2 \inf_{a\in\mathbb{R}}\|u - a\|_{\L^p(V)}\,.
\end{equation}
If we then make the choice
\begin{equation*}
a_* = \biggl( \int_U \varphi(y) dy \biggr)^{-1} \int_U u(y) \varphi(y) dy 
\end{equation*}
in~\eqref{e.inf} and apply Jensen's inequality to the measure~$\varphi(y)dy$, we obtain
\begin{align*}
\| u - (u)_V \|^p_{\L^p(V)} & \leq 2^p\| u - a_* \|^p_{\L^p(V)} \leq  2^p \fint_V \biggl( \int_U \varphi(y) dy \biggr)^{-p} \biggl| u(x) \biggl( \int_U \varphi(y) dy \biggr) - \int_U u(y)\varphi(y) dy \biggr|^p dx \\
& \leq \frac{C}{|V|}\fint_V \int_U |u(x)-u(y)|^p \varphi(y)dy dx\,.
\end{align*}
\end{proof}

For~$0<s<1$ and~$1\leq p <\infty$ the (volume-normalized) fractional regularity Sobolev space~$W^{s,p}$ on a bounded Lipschitz domain~$U\subseteq \mathbb{R}^d$ has semi-norm
\begin{equation*}
[g]_{\underline{W}^{s,p}(U)} \coloneqq \left(\fint_U \int_U \frac{|g(x)-g(y)|^p}{|x-y|^{d+sp}} dxdy \right)^{\nicefrac{1}{p}}\,.
\end{equation*}
We prove in the following proposition that
\begin{equation*}
B_{p,p}^s(I\times U) = L^p(I;W^{s,p}(U)) \cap W^{\nicefrac{s}{2},p}(I;L^p(U))\,,
\end{equation*}
where the left-hand side is defined in~\eqref{e.Besov.norm.def} for parabolic cubes, and extended to general domains in~\eqref{e.Besov.gen.def}.

\begin{proposition}
\label{p.Besov.equiv}
For any~$1\leq p < \infty$ and~$0<s<1$, there exists a constant~$C = C(s,d,p,I,U)$, such that
\begin{equation}
\label{e.Besov.upper}
\| g \|_{\underline{B}_{p,p}^{s}(I\times U)} \leq  C\bigl( \| g \|_{\underline{L}^p(I; \underline{W}^{s,p}(U))} + \|g\|_{\underline{W}^{\nicefrac{s}{2},p}(I;\L^p(U))} \bigr) \,,
\end{equation}
and
\begin{equation}
\label{e.Besov.lower}
\| g \|_{\underline{L}^p(I; \underline{W}^{s,p}(U))} + \|g\|_{\underline{W}^{\nicefrac{s}{2},p}(I;\L^p(U))} \leq C\| g \|_{\underline{B}_{p,p}^{s}(I\times U)} \,.
\end{equation}
\end{proposition}
\begin{proof}
We will give the proof in the case that~$I = I_n = (-\frac{3^{2n}}{2}, \frac{3^{2n}}{2} )$ and~$U = \cu_n = (-\frac{3^n}{2},\frac{3^n}{2})^d$. For a general domain, using the definition in~\eqref{e.Besov.semi.def}, the fact that the sub-domains~$z+\cudot_n \cap (I\times U)$ in that definition are chosen to have volume comparable to~$|\cudot_n|$ and partition the domain at each scale is sufficient for the following proof to apply with only notational modification.

\emph{Step 1:} We first give the proof of~\eqref{e.Besov.upper}.
Suppose that~$z= (z^0,z') \in \mathcal{Z}_k$. By the triangle inequality,
\begin{align*}
\| g - (g)_{z+\cudot_k}\|_{\L^p(z+\cudot_k)} & \leq \| g(x,t) - (g(t))_{z'+\cu_k} \|_{\L^p(z+\cudot_k)} + \| ( g(t) )_{z'+ \cu_k} - (g)_{z+\cudot_k} \|_{\L^p(z+\cudot_k)} \\
& \leq \biggl(\fint_{z^0 + I_k} \| g(t) - (g(t))_{z' + \cu_k}\|_{\L^p(z'+ \cu_k)}^p dt \biggr)^{\nicefrac{1}{p}} \\
& \qquad + \biggl(\fint_{z'+\cu_k} \| g(x) - (g(x))_{z^0+I_k}\|_{\L^p(z^0+I_k)}^p dx \biggr)^{\nicefrac{1}{p}}\,.
\end{align*}
Now let~$\{\psi_k\}_{k=-\infty}^\infty$ be a partition of unity in the spatial variable, so that~$\sum_{k\in\mathbb{Z}} \psi_k(x) = 1$ for~$x\neq 0$,~$\psi_k$ is supported in~$\cu_{k+2} \setminus \cu_k$,~$\|\nabla \psi_k\|_{L^\infty} \leq C3^{-k}$ and the support of~$\psi_k$ intersects only the supports of~$\psi_{n}$ for~$n \in \{k-1,k+1\}$. Similarly, let~$\varphi_k$ be a partition of unity in the time variable such that~$\sum_{k\in\mathbb{Z}} \varphi_k(t) = 1$ for~$t\neq 0$,~$\varphi_k$ is supported in~$I_{k+2} \setminus I_k$,~$\|\nabla \varphi_k\|_{L^\infty} \leq C3^{-2k}$ and the support of~$\varphi_k$ intersects only the supports of~$\varphi_{n}$ for~$n \in \{k-1,k+1\}$. It follows that for~$k\leq n$ we have~$\int_{\cu_n} \psi_k \geq c|\cu_k|$ and~$\int_{I_n}\varphi_k \geq c|I_k|$, so by the above display and Lemma~\ref{l.avgs},
\begin{align*}
\lefteqn{
\| g - (g)_{\cudot_k}\|^p_{\L^p(z+\cudot_k)}
} \quad & \\
& \leq \frac{C}{|\cu_k|} \fint_{z+\cudot_k} \int_U |g(x,t) - g(y,t)|^p \psi_k(y) dx dy dt  + \frac{C}{|I_k|} \fint_{z+\cudot_k} \int_I |g(x,t) - g(x,s)|^p \varphi_k(t) dx dt ds\,.
\end{align*}
Applying this in each sub-cube, summing, and using the result of~\cite[Lemma A.4]{AK.HC} independently in space and time,
\begin{align*}
\lefteqn{
\sum_{k=-\infty}^n 3^{-spk} \avsum_{z\in\mathcal{Z}_k \cap \cudot_n} \| g- (g)_{z+\cudot_k}\|^p_{\L^p(z+\cudot_k)}
} \qquad & \\
& \leq \sum_{k=-\infty}^n 3^{-spk} \fint_{I_n} \avsum_{z'\in 3^k\mathbb{Z}^d \cap \cudot_n}\frac{C}{|\cu_k|} \fint_{z'+ \cu_k} \int_{\cu_n} |g(x,t) - g(y,t)|^p \psi_k(z' - y) dx dy dt  \\
& \qquad  + \sum_{k=-\infty}^n 3^{-spk} \fint_{\cu_n} \avsum_{z^0\in 3^{2k}\mathbb{Z} \cap \cudot_n} \frac{C}{|I_k|} \fint_{z^0 + I_k} \int_{I_n} |g(x,t) - g(x,s)|^p \varphi_k(z^0 - t) dx dt ds \\
& \leq C\fint_{I_n} \fint_{\cu_n} \int_{\cu_n} \frac{|g(x,t) - g(y,t)|^p}{|x-y|^{d+sp}} dx dy dt + C\fint_{I_n} \fint_{\cu_n} \int_{I_n} \frac{|g(x,t) - g(x,s)|^p}{|t-s|^{1 + \nicefrac{sp}{2}}} dx dt ds\,.
\end{align*}
This concludes the proof of~\eqref{e.Besov.upper}.

\emph{Step 2:} We now give the proof of~\eqref{e.Besov.lower}.
Let~$\{\psi_k\}_{-\infty}^\infty$ be a partition of unity on~$\mathbb{R}^d$ such that~$\sum_{j\in \mathbb{Z}} \psi_j(x) = 1$ for all~$x\neq 0$,~$\mathrm{supp}\, \psi_j \subset \cu_{j+1}\setminus \cu_{j-1}$ and for~$x\in \cu_{j+1}\setminus \cu_{j-1}$ only~$\psi_{j-1}(x),\psi_j(x)$ and~$\psi_{j+1}(x)$ are nonzero. Then summing over the partition,
\begin{align*}
	\| g \|^p_{\underline{L}^p(I_n; W^{s,p}(\cu_n))} & = \fint_{I_n}\fint_{\cu_n} \int_{\cu_n} \frac{|u(x,t)-u(y,t)|^p}{|x-y|^{d+sp}} \sum_{k=-\infty}^{n+1} \psi_k(x-y) dxdydt \\
	& \leq C(d,s,p) \sum_{k=-\infty}^{n+1} \fint_{I_n}\fint_{\cu_n} \int_{\cu_n} \frac{|u(x,t)-u(y,t)|^p}{3^{k(d+sp)}}\psi_k(x-y) dxdydt\,.
\end{align*}
Since~$\{z+\cudot_{k+1}:z = (z_0,z')\in \mathcal{Z}_{k+1} \cap \cudot_n\}$ is a partition of~$\cudot_n$ we may decompose the integral in~$(t,y)$ into a sum over this partition. By our assumption on the support of~$\psi_k$, if~$y\in z' +\cu_{k+1}$ and~$\psi_k(x-y)\neq 0$ then~$x\in z'+\cu_{k+2}$ so we can restrict to integrating in~$x$ over this cube. We finally compare by the triangle inequality to the average over the parabolic cube containing both~$x$ and~$y$ so that we obtain
\begin{align*}
\| g & \|^p_{\underline{L}^p(I_n; W^{s,p}(\cu_n))} \\
& \leq C \sum_{k=-\infty}^{n+1} 3^{-spk}\avsum_{z\in \mathcal{Z}_{k+1} \cap \cudot_n} \fint_{z_0+I_{k+1}}\fint_{z'+\cu_{k+1}}\fint_{z'+\cu_{k+2} \cap \cu_n} |g(x,t)-g(y,t)|^p dxdydt \\
& \leq C\sum_{k=-\infty}^{n+1} 3^{-spk}\avsum_{z\in\mathcal{Z}_{k+1} \cap \cudot_n} \|g - (g)_{z+\cudot_{k+2}\cap \cudot_n}\|_{\L^p(z+\cudot_{k+2}\cap \cudot_n)}^p\,.
\end{align*}
Noting that the top terms~$k=n-1,n,n+1$ can all be combined into the~$k=n-2$ term (which integrates over~$\cudot_n$) and re-indexing by~$k\to k-2$ we get exactly the definition in~\eqref{e.Besov.def}.

The proof of the bound for~$W^{\nicefrac{s}{2},p}(I_n;L^p(\cu_n))$ is identical, by taking the partitions in time instead of space.
\end{proof}

We will use Lemma~\ref{l.dist} to prove Lemma~\ref{l.general.dual.norm}, which bounds negative norms in general domains. We first define the parabolic distance by
\begin{equation}
\label{e.dpar}
\mathrm{d}_{\mathrm{par}}( (x,t), \partial (I\times U) ) = \min\big\{ \dist(x,\partial U), \dist(t,\partial I)^{\nicefrac12} \big\}\,.
\end{equation}

The following lemma extends naturally by density, but to avoid unnecessary definitions we simply state it for compactly supported functions.

\begin{lemma}[Space-time Hardy-Poincar\'e Inequality]
\label{l.dist}
Suppose that~$g \in C_c^\infty(I\times U)$,~$p\in (1,\infty)$ and~$s \in (0,1)\setminus \{\nicefrac{1}{p},\nicefrac{2}{p}\}$. Then there exists a constant~$C = C(I,U,d,s,p)$ such that
\begin{equation*}
\fint_{I\times U} \frac{|g(x,t)|^p}{\mathrm{d}_{\mathrm{par}}((x,t),\partial (I\times U) )^{sp}}\, dxdt \leq C \|g\|_{\underline{B}^s_{p,p}(I\times U)}
\end{equation*}
\end{lemma}
\begin{proof}
The fractional Hardy-Poincar\'e inequality~\cite[Section 3.2.6, Lemma 1b]{Triebel1978} applied in time, integrated in space, implies that
\begin{equation*}
\fint_{I\times U} \frac{|g(x,t)|^p}{\dist(t,\partial I)^{\frac{sp}{2}}}\, dxdt  \leq C \|g \|_{\underline{W}^{\nicefrac{s}{2},p}(I;\L^p(U))}\,.
\end{equation*}
Similarly, the fractional Hardy-Poincar\'e inequality in space, integrated in time, implies that
\begin{equation*}
\fint_{I\times U} \frac{|g(x,t)|^p}{\dist(x,\partial U)^{sp}}\,dx dt  \leq C \|g \|_{\L^p(I; \underline{W}^{s,p}(U))} \,.
\end{equation*}
Combining these and using Proposition~\ref{p.Besov.equiv} yields
\begin{align*}
\fint_{I\times U} \frac{|g(x,t)|^p}{\mathrm{d}_{\mathrm{par}}((x,t),\partial (I\times U) )^{sp}}\, dx dt & \leq \fint_{I\times U} \frac{|g(x,t)|^p}{\dist(t,\partial I)^{\frac{sp}{2}}}\, dx dt + \fint_{I\times U} \frac{|g(x,t)|^p}{\dist(x,\partial U)^{sp}} \, dx dt \\
& \leq C\|g\|_{\underline{B}^s_{p,p}(I\times U)}\,.
\end{align*}
\end{proof}

In the following lemma we assume that~$n\in\mathbb{Z}$ is the smallest integer such that~$I\times U \subseteq \cudot_n$, and we let constants depend on the ratio~$\frac{|\cudot_n|}{|I\times U|}$. Lemma~\ref{l.general.dual.norm} is stated for the~$B^{-s}_{p,p}(I\times U)$ norm, which tests against compactly supported functions. If~$s < \nicefrac12$ then compactly supported functions are dense in~$B^{s}_{p,p}(I\times U)$, so we can replace the left-hand side with the~$\hat{B}^{-s}_{p,p}(I\times U)$ norm.

\begin{lemma}
\label{l.general.dual.norm}
Let~$p\in (1,\infty)$,~$s\in (0,1)\setminus \{1-\nicefrac{1}{p},2-\nicefrac{2}{p}\}$, and suppose that~$I$ is a finite time interval,~$U$ is a bounded Lipschitz domain and~$n\in\mathbb{Z}$ is the smallest integer such that~$I\times U \subseteq \cudot_n$. There exists a constant~$C = C(I,U,p,s,d)$ such that
\begin{equation}
\|f\|_{\underline{B}_{p,p}^{-s}(I\times U)} \leq C\biggl(\sum_{k=-\infty}^n 3^{spk} \avsum_{z\in\mathcal{Z}_k, z+\cu_k \subseteq I\times U} | (f)_{z+\cu_k}|^p \biggr)^{\nicefrac{1}{p}}\,.
\end{equation}
\end{lemma}
\begin{proof}
Define for each~$k\leq n$ the set of boundary layer cubes
\begin{equation*}
V_k = \{z + \cudot_k : z\in\mathcal{Z}_k\,, z+\cudot_{k+1} \subseteq I \times U\,, z+\cudot_{k+2} \cap \partial (I\times U) \neq \emptyset \}\,,
\end{equation*}
Fix~$g\in C_c^\infty(I\times U)$, let~$p'$ denote the H\"older conjugate to~$p$, and decompose the domain into the (overlapping) boundary layers to get
\begin{align}
\label{e.general.splitting}
\lefteqn{
\fint_I \fint_U fg
} \quad & \notag \\
& \leq C\sum_{k=-\infty}^n \frac{|V_k|}{|I\times U|} \avsum_{z\in V_k \cap \mathcal{Z}_k} \fint_{z+\cudot_k} fg  \notag\\
& \leq C\sum_{k=-\infty}^n \frac{|V_k|}{|I\times U|} \avsum_{z\in V_k \cap \mathcal{Z}_k} \bigl( \|f\|_{\underline{\hat{B}}_{p,p}^{-s}(z+\cudot_k)}[g-(g)]_{\underline{B}_{p',p'}^{s}(z+\cudot_k)} + |(f)_{z+\cudot_k}| |(g)_{z+\cudot_k}| \bigr) \notag\\
& \leq C\biggl( \sum_{k=-\infty}^n \frac{|V_k|}{|I\times U|} \avsum_{z\in V_k \cap \mathcal{Z}_k}\|f\|_{\underline{\hat{B}}_{p,p}^{-s}(z+\cudot_k)}^p  \biggr)^{\nicefrac{1}{p}} \biggl( \sum_{k=-\infty}^n \frac{|V_k|}{|I\times U|} \avsum_{z\in V_k \cap \mathcal{Z}_k}[g-(g)]_{\underline{B}_{p',p'}^{s}(z+\cudot_k)}^{p'} \biggr)^{\nicefrac{1}{p'}} \notag\\
& \qquad + \biggl( \sum_{k=-\infty}^n \frac{|V_k|}{|I\times U|} 3^{spk} \avsum_{z\in V_k \cap \mathcal{Z}_k}|(f)_{z+\cudot_k}|^p  \biggr)^{\nicefrac{1}{p}} \biggl( \sum_{k=-\infty}^n \frac{|V_k|}{|I\times U|} 3^{-sp'k} \avsum_{z\in V_k \cap \mathcal{Z}_k} |(g)_{z+\cudot_k}|^{p'} \biggr)^{\nicefrac{1}{p'}}\,.
\end{align}
The factors involving~$f$ we bound by
\begin{align*}
\sum_{k=-\infty}^n \frac{|V_k|}{|I\times U|} \avsum_{z\in V_k \cap \mathcal{Z}_k}\|f\|_{\underline{\hat{B}}_{p,p}^{-s}(z+\cudot_k)}^p & \leq \sum_{k=-\infty}^n \frac{|V_k|}{|I\times U|} \avsum_{z\in V_k \cap \mathcal{Z}_k} \sum_{j=-\infty}^k 3^{spj} \avsum_{z' \in \mathcal{Z}_j \cap (z+\cudot_k)} |(f)_{z'+\cudot_j}|^p \\
& \leq C \sum_{j=-\infty}^n 3^{spj} \frac{|\cudot_j|}{|I\times U|} \sum_{k=j}^n \sum_{z'\in\mathcal{Z}_j, z' + \cudot_j \subseteq V_k} |(f)_{z'+\cudot_j}|^p \\
& = C \sum_{j=-\infty}^n 3^{spj} \avsum_{z'\in\mathcal{Z}_j, z' + \cudot_j \subseteq I \times U} |(f)_{z'+\cudot_j}|^p\,.
\end{align*}
Since~$3^{sk}|(f)_{z+\cudot_k}|$ is just the top term in~$\|f\|_{\underline{B}_{p,p}^{-s}(z+\cudot_k)}$ we can likewise bound the second factor in~\eqref{e.general.splitting} by the same quantity.

For the terms involving~$g$, we use the integral representation~\eqref{e.fractional.integral} to find
\begin{align*}
\lefteqn{
\biggl( \sum_{k=-\infty}^n \frac{|V_k|}{|I\times U|} \avsum_{z\in V_k \cap \mathcal{Z}_k}[g-(g)]_{\underline{B}_{p',p'}^{s}(z+\cudot_k)}^{p'} \biggr)^{\nicefrac{1}{p'}}
} \qquad \qquad & \\
& \leq  \biggl( \sum_{k=-\infty}^n \frac{|V_k|}{|I\times U|} \avsum_{z\in V_k \cap \mathcal{Z}_k} \fint_{z+\cudot_k} \int_{z+\cudot_k} \frac{|g(x,t)-g(y,s)|^{p'}}{(|x-y| + |t-s|^{\nicefrac12} )^{d+sp'}} dxdy dt ds\biggr)^{\nicefrac{1}{p'}} \\
& \leq \biggl( \fint_{I\times U} \int_{I\times U} \frac{|g(x)-g(y)|^{p'}}{(|x-y| + |t-s|^{\nicefrac12} )^{d+sp'}} dxdy dt ds \biggr)^{\nicefrac{1}{p'}}\,.
\end{align*}
For the second term in~\eqref{e.general.splitting} involving $g$, we use Lemma~\ref{l.dist} and the assumption~$s\neq \nicefrac{1}{p'},\nicefrac{2}{p'}$ to bound
\begin{align*}
\sum_{k=-\infty}^n \frac{|V_k|}{|I\times U|} 3^{-sp'k} \avsum_{z\in V_k \cap \mathcal{Z}_k} |(g)_{z+\cudot_k}|^{p'}
& \leq \sum_{k=-\infty}^n \frac{|V_k|}{|I\times U|} \avsum_{z\in V_k \cap \mathcal{Z}_k} \fint_{z+\cudot_k} \frac{|g(x,t)|^{p'}}{\mathrm{d}_{\mathrm{par}}( (x,t), \partial(I\times U) )^{sp'}} \\
& \leq C \fint_{I\times U} \int_{I\times U} \frac{|g(x,t)|^{p'}}{\mathrm{d}_{\mathrm{par}}((x,t),\partial(I\times U) )^{sp'}} \\
& \leq C \|g\|_{\underline{B}^s_{p',p'}(I\times U)}^{p'} \,.
\end{align*}
Putting our estimates together, we have shown that for~$g\in C_c^\infty(I\times U)$,
\begin{equation*}
\fint_{I\times U} fg \leq C \|g\|_{\underline{B}^s_{p',p'}(I\times U)} \biggl(\sum_{j=-\infty}^n 3^{spj} \avsum_{z'\in\mathcal{Z}_j, z' + \cudot_j \subseteq I \times U} |(f)_{z'+\cudot_j}|^p \biggr)^{\nicefrac{1}{p}}\,,
\end{equation*}
which concludes the proof.
\end{proof}

The following lemma obtains estimates for the heat equation by interpolating between the standard energy estimate and the~$H^{2,1}$ regularity estimate. This is a standard estimate, but a direct reference could not be found.

\begin{lemma}
\label{l.heat.estimate}
Let~$U\subseteq \mathbb{R}^n$ be a bounded domain which is either~$C^{1,1}$ or convex and Lipschitz,~$I = (0,T)$ a finite time interval,~$s\in (0,1)$,~$\mathbf{f}\in L^2(I;H^s(U))^d$, and~$u \in H^1_{\mathrm{par}}(I\times U)$ the unique function satisfying~$\partial_t u = \Delta u + \nabla \cdot \mathbf{f}$ with~$u=0$ on~$\partial_{\sqcup} (I\times U)$. Then there exists a constant~$C = C(I,U,s,d)$ such that
\begin{equation}
\label{e.heat.eq}
\|\nabla u\|_{\underline{B}^s_{2,2}(I\times U)} \leq C \| \mathbf{f}\|_{\L^2(I;\underline{H}^s(U) )}\,.
\end{equation}
\end{lemma}
\begin{proof}
We use in this proof the spaces,
\begin{equation*}
H^{r,q}(I\times U) = L^2(I; H^r(U)) \cap H^q(I;L^2(U))\,,
\end{equation*}
defined in~\cite[Chapter 4, Section 2]{LMV2} for~$r,q\geq 0$. By the intermediate derivatives theorem~\cite[Theorem 2.3, Theorem 4.1]{LMV1}, a function~$v \in H^{r,q}(I\times U)$ belongs to~$H^{\alpha}(I; H^{(1-\nicefrac{\alpha}{q})r}(U))$ with the bound
\begin{equation}
\label{e.mixed.diff}
\|u\|_{H^{\alpha}(I; H^{(1-\nicefrac{\alpha}{q})r}(U))} \leq C(\|u\|_{L^2(I; H^r(U))} + \|u\|_{H^q(I;L^2(U))} )\,.
\end{equation}
Proposition~\ref{p.Besov.equiv} states that~$B_{2,2}^s(I\times U) = H^{s,\nicefrac{s}{2}}(I\times U)$. By the interpolation theorem~\cite[Theorem 7.23]{Adams} with parameter $s$,
\begin{align}
\label{e.to.proof}
\left.
\begin{aligned}
& \|\nabla u\|_{H^{1,\nicefrac12}(I\times U)} \leq C \|\mathbf{f}\|_{L^2(I;H^1(U))} & \\
& \|\nabla u\|_{L^2(I\times U)} \leq C \|\mathbf{f}\|_{\L^2(I\times U)} &
\end{aligned}
\right\}
\implies \|\nabla u\|_{H^{s,\nicefrac{s}{2}}(I\times U)} \leq C \|\mathbf{f}\|_{\L^2(I;H^{s}(U))}\,.
\end{align}
Of the two estimates on the left-hand side, the lower one comes directly from testing the equation with itself, while the upper one is a consequence of the~$H^{2,1}$ estimate
\begin{equation*}
\|u\|_{H^1(I; L^2(U))} + \|u\|_{L^2(I;H^2(U))} \leq C \|\nabla \cdot \mathbf{f}\|_{L^2(I\times U)}\,,
\end{equation*}
because
\begin{align*}
\|\nabla u\|_{H^{1,\nicefrac12}(I\times U)}  = \|\nabla u\|_{\L^2(I;H^1(U))} + \|\nabla u \|_{H^{\nicefrac12}(I;L^2(U))} & \leq \|u\|_{\L^2(I;H^2(U))} + \|u \|_{H^{\nicefrac12}(I;H^1(U))} \\
& \leq C \|u\|_{H^{2,1}(I\times U)}\,,
\end{align*}
using in the last line~\eqref{e.mixed.diff}. Finally, we make use of the zero boundary data to prove the~$H^{2,1}$ estimate. First, by standard arguments (as in~\cite[Chapter 7.1]{Evans})~$\|\partial_t u\|_{\L^2(I\times U)} \leq C\|\nabla \cdot \mathbf{f}\|_{\L^2(I\times U)}$. Then~$\Delta u = \partial_t u - \nabla \cdot \mathbf{f}$ in time slices, so by~\cite[Theorems 2.4.2.5 and 3.1.2.1]{Grisvard} we can apply the elliptic~$H^2$ estimate provided that the domain is either convex and Lipschitz or~$C^{1,1}$, which concludes the proof.
\end{proof}

\section{Matrix Partial Ordering and Geometric Means}
\label{aa.matrices}

This appendix collects some elementary facts which are used repeatedly in the technical work of the paper. We denote the set of real-valued~$d\times d$ matrices by~$\R^{d\times d}$ and the subset of symmetric matrices by~$\R^{d\times d}_{\mathrm{sym}}$, with the Loewner partial ordering. That is, if~$A,B\in\R^{d\times d}_{\mathrm{sym}}$ then we write~$A\leq B$ if~$B-A$ has nonnegative eigenvalues. We use the spectral norm on matrices, defined for any~$A\in\R^{d\times d}$ by~$|A| = \sup_{|e|=1}|Ae|$.

\smallskip

It is true that~$A \leq B$ if and only if~$e \cdot (B-A)e \geq 0$ for all~$e\in\Rd$. If~$0\leq A\leq B$ it is true that~$A^r \leq B^r$ for~$0\leq r\leq 1$, but not generally for~$r>1$. For instance a counter-example with~$r=2$ is
\begin{equation*}
\begin{pmatrix}
1 & 0 \\ 0 & 0
\end{pmatrix}
\,,
\begin{pmatrix}
2 & 1 \\ 1 & 1
\end{pmatrix}
\end{equation*}
from \cite{Zhan}. If~$A\leq B$ and~$C\in\R^{d\times d}_{\mathrm{sym}}$ then~$CAC \leq CBC$. In particular, if~$0<A\leq B$ then~$B^{-\nicefrac12}AB^{-\nicefrac12} \leq \Id$. If~$A \leq B$ it is not necessarily true that~$ACA \leq BCB$, or (assuming further that~$A,B\geq 0$) that~$A^{\nicefrac12}CA^{\nicefrac12} \leq B^{\nicefrac12}CB^{\nicefrac12}$ or even~$(A^{\nicefrac12}CA^{\nicefrac12})^{\nicefrac12} \leq (B^{\nicefrac12}CB^{\nicefrac12})^{\nicefrac12}$. Even taking any two of the matrices to be diagonal still does not suffice. As counter-examples one can take
\begin{equation*}
A^{\nicefrac12} = \begin{pmatrix}
1&0\\0&0
\end{pmatrix}
\,,
B^{\nicefrac12} = \begin{pmatrix}
1 & 0 \\ 0 & 1
\end{pmatrix}
\,,
C = \begin{pmatrix}
1&1\\1&1
\end{pmatrix}
\end{equation*}
which has both~$A$ and~$B$ diagonal. If we want~$A$ and~$C$ diagonal then take
\begin{equation*}
A^{\nicefrac12} = \begin{pmatrix}
1 & 0 \\ 0 & 1
\end{pmatrix}
\,,
B^{\nicefrac12} = \begin{pmatrix}
3 & -1 \\ -1 & 3
\end{pmatrix}
\,,
C = \begin{pmatrix}
1 & 0 \\ 0 & 0
\end{pmatrix}\,,
\end{equation*}
while if we want~$B$ and~$C$ diagonal then for small~$\epsilon >0$,
\begin{equation*}
A^{\nicefrac12} = \begin{pmatrix}
1 & \epsilon \\
\epsilon & 1
\end{pmatrix}
\,,
B^{\nicefrac12} = \begin{pmatrix}
3 & 0 \\
0 & 3
\end{pmatrix}
\,,
C = \begin{pmatrix}
1 & 0 \\
0 & 0
\end{pmatrix}\,.
\end{equation*}

\smallskip

The spectral norm~$|A|$ is the largest eigenvalue of~$(A^tA)^{\nicefrac12}$. It follows that~$|A| = |A^t|$ so if~$A,B\in \R^{d\times d}_{\mathrm{sym}}$ then~$|AB| = |BA|$, although this is not generally true for non-symmetric matrices.
%\begin{equation*}
%A = \begin{pmatrix}
%1 & 0 \\ 0 & 0
%\end{pmatrix}
%\,,
%B = \begin{pmatrix}
%0 & 1 \\ 0 & 0
%\end{pmatrix}
%\end{equation*}
The spectral norm satisfies~$|AB| \leq |A||B|$. If~$A,B$ are positive-definite, symmetric matrices then
\begin{equation}
\label{e.symmetrization}
|A^{-\nicefrac12}BA^{\nicefrac12}| \geq |B|\,,
\end{equation}
although the reverse inequality is not in general true. Applying this to~$A^{-\nicefrac12}BA^{-\nicefrac12}$ yields~$|A^{-1}B| \geq |A^{-\nicefrac12}BA^{-\nicefrac12}|$.

The spectral norm and the Loewner partial ordering are related. For instance~$0\leq A \leq B \implies |A| \leq |B|$. If~$0\leq A \leq B$ and further~$A^2 \leq B^2$ then~$|AC| \leq |BC|$ for any~$C\in\R^{d\times d}_{\mathrm{sym}}$. If we remove the condition~$A^2 \leq B^2$ the statement may not hold; for instance we may take
\begin{equation*}
A = \begin{pmatrix}
1 & 0 \\ 0 & 0
\end{pmatrix}
\,,
B = \begin{pmatrix}
2 & 1 \\ 1 & 1
\end{pmatrix}
\,,
C = \begin{pmatrix}
\frac{3}{4} & -1 \\ - 1 & \frac{3}{2}
\end{pmatrix}
\end{equation*}
%In case it is needed square root is:
%\begin{equation*}
%B^{\nicefrac12} = \begin{pmatrix}
%\frac{3}{\sqrt{5}} & \frac{1}{\sqrt{5}} \\ \frac{1}{\sqrt{5}} & \frac{2}{\sqrt{5}}
%\end{pmatrix}
%\end{equation*}
%Here~$|AC| = 1 > |BC| \approx 0.89$.
However, if~$0\leq A\leq B$ then~$|A^{\nicefrac12}CA^{\nicefrac12}| \leq |B^{\nicefrac12}CB^{\nicefrac12}|$ for any~$C\in\R^{d\times d}_{\mathrm{sym}}$, because, denoting by~$\lambda_{\mathrm{max}}(A)$ the largest eigenvalue of a symmetric matrix,
\begin{align*}
|A^{\nicefrac12}CA^{\nicefrac12}|^2 & = \lambda_{\mathrm{max}}(A^{\nicefrac12}CACA^{\nicefrac12}) \leq \lambda_{\mathrm{max}}(A^{\nicefrac12}CBCA^{\nicefrac12}) = |B^{\nicefrac12}CA^{\nicefrac12}|^2 = |(B^{\nicefrac12}CA^{\nicefrac12})^t|^2 \\
& = |A^{\nicefrac12}CB^{\nicefrac12}|^2 = \lambda_{\mathrm{max}}(B^{\nicefrac12}CACB^{\nicefrac12}) \leq \lambda_{\mathrm{max}}(B^{\nicefrac12}CBCB^{\nicefrac12}) = |B^{\nicefrac12}CB^{\nicefrac12}|^2\,.
\end{align*}

\smallskip

There are two different notions of geometric mean for positive definite matrices. The first one, introduced by Ando~\cite{Ando}, is called the \emph{metric geometric mean}. It is defined for any pair of positive definite matrices~$A$ and~$B$ by 
\begin{equation*}
A\#B = A^{\nicefrac12} \bigl(A^{-\nicefrac12} B  A^{-\nicefrac12}  \bigr)^{\nicefrac12} \! A^{\nicefrac12} \,.
\end{equation*}
The matrix~$A\#B$ is the unique positive definite matrix solution~$X$ of the equation
\begin{equation}
\label{e.riccati}
X A^{-1} X = B. 
\end{equation}
We see from this characterization that the metric geometric mean is symmetric in~$A$ and~$B$, that is,~$A\#B = B\#A$. The harmonic mean is defined for positive matrices~$A,B$ by
\begin{equation*}
A:B : = \bigg(\frac{A^{-1} + B^{-1}}{2}\bigg)^{-1}\,.
\end{equation*}
The geometric mean is bounded above by the arithmetic mean and below by the harmonic mean. That is,
\begin{equation*}
A:B \leq A\#B \leq \frac{A+B}{2}\,.
\end{equation*}
In particular, if~$0\leq A \leq B$ then
\begin{equation*}
A \leq A\# B \leq B\,,
\end{equation*}
and if we assume no ordering of~$A$ and~$B$ we still have
\begin{equation*}
(|A^{-1}|^{-1} \wedge |B^{-1}|^{-1})\Id \leq A\# B \leq (|A| \vee |B|)\Id\,.
\end{equation*}
These results can be found in~\cite{Ando} and~\cite{FP}.

\subsubsection*{\bf Acknowledgments}
The author was partially supported by NSF grant DMS-2350340. The author acknowledges the support of the Natural Sciences and Engineering Research Council of Canada (NSERC) through a PGS-D award (598693-2025). L'auteur remercie le Conseil de recherches en sciences naturelles et en g\'enie du Canada (CRSNG) de son soutien (PGS-D 598693-2025).

{\small
\bibliographystyle{alpha}
\bibliography{parahighcontrast}
}

\end{document}